\documentclass[a4paper,reqno]{amsart}

\usepackage{graphicx}
\usepackage{tikz}
\usetikzlibrary{patterns,shapes,decorations.pathmorphing,decorations.pathreplacing,calc,arrows}
\usepackage{verbatim}
\usepackage{amssymb}
\usepackage{amstext}
\usepackage{amsmath}
\usepackage{amscd}
\usepackage{amsthm}
\usepackage{amsfonts}
\usepackage{enumerate}
\usepackage{latexsym}
\usepackage{mathrsfs}
\usepackage{hyperref}
\usepackage{mathtools}
\usepackage[all]{xy}
\usepackage{pgfplots}
\usepackage{caption}
\usepackage{multirow}
\xyoption{all}

\usepackage{pstricks}
\usepackage{lscape}
\usepackage{comment}

\newtheorem{theorem}{Theorem}[section]

\newtheorem{lemma}[theorem]{Lemma}
\newtheorem{proposition}[theorem]{Proposition}
\newtheorem{definition-theorem}[theorem]{Definition-Theorem}
\newtheorem{definition-proposition}[theorem]{Definition-Proposition}
\newtheorem{problem}[theorem]{Problem}

\theoremstyle{definition}
\newtheorem{definition}[theorem]{Definition}
\newtheorem{assumption}[theorem]{Assumption}
\newtheorem{remark}[theorem]{Remark}
\newtheorem{example}[theorem]{Example}

\newtheorem{construction}[theorem]{Construction}

\renewcommand{\P}{{\rm P}}

\newcommand{\Z}{\mathbb{Z}}
\newcommand{\R}{\mathbb{R}}

\newcommand{\Hom}{\operatorname{Hom}\nolimits}
\newcommand{\rad}{\operatorname{rad}\nolimits}
\newcommand{\End}{\operatorname{End}\nolimits}

\newcommand{\op}{\operatorname{op}\nolimits}

\newcommand{\RHom}{\mathbf{R}\strut\kern-.2em\operatorname{Hom}\nolimits}

\newcommand{\GL}{\operatorname{GL}\nolimits}

\newcommand{\Ann}{\operatorname{Ann}\nolimits}

\DeclareMathOperator{\proj}{\mathsf{proj}}

\DeclareMathOperator{\thick}{\mathsf{thick}}

\DeclareMathOperator{\per}{\mathsf{per}}

\DeclareMathOperator{\add}{\mathsf{add}}

\newcommand{\cut}{\ar@{-}@[|(5)]}

\newcommand{\silt}{\mathsf{silt}}
\newcommand{\psilt}{\mathsf{psilt}}

\newcommand{\twosilt}{\mathsf{2\mbox{-}silt}}

\newcommand{\tscfan}{\mathsf{c}\mbox{-}\mathsf{Fan}_{\rm sc}(2)}
\newcommand{\tscfanpm}{\mathsf{c}\mbox{-}\mathsf{Fan}_{\rm sc}^{+-}(2)}
\newcommand{\tscfanmp}{\mathsf{c}\mbox{-}\mathsf{Fan}_{\rm sc}^{-+}(2)}
\newcommand{\itscfan}{\mathsf{Fan}_{\rm sc}(2)}
\newcommand{\itscfanpm}{\mathsf{Fan}_{\rm sc}^{+-}(2)}
\newcommand{\itscfanmp}{\mathsf{Fan}_{\rm sc}^{-+}(2)}
\newcommand{\gfan}{\mathsf{c}\mbox{-}\mathsf{Fan}_k}
\newcommand{\gelfan}{\mathsf{c}\mbox{-}\mathsf{Fan}_{k,{\rm el}}}
\newcommand{\igfan}{\mathsf{Fan}_k}
\newcommand{\twopsilt}{\mathsf{2\mbox{-}psilt}}

\newcommand{\Kb}{\mathsf{K}^{\rm b}}

\newcommand{\df}{\mathrm{s}}

\newcommand{\conv}{\operatorname{conv}\nolimits}
\newcommand{\cone}{\operatorname{cone}\nolimits}

\numberwithin{equation}{section}

\newcommand{\svdots}{\raisebox{3pt}{\scalebox{.75}{\vdots}}}
\newcommand{\Hasse}{\operatorname{Hasse}\nolimits}
\begin{document}

\title[Fans and polytopes in tilting theory II: $g$-fans of rank 2]{Fans and polytopes in tilting theory II: $g$-fans of rank 2}
%\date{\today}

\author{Toshitaka Aoki}
\address{Graduate School of Information Science and Technology,
Osaka University, 1-5 Yamadaoka, Suita, Osaka 565-0871, Japan}
\email{aoki-t@ist.osaka-u.ac.jp}
%\address{Graduate School of Human Development and Environment, Kobe University, 3-11 Tsurukabuto, Nada-ku, Kobe 657-8501, Japan}
%\email{toshitaka.aoki@people.kobe-u.ac.jp}
\thanks{}

\author{Akihiro Higashitani}
\address{Department of Pure and Applied Mathematics, Graduate School of Information Science and Technology, Osaka University, 1-5 Yamadaoka, Suita, Osaka 565-0871, Japan}
\email{higashitani@ist.osaka-u.ac.jp}
\thanks{}

\author{Osamu Iyama}
\address{Graduate School of Mathematical Sciences, University of Tokyo, 3-8-1 Komaba Meguro-ku Tokyo 153-8914, Japan}
\email{iyama@ms.u-tokyo.ac.jp}
\thanks{}

\author{Ryoichi Kase}
\address{Department of Information Science and Engineering, Okayama University of Science, 1-1 Ridaicho, Kita-ku, Okayama 700-0005, Japan}
\email{r-kase@ous.ac.jp}
\thanks{}

\author{Yuya Mizuno}
\address{Faculty of Liberal Arts, Sciences and Global Education / Graduate School of Science, Osaka Metropolitan University, 1-1 Gakuen-cho, Naka-ku, Sakai, Osaka 599-8531, Japan}
\email{yuya.mizuno@omu.ac.jp}

\begin{abstract}
The $g$-fan of a finite dimensional algebra is a fan in its real Grothendieck group defined by tilting theory. We give a classification of complete $g$-fans of rank 2. More explicitly, our first main result asserts that every complete sign-coherent fan of rank 2 is a $g$-fan of some finite dimensional algebra. Our proof is based on three fundamental results, Gluing Theorem, Rotation Theorem and Subdivision Theorem, which realize basic operations on fans in the level of finite dimensional algebras.
For each of 16 convex sign-coherent fans $\Sigma$ of rank 2, our second main result gives a characterization of algebras $A$ of rank 2 satisfying $\Sigma(A)=\Sigma$. 
As a by-product of our method, we prove that for each positive integer $N$, there exists a finite dimensional algebra $A$ of rank 2 such that the Hasse quiver of the poset of 2-term silting complexes of $A$ has precisely $N$ connected components.
\end{abstract}

\maketitle
\setcounter{tocdepth}{1}
\tableofcontents

\section{Introduction}

The notion of tilting complexes is central to control equivalences of derived categories.
The class of silting complexes \cite{KV} gives a completion of the class of tilting complexes with respect to mutation, which is an operation to replace a direct summand of a given silting complex to construct a new silting complex \cite{AiI}.
The subclass of 2-term silting complexes enjoys remarkable properties \cite{AIR,DF}. They give rise to a fan $\Sigma(A)$ called the \emph{$g$-fan} in the real Grothendieck group of a finite dimensional algebra $A$, see e.g. \cite{H1,H2,DIJ,BST,AY,AsI,M,PY,P,B,As}. In our previous article \cite{AHIKM}, we developed a basic theory of $g$-fans and the associated $g$-polytopes (see Section \ref{section 3.1}). In particular, the following class of fans was introduced.

\begin{definition}\label{define sign-coherent}
A \emph{sign-coherent fan} is a pair $(\Sigma,\sigma_+)$ satisfying the following conditions.
\begin{enumerate}[\rm(a)]
\item $\Sigma$ is a nonsingular fan in $\mathbb{R}^d$, and $\sigma_+,-\sigma_+\in\Sigma_d$.
\item Take $e_1,\dots,e_d\in\R^d$ such that $\sigma_+=\cone\{e_i\mid 1\le i\le d\}$. 
Then for each $\sigma\in\Sigma$, there exists $\epsilon_1,\ldots,\epsilon_d\in\{1,-1\}$ such that $\sigma\subseteq\cone\{\epsilon_1e_1,\dots,\epsilon_de_d\}$.
\item Each cone of dimension $d-1$ is contained in precisely two cones of dimension $d$.
\item Each maximal cone has dimension $d$. 
\end{enumerate}
\end{definition}

The $g$-fan $\Sigma(A)$ of a finite dimensional algebra $A$ over a field $k$ is sign-coherent \cite{AHIKM}, and $\Sigma(A)$ is complete if and only if $A$ is $g$-finite \cite{As,DIJ} (Proposition \ref{characterize g-finite}). The following problem is central in the study of $g$-fans.

\begin{problem}\label{characterize g-fan}
Characterize complete sign-coherent fans in $\R^d$ which can be realized as $g$-fans of some finite dimensional algebras.
\end{problem}

The main aim of this paper is to give a complete answer to this problem for the case $d=2$. We denote by $\mathsf{Fan}_{\rm sc}(2)$ (respectively, $\tscfan$) the set of sign-coherent fans  (respectively, the set of complete sign-coherent fans) in $\mathbb{R}^2$ (Definition \ref{define fan_sc}), and by 
$\igfan(2)$ (respectively,
$\gfan(2)$) the set of $g$-fans (respectively, the set of  complete $g$-fans) of some finite dimensional $k$-algebras $A$ of rank $2$ (Definition \ref{g-fan}).
Identifying $K_0(\proj A)_\R$ with $\R^2$ by $P_1\mapsto(1,0)$ and $P_2\mapsto(0,1)$ where $P_1$ and $P_2$ are indecomposable projective $A$-modules, we obtain inclusions
\[\begin{array}{ccc}
\itscfan&\supset&\igfan(2)\\
\cup&&\cup\\
\tscfan&\supset&\gfan(2).
\end{array}\]
The following answer to Problem \ref{characterize g-fan} for the case $d=2$ was very simple and a surprise to us. We state it as our first main result.

\begin{theorem}[Theorem \ref{main theorem}]\label{theorem:sc fan = g fan intro}Let $k$ be an arbitrary field. Then we have
\[\tscfan=\gfan(2).\]
Thus any complete sign-coherent fan in $\mathbb{R}^2$ can be realized as a $g$-fan of some finite dimensional $k$-algebra. 
\end{theorem}

We explain our method to prove Theorem \ref{theorem:sc fan = g fan intro}. 
Each sign-coherent fan of rank 2 is obtained by gluing two fans of the following form, where each triangle represents a cone in the fan, and a triangle containing $+$ (respectively, $-$) represents the cone $\sigma_+$ (respectively, $-\sigma_+$).
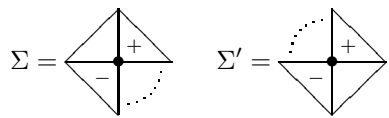
\[{\Sigma=\begin{xy}
0;<4pt,0pt>:<0pt,4pt>::
(0,-5)="0",
(-5,0)="1",
(0,0)*{\bullet},
(0,0)="2",
(5,0)="3",
(0,5)="4",
(0,-4)="a",
(4,0)="b",
(1.5,1.5)*{{\scriptstyle +}},
(-1.5,-1.5)*{{\scriptstyle -}},
\ar@{-}"0";"1",
\ar@{-}"1";"4",
\ar@{-}"4";"3",
\ar@{-}"2";"0",
\ar@{-}"2";"1",
\ar@{-}"2";"3",
\ar@{-}"2";"4",
\ar@/^-2mm/@{.} "a";"b",
\end{xy}}\ \ \ \ \  
{\Sigma'=\begin{xy}
0;<4pt,0pt>:<0pt,4pt>::
(0,-5)="0",
(-5,0)="1",
(0,0)*{\bullet},
(0,0)="2",
(5,0)="3",
(0,5)="4",
(0,4)="c",
(-4,0)="d",
(1.5,1.5)*{{\scriptstyle +}},
(-1.5,-1.5)*{{\scriptstyle -}},
\ar@{-}"0";"1",
\ar@{-}"0";"3",
\ar@{-}"4";"3",
\ar@{-}"2";"0",
\ar@{-}"2";"1",
\ar@{-}"2";"3",
\ar@{-}"2";"4",
\ar@/^-2mm/@{.} "c";"d",
\end{xy}}\]
Recall that a finite dimensional $k$-algebra $A$ is \emph{elementary} if the $k$-algebra $A/J_A$ is isomorphic to a product of $k$. This is automatic if $A$ is basic and $k$ is algebraically closed. 
We prove
\begin{enumerate}[$\bullet$]
\item \emph{Gluing Theorem \ref{theorem:gluing1},} which asserts that if both $\Sigma$ and $\Sigma'$ are $g$-fans of finite dimensional elementary $k$-algebras, then so is their gluing.
\end{enumerate}
Therefore by symmetry, it suffices to consider sign-coherent fans $\Sigma$ of the form above. Now such $\Sigma$ can be obtained from the fan
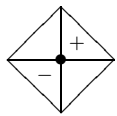
\[\begin{xy}
		0;<4pt,0pt>:<0pt,4pt>::
		(0,-5)="0",
		(-5,0)="1",
		(0,0)*{\bullet},
		(0,0)="2",
		(5,0)="3",
		(0,5)="4",
		(1.5,1.5)*{{\scriptstyle +}},
		(-1.5,-1.5)*{{\scriptstyle -}},
		\ar@{-}"0";"1",
		\ar@{-}"1";"4",
		\ar@{-}"4";"3",
		\ar@{-}"2";"0",
		\ar@{-}"2";"1",
		\ar@{-}"2";"3",
		\ar@{-}"2";"4",
		\ar@{-}"3";"0",
	\end{xy}
\]
by applying subdivision in the fourth quadrant repeatedly.
We prove
\begin{enumerate}[$\bullet$]
\item \emph{Rotation Theorem \ref{theorem:rotation}} and
\item \emph{Subdivision Theorem \ref{theorem:subdivision}},
\end{enumerate}
which imply that if $\Sigma$ is a $g$-fan of a finite dimensional $k$-algebra, then so are the subdivisions of $\Sigma$ in the fourth quadrant. 
We denote by $\tscfanpm$ the subsets of $\tscfan$ respectively which consist of fans containing $\cone\{(1,0),(0,-1)\}$.  
Figure \ref{fig:tree diagram} gives fans in $\tscfanpm$ with at most 8 facet, where each edge shows a subdivision. Figure \ref{fig:tree diagram 2} gives examples of algebras whose $g$-fans are given in Figure \ref{fig:tree diagram} (see Subsection \ref{Sign-coherent fans of rank $2$} for the parametrization of fans).

\begin{figure}

    \caption{Fans in $\tscfanpm$ with at most 8 facets}
    \label{fig:tree diagram}
\end{figure}

\begin{figure}
%
    \caption{Algebras whose $g$-fans are given in Figure \ref{fig:tree diagram}}
    \label{fig:tree diagram 2}
\end{figure}

\medskip
Since the $g$-fan $\Sigma(A)$ gives a combinatorial invariant of the tilting theory of $A$, it is important to study the following problem.

\begin{problem}\label{characterize A}
Let $\Sigma$ be a sign-coherent fan. Give a characterization of algebras $A$ satisfying $\Sigma(A)=\Sigma$.
\end{problem}

For a nonsingular fan $\Sigma$, we denote by $\P(\Sigma)$ the gluing of each simplex associated with the cones in $\Sigma$ (see Definition \ref{from fan to polytope}). If $\P(\Sigma)$ is convex, we call $\Sigma$ \emph{convex}. In Theorem \ref{theorem:charactarization_g-convex_ranktwo intro} below, we give a complete answer to Problem \ref{characterize A} when $\Sigma$ is rank $2$ and convex.
A finite dimensional algebra $A$ is called \emph{$g$-convex} if $\Sigma(A)$ is convex. For example, Brauer tree algebras $A$ are $g$-convex, and this fact plays an important role in the classification of 2-term tilting complexes of $A$ \cite{AMN}. From tilting theoretic point of view, $g$-convex algebras can be regarded as one of the most fundamental classes. Therefore it is important to study the following special case of Problem \ref{characterize g-fan}.

\begin{problem}\label{classify convex}
Classify convex $g$-fans in $\R^d$.
\end{problem}

An answer to the case $d=2$ was given in \cite[Theorem 6.3]{AHIKM}: There are precisely 7 convex $g$-fans up to isomorphism of $g$-fans.
	
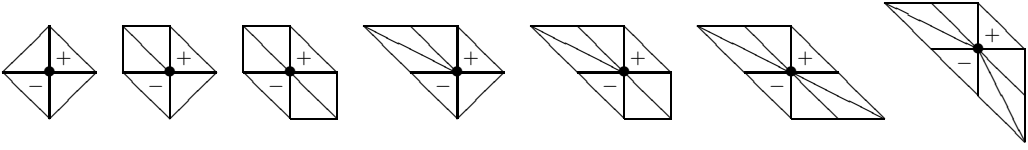
\[{\begin{xy}
			0;<3.5pt,0pt>:<0pt,3.5pt>::
			(0,-5)="0",
			(-5,0)="1",
			(0,0)*{\bullet},
			(0,0)="2",
			(5,0)="3",
			(0,5)="4",
			(1.5,1.5)*{{\scriptstyle +}},
			(-1.5,-1.5)*{{\scriptstyle -}},
			\ar@{-}"0";"1",
			\ar@{-}"1";"4",
			\ar@{-}"4";"3",
			\ar@{-}"3";"0",
			\ar@{-}"2";"0",
			\ar@{-}"2";"1",
			\ar@{-}"2";"3",
			\ar@{-}"2";"4",
	\end{xy}}\ \ \ 
	{\begin{xy}
			0;<3.5pt,0pt>:<0pt,3.5pt>::
			(0,-5)="0",
			(-5,0)="1",
			(0,0)*{\bullet},
			(0,0)="2",
			(5,0)="3",
			(-5,5)="4",
			(0,5)="5",
			(1.5,1.5)*{{\scriptstyle +}},
			(-1.5,-1.5)*{{\scriptstyle -}},
			\ar@{-}"0";"1",
			\ar@{-}"1";"4",
			\ar@{-}"4";"5",
			\ar@{-}"5";"3",
			\ar@{-}"3";"0",
			\ar@{-}"2";"0",
			\ar@{-}"2";"1",
			\ar@{-}"2";"3",
			\ar@{-}"2";"4",
			\ar@{-}"2";"5",
	\end{xy}}\ \ \ 
	{\begin{xy}
			0;<3.5pt,0pt>:<0pt,3.5pt>::
			(0,-5)="0",
			(5,-5)="1",
			(-5,0)="2",
			(0,0)*{\bullet},
			(0,0)="3",
			(5,0)="4",
			(-5,5)="5",
			(0,5)="6",
			(1.5,1.5)*{{\scriptstyle +}},
			(-1.5,-1.5)*{{\scriptstyle -}},
			\ar@{-}"0";"2",
			\ar@{-}"2";"5",
			\ar@{-}"5";"6",
			\ar@{-}"6";"4",
			\ar@{-}"4";"1",
			\ar@{-}"1";"0",
			\ar@{-}"3";"0",
			\ar@{-}"3";"1",
			\ar@{-}"3";"2",
			\ar@{-}"3";"4",
			\ar@{-}"3";"5",
			\ar@{-}"3";"6",
	\end{xy}}\ \ \ 
	{\begin{xy}
			0;<3.5pt,0pt>:<0pt,3.5pt>::
			(0,-5)="0",
			(-5,0)="1",
			(0,0)*{\bullet},
			(0,0)="2",
			(5,0)="3",
			(-10,5)="4",
			(-5,5)="5",
			(0,5)="6",
			(1.5,1.5)*{{\scriptstyle +}},
			(-1.5,-1.5)*{{\scriptstyle -}},
			\ar@{-}"0";"3",
			\ar@{-}"3";"6",
			\ar@{-}"6";"4",
			\ar@{-}"4";"0",
			\ar@{-}"2";"0",
			\ar@{-}"2";"1",
			\ar@{-}"2";"3",
			\ar@{-}"2";"4",
			\ar@{-}"2";"5",
			\ar@{-}"2";"6",
	\end{xy}}\ \ \ 
	{\begin{xy}
			0;<3.5pt,0pt>:<0pt,3.5pt>::
			(0,-5)="0",
			(5,-5)="1",
			(-5,0)="2",
			(0,0)*{\bullet},
			(0,0)="3",
			(5,0)="4",
			(-10,5)="5",
			(-5,5)="6",
			(0,5)="7",
			(1.5,1.5)*{{\scriptstyle +}},
			(-1.5,-1.5)*{{\scriptstyle -}},
			\ar@{-}"0";"1",
			\ar@{-}"1";"4",
			\ar@{-}"4";"7",
			\ar@{-}"7";"5",
			\ar@{-}"5";"0",
			\ar@{-}"3";"0",
			\ar@{-}"3";"1",
			\ar@{-}"3";"2",
			\ar@{-}"3";"4",
			\ar@{-}"3";"5",
			\ar@{-}"3";"6",
			\ar@{-}"3";"7",
	\end{xy}}\ \ \ 
	{\begin{xy}
			0;<3.5pt,0pt>:<0pt,3.5pt>::
			(0,-5)="0",
			(5,-5)="1",
			(10,-5)="2",
			(-5,0)="3",
			(0,0)*{\bullet},
			(0,0)="4",
			(5,0)="5",
			(-10,5)="6",
			(-5,5)="7",
			(0,5)="8",
			(1.5,1.5)*{{\scriptstyle +}},
			(-1.5,-1.5)*{{\scriptstyle -}},
			\ar@{-}"0";"2",
			\ar@{-}"2";"8",
			\ar@{-}"8";"6",
			\ar@{-}"6";"0",
			\ar@{-}"4";"0",
			\ar@{-}"4";"1",
			\ar@{-}"4";"2",
			\ar@{-}"4";"3",
			\ar@{-}"4";"5",
			\ar@{-}"4";"6",
			\ar@{-}"4";"7",
			\ar@{-}"4";"8",
	\end{xy}}
	{\begin{xy}
			0;<3.5pt,0pt>:<0pt,3.5pt>::
			(0,-2.5)="0",
			(5,-7.5)="1",
			(5,-2.5)="2",
			(-5,2.5)="3",
			(0,2.5)*{\bullet},
			(0,2.5)="4",
			(5,2.5)="5",
			(-10,7.5)="6",
			(0,7.5)="7",
			(-5,7.5)="8",
			(1.5,4)*{{\scriptstyle +}},
			(-1.5,1)*{{\scriptstyle -}},
			\ar@{-}"6";"1",
			\ar@{-}"1";"5",
			\ar@{-}"5";"7",
			\ar@{-}"7";"6",
			\ar@{-}"4";"0",
			\ar@{-}"4";"1",
			\ar@{-}"4";"2",
			\ar@{-}"4";"3",
			\ar@{-}"4";"5",
			\ar@{-}"4";"6",
			\ar@{-}"4";"7",
			\ar@{-}"4";"8",
	\end{xy}}
	\]
In Section 5, we give a more explicit result. Our second main result below shows that there are 16 convex $g$-fans $\Sigma_{a;b}$ with $a,b\in \{(0,0),(1,1,1),(1,2,1,2),(2,1,2,1)\}$ in $\tscfan$, and also gives a characterization of algebras whose $g$-fans are one of them. Here we denote by $t(M)_{A}$ (respectively, $t_{A}(M)$)  the minimal number of generators of a right (respectively, left) $A$-module $M$, and ${}^t\Sigma$ is the transpose of $\Sigma$ (Lemma \ref{transpose fan}).  

\begin{theorem}[Theorem \ref{theorem:charactarization_g-convex_ranktwo}]\label{theorem:charactarization_g-convex_ranktwo intro}
	Let $A$ be a basic finite dimensional algebra, $\{e_1, e_2\}$ a complete set of primitive orthogonal idempotents in $A$, and $P_i=e_i A$ $(i=1,2)$. 
	\begin{enumerate}[\rm(a)]
		\item $A$ is $g$-convex if and only if 
			\[\Sigma(A)=\Sigma_{a;b}\ \mbox{ for some }\ a,b\in \{(0,0),(1,1,1),(1,2,1,2),(2,1,2,1)\}.\]
		\item Let $a\in\{(0,0),(1,1,1),(1,2,1,2),(2,1,2,1)\}$. Then $\Sigma(A)=\Sigma_a\ast\Sigma$ holds for some $\Sigma\in\itscfanmp$ if and only if certain explicit conditions on $t_{e_1Ae_1}(e_1Ae_2)$, $t(e_1Ae_2)_{e_2Ae_2}$ and some other invariant is satisfied (see Theorem \ref{theorem:charactarization_g-convex_ranktwo}(b) for details). 
		
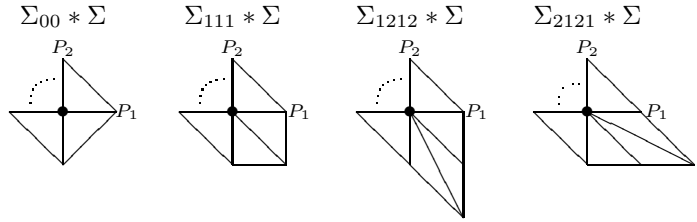
\[
		{\begin{xy}
				0;<4pt,0pt>:<0pt,4pt>::
				( 0,9) *+{\Sigma_{00}\ast\Sigma},
				(  0,0) *{\bullet}="(0,0)",
				( 0,5) ="(0,10)",
				( 0,6) *+{{\scriptstyle P_2}},
				( 5,0) ="(10,0)",
				( 6,0) *+{{\scriptstyle P_1}},
				%(-20,10) *+{\bullet}="(-20,10)",
				( 0,-5) ="(0,-10)",
				(  -5,0) ="(-10,0)",
				%(20,-10) *{\bullet}="(20,-10)",
				%(10,-10) *{\bullet}="(10,-10)",
				%(10,-20) *{\bullet}="(10,-20)",
				(0,3)*{}="(0,5)",
				(-3,0)*{}="(-5,0)",
				%(-10,10) *+{\bullet}="(-10,10)",
				\ar@{-} "(10,0)";"(-10,0)"
				\ar@{-} "(0,10)";"(0,-10)"
				%\ar@{-} "(0,0)";"(20,-10)"
				%\ar@{-} "(0,0)";"(10,-10)"
				%\ar@{-} "(0,0)";"(10,-20)"
				%\ar@{-} "(-20,10)";"(20,-10)"
				%\ar@{-} "(-10,10)";"(10,-10)"
				\ar@{-} "(10,0)";"(0,10)"
				%\ar@{-} "(0,10)";"(-20,10)"
				%\ar@{-} "(-20,10)";"(-10,0)"
				\ar@{-} "(-10,0)";"(0,-10)"
				%\ar@{-} "(0,-10)";"(20,-10)"
				\ar@{-} "(10,0)";"(0,-10)"
				%\ar@{-} "(10,-10)";"(0,-10)"
				%\ar@{-} "(10,-20)";"(0,-10)"
				\ar@/^-2mm/@{.} "(0,5)";"(-5,0)"
		\end{xy}}
		\ \ \ {\begin{xy}
				0;<2pt,0pt>:<0pt,2pt>::
				( 0,18) *+{\Sigma_{111}\ast\Sigma},
				(  0,0) *{\bullet},
				(  0,0) ="(0,0)",
				( 0,10) ="(0,10)",
				( 0,12) *+{{\scriptstyle P_2}},
				( 10,0) ="(10,0)",
				( 12.5,0) *+{{\scriptstyle P_1}},
				%(-20,10) *+{\bullet}="(-20,10)",
				( 0,-10)="(0,-10)",
				(  -10,0)="(-10,0)",
				%(20,-10) *{\bullet}="(20,-10)",
				(10,-10)="(10,-10)",
				%(10,-20) *{\bullet}="(10,-20)",
				(0,6)*{}="(0,5)",
				(-6,0)*{}="(-5,0)",
				%(-10,10) *+{\bullet}="(-10,10)",
				\ar@{-} "(10,0)";"(-10,0)"
				\ar@{-} "(0,10)";"(0,-10)"
				%\ar@{-} "(0,0)";"(20,-10)"
				\ar@{-} "(0,0)";"(10,-10)"
				%\ar@{-} "(0,0)";"(10,-20)"
				%\ar@{-} "(-20,10)";"(20,-10)"
				%\ar@{-} "(-10,10)";"(10,-10)"
				\ar@{-} "(10,0)";"(0,10)"
				%\ar@{-} "(0,10)";"(-20,10)"
				%\ar@{-} "(-20,10)";"(-10,0)"
				\ar@{-} "(-10,0)";"(0,-10)"
				%\ar@{-} "(0,-10)";"(20,-10)"
				\ar@{-} "(10,0)";"(10,-10)"
				\ar@{-} "(10,-10)";"(0,-10)"
				%\ar@{-} "(10,-20)";"(0,-10)"
				\ar@/^-2mm/@{.} "(0,5)";"(-5,0)"
		\end{xy}}
		\ \ \ {\begin{xy}
				0;<2pt,0pt>:<0pt,2pt>::
				( 0,18) *+{\Sigma_{1212}\ast\Sigma},
				(  0,0) *{\bullet},
				(  0,0) ="(0,0)",
				( 0,10) ="(0,10)",
				( 0,12) *+{{\scriptstyle P_2}},
				( 10,0) ="(10,0)",
				( 12.5,0) *+{{\scriptstyle P_1}},
				%(-20,10) *+{\bullet}="(-20,10)",
				( 0,-10)="(0,-10)",
				(  -10,0)="(-10,0)",
				%(20,-10) *{\bullet}="(20,-10)",
				(10,-10)="(10,-10)",
				(10,-20)="(10,-20)",
				(0,6)*{}="(0,5)",
				(-6,0)*{}="(-5,0)",
				%(-10,10) *+{\bullet}="(-10,10)",
				\ar@{-} "(10,0)";"(-10,0)"
				\ar@{-} "(0,10)";"(0,-10)"
				%\ar@{-} "(0,0)";"(20,-10)"
				\ar@{-} "(0,0)";"(10,-10)"
				\ar@{-} "(0,0)";"(10,-20)"
				%\ar@{-} "(-20,10)";"(20,-10)"
				%\ar@{-} "(-10,10)";"(10,-10)"
				\ar@{-} "(10,0)";"(0,10)"
				%\ar@{-} "(0,10)";"(-20,10)"
				%\ar@{-} "(-20,10)";"(-10,0)"
				\ar@{-} "(-10,0)";"(0,-10)"
				%\ar@{-} "(0,-10)";"(20,-10)"
				\ar@{-} "(10,0)";"(10,-10)"
				\ar@{-} "(10,-10)";"(10,-20)"
				\ar@{-} "(10,-20)";"(0,-10)"
				\ar@/^-2mm/@{.} "(0,5)";"(-5,0)"
			\end{xy}
		}\ \ \ {\begin{xy}
				0;<2pt,0pt>:<0pt,2pt>::
				( 0,18) *+{\Sigma_{2121}\ast\Sigma},
				(  0,0) *{\bullet},
				(  0,0) ="(0,0)",
				( 0,10) ="(0,10)",
				( 0,12) *+{{\scriptstyle P_2}},
				( 10,0) ="(10,0)",
				( 13,0) *+{{\scriptstyle P_1}},
				%(-20,10) *+{\bullet}="(-20,10)",
				( 0,-10)="(0,-10)",
				(  -10,0)="(-10,0)",
				(20,-10)="(20,-10)",
				(10,-10)="(10,-10)",
				(0,5)*{}="(0,5)",
				(-5,0)*{}="(-5,0)",
				%(-10,10) *+{\bullet}="(-10,10)",
				\ar@{-} "(10,0)";"(-10,0)"
				\ar@{-} "(0,10)";"(0,-10)"
				\ar@{-} "(0,0)";"(20,-10)"
				\ar@{-} "(0,0)";"(10,-10)"
				%\ar@{-} "(-20,10)";"(20,-10)"
				%\ar@{-} "(-10,10)";"(10,-10)"
				\ar@{-} "(10,0)";"(0,10)"
				%\ar@{-} "(0,10)";"(-20,10)"
				%\ar@{-} "(-20,10)";"(-10,0)"
				\ar@{-} "(-10,0)";"(0,-10)"
				\ar@{-} "(0,-10)";"(20,-10)"
				\ar@{-} "(20,-10)";"(10,0)"
				\ar@/^-2mm/@{.} "(0,5)";"(-5,0)"
		\end{xy}}
		%\hline
		%\end{tabular}
		\]
		\item Let $b\in\{(0,0),(1,1,1),(1,2,1,2),(2,1,2,1)\}$. Then $\Sigma(A)=\Sigma\ast{}^t\Sigma_b$ holds for some $\Sigma\in\itscfanpm$ if and only if certain explicit conditions on $t_{e_2Ae_2}(e_2Ae_1)$, $t(e_2Ae_1)_{e_1Ae_1}$ and some other invariant is satisfied (see Theorem \ref{theorem:charactarization_g-convex_ranktwo}(c) for details).
		
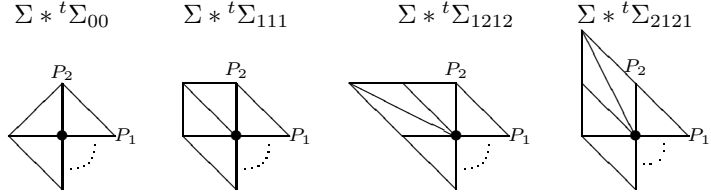
\[
		%\begin{tabular}{|c|c|c|c|}
		%\hline {$\Sigma_{0,0;-}$}&{$\Sigma_{1,1,1;-}$}&{$\Sigma_{1,2,1,2;-}$}&{$\Sigma_{2,1,2,1;-}$}\\
		{\begin{xy}
				0;<0pt,4pt>:<4pt,0pt>::
				( 11.5,0) *+{\Sigma\ast{}^t\Sigma_{00}},
				(  0,0) *{\bullet}="(0,0)",
				( 0,5) ="(0,10)",
				( 0,6) *+{{\scriptstyle P_1}},
				( 5,0) ="(10,0)",
				( 6,0) *+{{\scriptstyle P_2}},
				%(-20,10) *+{\bullet}="(-20,10)",
				( 0,-5) ="(0,-10)",
				(  -5,0) ="(-10,0)",
				%(20,-10) *{\bullet}="(20,-10)",
				%(10,-10) *{\bullet}="(10,-10)",
				%(10,-20) *{\bullet}="(10,-20)",
				(0,3)*{}="(0,5)",
				(-3,0)*{}="(-5,0)",
				%(-10,10) *+{\bullet}="(-10,10)",
				\ar@{-} "(10,0)";"(-10,0)"
				\ar@{-} "(0,10)";"(0,-10)"
				%\ar@{-} "(0,0)";"(20,-10)"
				%\ar@{-} "(0,0)";"(10,-10)"
				%\ar@{-} "(0,0)";"(10,-20)"
				%\ar@{-} "(-20,10)";"(20,-10)"
				%\ar@{-} "(-10,10)";"(10,-10)"
				\ar@{-} "(10,0)";"(0,10)"
				%\ar@{-} "(0,10)";"(-20,10)"
				%\ar@{-} "(-20,10)";"(-10,0)"
				\ar@{-} "(-10,0)";"(0,-10)"
				%\ar@{-} "(0,-10)";"(20,-10)"
				\ar@{-} "(10,0)";"(0,-10)"
				%\ar@{-} "(10,-10)";"(0,-10)"
				%\ar@{-} "(10,-20)";"(0,-10)"
				\ar@/^2mm/@{.} "(0,5)";"(-5,0)"
		\end{xy}}
		\ \ \ {\begin{xy}
				0;<0pt,2pt>:<2pt,0pt>::
				( 23,0) *+{\Sigma\ast{}^t\Sigma_{111}},
				(  0,0) *{\bullet},
				(  0,0) ="(0,0)",
				( 0,10) ="(0,10)",
				( 0,12) *+{{\scriptstyle P_1}},
				( 10,0) ="(10,0)",
				( 12.5,0) *+{{\scriptstyle P_2}},
				%(-20,10) *+{\bullet}="(-20,10)",
				( 0,-10)="(0,-10)",
				(  -10,0)="(-10,0)",
				%(20,-10) *{\bullet}="(20,-10)",
				(10,-10)="(10,-10)",
				%(10,-20) *{\bullet}="(10,-20)",
				(0,6)*{}="(0,5)",
				(-6,0)*{}="(-5,0)",
				%(-10,10) *+{\bullet}="(-10,10)",
				\ar@{-} "(10,0)";"(-10,0)"
				\ar@{-} "(0,10)";"(0,-10)"
				%\ar@{-} "(0,0)";"(20,-10)"
				\ar@{-} "(0,0)";"(10,-10)"
				%\ar@{-} "(0,0)";"(10,-20)"
				%\ar@{-} "(-20,10)";"(20,-10)"
				%\ar@{-} "(-10,10)";"(10,-10)"
				\ar@{-} "(10,0)";"(0,10)"
				%\ar@{-} "(0,10)";"(-20,10)"
				%\ar@{-} "(-20,10)";"(-10,0)"
				\ar@{-} "(-10,0)";"(0,-10)"
				%\ar@{-} "(0,-10)";"(20,-10)"
				\ar@{-} "(10,0)";"(10,-10)"
				\ar@{-} "(10,-10)";"(0,-10)"
				%\ar@{-} "(10,-20)";"(0,-10)"
				\ar@/^2mm/@{.} "(0,5)";"(-5,0)"
		\end{xy}}
		\ \ \ {\begin{xy}
				0;<0pt,2pt>:<2pt,0pt>::
				( 23,0) *+{\Sigma\ast{}^t\Sigma_{1212}},
				(  0,0) *{\bullet},
				(  0,0) ="(0,0)",
				( 0,10) ="(0,10)",
				( 0,12) *+{{\scriptstyle P_1}},
				( 10,0) ="(10,0)",
				( 12.5,0) *+{{\scriptstyle P_2}},
				%(-20,10) *+{\bullet}="(-20,10)",
				( 0,-10)="(0,-10)",
				(  -10,0)="(-10,0)",
				%(20,-10) *{\bullet}="(20,-10)",
				(10,-10)="(10,-10)",
				(10,-20)="(10,-20)",
				(0,6)*{}="(0,5)",
				(-6,0)*{}="(-5,0)",
				%(-10,10) *+{\bullet}="(-10,10)",
				\ar@{-} "(10,0)";"(-10,0)"
				\ar@{-} "(0,10)";"(0,-10)"
				%\ar@{-} "(0,0)";"(20,-10)"
				\ar@{-} "(0,0)";"(10,-10)"
				\ar@{-} "(0,0)";"(10,-20)"
				%\ar@{-} "(-20,10)";"(20,-10)"
				%\ar@{-} "(-10,10)";"(10,-10)"
				\ar@{-} "(10,0)";"(0,10)"
				%\ar@{-} "(0,10)";"(-20,10)"
				%\ar@{-} "(-20,10)";"(-10,0)"
				\ar@{-} "(-10,0)";"(0,-10)"
				%\ar@{-} "(0,-10)";"(20,-10)"
				\ar@{-} "(10,0)";"(10,-10)"
				\ar@{-} "(10,-10)";"(10,-20)"
				\ar@{-} "(10,-20)";"(0,-10)"
				\ar@/^2mm/@{.} "(0,5)";"(-5,0)"
			\end{xy}
		}\ \ \ {\begin{xy}
				0;<0pt,2pt>:<2pt,0pt>::
				( 23,0) *+{\Sigma\ast{}^t\Sigma_{2121}},
				(  0,0) *{\bullet},
				(  0,0) ="(0,0)",
				( 0,10) ="(0,10)",
				( 0,12) *+{{\scriptstyle P_1}},
				( 10,0) ="(10,0)",
				( 13,0) *+{{\scriptstyle P_2}},
				%(-20,10) *+{\bullet}="(-20,10)",
				( 0,-10)="(0,-10)",
				(  -10,0)="(-10,0)",
				(20,-10)="(20,-10)",
				(10,-10)="(10,-10)",
				(0,5)*{}="(0,5)",
				(-5,0)*{}="(-5,0)",
				%(-10,10) *+{\bullet}="(-10,10)",
				\ar@{-} "(10,0)";"(-10,0)"
				\ar@{-} "(0,10)";"(0,-10)"
				\ar@{-} "(0,0)";"(20,-10)"
				\ar@{-} "(0,0)";"(10,-10)"
				%\ar@{-} "(-20,10)";"(20,-10)"
				%\ar@{-} "(-10,10)";"(10,-10)"
				\ar@{-} "(10,0)";"(0,10)"
				%\ar@{-} "(0,10)";"(-20,10)"
				%\ar@{-} "(-20,10)";"(-10,0)"
				\ar@{-} "(-10,0)";"(0,-10)"
				\ar@{-} "(0,-10)";"(20,-10)"
				\ar@{-} "(20,-10)";"(10,0)"
				\ar@/^2mm/@{.} "(0,5)";"(-5,0)"
		\end{xy}}
		\]
		%\hline
		%\end{tabular}
	\end{enumerate}
\end{theorem}

Further, in a forthcoming paper \cite{AHIKM2}, we will give a complete answer to Problem \ref{classify convex} for $d=3$.

In Section 6, we give another application of methods developed in this paper. For a finite dimensional $k$-algebra $A$, we denote by  $\Hasse(\twosilt A)$ the Hasse quiver of 2-term basic silting complexes of $A$. 
It is important to study the following problem.

\begin{problem}
For a given finite dimensional algebra $A$, determine the number of connected components of $\Hasse(\twosilt A)$.
\end{problem}

This question is a 2-term version of transitivity problem of iterated silting mutation \cite{AiI,D} and also related to connectedness of the space of Bridgeland stability conditions \cite{B}, e.g.\ \cite{AMY,BPP,QW}. 

If $A$ is either hereditary \cite{BMRRT} or $g$-finite \cite{AIR}, then $\Hasse(\twosilt A)$ is connected.
There are several examples of $A$ such that $\Hasse(\twosilt A)$ has precisely 2 connected components \cite[Theorems 6.17, 6.18]{DIJ}\cite{T,KM}.
There are examples of $A$ such that $\Hasse(\twosilt A)$ has at least 4 connected components \cite{T}.
The following our third main result of this paper shows that the number of connected components of $\Hasse(\twosilt A)$ can be arbitrary large.

\begin{theorem}[Theorem \ref{many}]\label{many in intro}
Let $k$ be an arbitrary field. For each positive integer $N\ge1$, there exists a finite dimensional $k$-algebra $A_N$ with $|A_N|=2$ such that $\Hasse(\twosilt A_N)$ has precisely $N$ connected components.
\end{theorem}

We do not know if there exists a finite dimensional algebra $A$ such that $\Hasse(\twosilt A)$ has infinitely many connected components.

\section{Preliminaries}

\subsection{Preliminaries on fans}

We recall some fundamental materials on fans. 
We refer the reader to e.g.\ \cite{F,BR,BP} for these materials. 

\medskip
Let $\R^d$ be an Euclidean space with inner product $\langle \cdot, \cdot \rangle$. A \textit{convex polyhedral cone} $\sigma$ is a set of the form
\[\sigma=\cone\{v_1,\ldots,v_s\}:=\{ \sum_{i=1}^s r_i v_i \mid r_i \geq 0\}\subset\R^d,\] 
where $v_1,\ldots,v_s \in \mathbb{R}^d$. For example, $\{0\}$ is a convex polyhedral cone. We collect some notions concerning convex polyhedral cones. Let $\sigma$ be a convex polyhedral cone. 
\begin{enumerate}[$\bullet$]
\item The dimension of $\sigma$ is the dimension of the linear space generated by $\sigma$. 
\item We say that $\sigma$ is \textit{strongly convex} if $\sigma \cap (-\sigma)=\{0\}$ holds, 
i.e., $\sigma$ does not contain a linear subspace of positive dimension. 
\item We call $\sigma$ \textit{rational} if each $v_i$ can be taken from $\mathbb{Q}^d$. 
\item A \emph{supporting hyperplane} of $\sigma$ is a hyperplane $\{v \in \sigma \mid \langle u,v \rangle=0\}$ in $\R^d$ given by some $u \in \mathbb{R}^d$ satisfying $\sigma\subset\{v \in \R^d \mid \langle u,v \rangle\ge0\}$.
\item A \textit{face} of $\sigma$ is the intersection of $\sigma$ with a supporting hyperplane of $\sigma$.
\end{enumerate}

\emph{In what follows, a cone means a strongly convex rational polyhedral cone for short.}

\begin{definition}
A \textit{fan} $\Sigma$ in $\mathbb{R}^d$ is a collection of cones in $\mathbb{R}^d$ such that 
\begin{enumerate}[\rm(a)]
\item each face of a cone in $\Sigma$ is also contained in $\Sigma$, and 
\item the intersection of two cones in $\Sigma$ is a face of each of those two cones. 
\end{enumerate}
For each $i\ge0$, we denote by $\Sigma_i$ the subset of cones of dimension $i$. For example, $\Sigma_0$ consists of the trivial cone $\{0\}$. We call each element in $\Sigma_1$ a \emph{ray} of $\Sigma$.
\end{definition}

We collect some notions concerning fans used in this paper. Let $\Sigma$ be a fan in $\mathbb{R}^d$. 
\begin{enumerate}[$\bullet$]
\item We call $\Sigma$ \textit{finite} if it consists of a finite number of cones. 
\item We call $\Sigma$ \textit{complete} if $\bigcup_{\sigma \in \Sigma}\sigma=\mathbb{R}^d$. 
\item We call $\Sigma$ \textit{nonsingular} (or \textit{smooth}) if each maximal cone in $\Sigma$ is generated by a $\Z$-basis for $\Z^d$. 
\end{enumerate}

We prepare some notions which will be used in this paper.

\begin{definition}
Let $\Sigma$ be a nonsingular fan in $\R^d$. We call $\Sigma$ \emph{pairwise positive} if the following condition is satisfied.
\begin{enumerate}[$\bullet$]
\item For each two adjacent maximal cones $\sigma,\tau\in\Sigma_d$, take $\Z$-basis $\{v_1,\ldots,v_{d-1},v_d\}$ and $\{v_1,\ldots,v_{d-1},v'_d\}$ of $\Z^d$ such that $\sigma=\cone\{v_1,\ldots,v_{d-1},v_d\}$ and $\tau=\cone\{v_1,\ldots,v_{d-1},v'_d\}$.
Then $v_d+v'_d$ belongs to $\cone\{v_1,\ldots,v_{d-1}\}$.
\end{enumerate}
\end{definition}

\begin{definition}\label{define isomorphism of g-fans}
Let $\Sigma$ and $\Sigma'$ be fans in $\R^d$ and $\R^{d'}$ respectively. 
\begin{enumerate}[\rm(a)]
\item An \emph{isomorphism $\Sigma\simeq\Sigma'$ of fans} is an isomorphism $\Z^d\simeq\Z^{d'}$ of abelian groups such that the induced linear isomorphism $\R^d\to\R^{d'}$ gives a bijection $\Sigma\simeq\Sigma'$ between cones.
\item Let $(\Sigma,\sigma_+)$ and $(\Sigma',\sigma'_+)$ be sign-coherent fans in $\R^d$ and $\R^{d'}$ respectively. An \emph{isomorphism of sign-coherent fans} is an isomorphism $f:\Sigma\simeq\Sigma'$ of fans such that $\{f(\sigma_+),f(-\sigma_+)\}=\{\sigma'_+,-\sigma'_+\}$.
\end{enumerate}
\end{definition}

\begin{definition}\label{from fan to polytope}
Let $\Sigma$ be a nonsingular fan in $\R^d$. For each $\sigma\in\Sigma_d$, take a basis $v_1,\ldots,v_d$ of $\Z^d$ such that $\sigma=\cone\{v_1,\ldots,v_d\}$, and let $\sigma_{\le1}:=\conv\{0,v_1,\ldots,v_d\}\subset \R^d$.
Define a (not necessarily convex) polytope in $\R^d$ by
\[\P(\Sigma):=\bigcup_{\sigma\in\Sigma_d}\sigma_{\le1}.\]
We say that $\Sigma$ is \emph{convex} if $\P(\Sigma)$ is convex.
\end{definition}

\subsection{Sign-coherent fans of rank $2$}\label{Sign-coherent fans of rank $2$}
In this subsection, we introduce some terminologies of sign-coherent fans of rank $2$, and discuss some fundamental properties.

Let $\Sigma$ be a complete nonsingular fan of rank 2. 
We denote the rays of $\Sigma$ by
\begin{equation}\label{ray}
v_1,v_2,\ldots,v_{n-1},v_{n}=v_0
\end{equation}
which are indexed in a clockwise orientation. 
For each $1\le i\le n$, since $\Sigma$ is nonsingular, there exists an integer $a_i$ satisfying
\[a_i v_{i}=v_{i-1} +v_{i+1}\ \mbox{ for each }\ 1\le i\le n.\]
Following the terminology of frieze by \cite{CC1,CC2}, we call the sequence of integers
\begin{equation}\label{define s}
\df(\Sigma)=(a_1,\ldots,a_{n})
\end{equation}
the \emph{quiddity sequence} of $\Sigma$. In fact, $\Sigma$ is uniquely determined by its quiddity sequence.
A fan with quiddity sequence $(a_1,\ldots,a_{n})$ is denoted by
\[\Sigma(a_1,\ldots,a_n).\]

\begin{remark}
\cite[Section 2.5]{F} An integer sequence $(a_1,\ldots,a_n)$  is a quiddity sequence of nonsingular complete fan of rank 2 if and only if it satisfies
\begin{eqnarray*}
\sum_{i=1}^na_i=3n-12\ \mbox{ and }\ \begin{bmatrix}0&-1\\ 1&a_1\end{bmatrix}\begin{bmatrix}0&-1\\ 1&a_2\end{bmatrix}\cdots\begin{bmatrix}0&-1\\ 1&a_n\end{bmatrix}=\begin{bmatrix}1&0\\ 0&1\end{bmatrix}.
\end{eqnarray*} 
\end{remark}

\begin{definition}\label{define fan_sc}
We denote by $\itscfan$ the set of all (possibly infinite) fans $\Sigma$ satisfying that
\begin{enumerate}[$\bullet$]
\item $\Sigma$ is a sign-coherent fans (Definition \ref{define sign-coherent}) of rank 2 with positive and negative cones
\[\sigma_+:=\cone\{(1,0),(0,1)\}\ \mbox{ and }\ \sigma_-:=\cone\{(-1,0),(0,-1)\}\ \mbox{ respectively,}\]
\item each ray is a face of precisely two facets. 
\end{enumerate}
We denote the subset of complete fans by
\[\tscfan\subset\itscfan.\] 
\end{definition}

For $\Sigma\in\tscfan$, we denote the rays of $\Sigma$ in a clockwise orientation by
\begin{eqnarray*}
\Sigma_1=\{v_1:=(1,0),v_2,\ldots,v_{n-1},v_{n}=v_0:=(0,1)\}. 
\end{eqnarray*}
Then there exists $2\le i\le n-2$ such that $v_i=(0,-1)$ and $v_{i+1}=(-1,0)$.
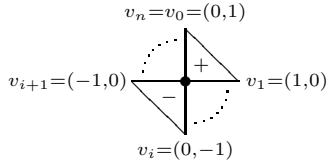
\[\begin{xy}
0;<4pt,0pt>:<0pt,4pt>::
(0,-5)="0",
(-5,0)="1",
(0,0)*{\bullet},
(0,0)="2",
(5,0)="3",
(0,5)="4",
(0,-4)="a",
(4,0)="b",
(0,4)="c",
(-4,0)="d",
(0,6.5)*{{\scriptstyle v_n=v_0=(0,1)}},
(9.5,0)*{{\scriptstyle v_1=(1,0)}},
(0,-6.5)*{{\scriptstyle v_i=(0,-1)}},
(-11,0)*{{\scriptstyle v_{i+1}=(-1,0)}},
(1.5,1.5)*{{\scriptstyle +}},
(-1.5,-1.5)*{{\scriptstyle -}},
\ar@{-}"0";"1",
\ar@{-}"4";"3",
\ar@{-}"2";"0",
\ar@{-}"2";"1",
\ar@{-}"2";"3",
\ar@{-}"2";"4",
\ar@/^-2mm/@{.} "a";"b",
\ar@/^-2mm/@{.} "c";"d",
\end{xy}\]
In this case, it is more convenient to rewrite \eqref{define s} as 
\[\df(\Sigma)=(a_1,a_2,\ldots,a_i;a_n,a_{n-1},\ldots,a_{i+1}).\] 
Thus we mainly use the notation
\[\Sigma(a_1,\ldots,a_i;a_n,\ldots,a_{i+1})=\Sigma_{a_1,\ldots, a_i;a_n,\ldots,a_{i+1}}\]
instead of $\Sigma(a_1,\ldots,a_n)$.

We consider subsets
\[\begin{array}{ccc}\itscfanpm&\subset&\itscfan\\
\cup&&\cup\\
\tscfanpm&\subset&\tscfan
\end{array}\]
which consist of fans $\Sigma$ containing $\sigma_{-+}:=\cone\{(-1,0),(0,1)\}$, i.e.\ $\Sigma$ has the following form.
	
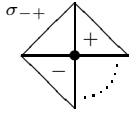
\[
	\begin{xy}
		0;<4pt,0pt>:<0pt,4pt>::
		(0,-5)="0",
		(-5,0)="1",
		(0,0)*{\bullet},
		(0,0)="2",
		(5,0)="3",
		(0,5)="4",
		(-4.5,4)*{{\scriptstyle \sigma_{-+}}}, 
		(1.5,1.5)*{{\scriptstyle +}},
		(-1.5,-1.5)*{{\scriptstyle -}},
		%(2.8,-2.8)*{{\scriptstyle ?}},
		( 0,-4) ="a",( 4,0) ="b", 
		\ar@{-}"0";"1",
		\ar@{-}"1";"4",
		\ar@{-}"4";"3",
		\ar@{-}"2";"0",
		\ar@{-}"2";"1",
		\ar@{-}"2";"3",
		\ar@{-}"2";"4",
		\ar@/^-2mm/@{.} "a";"b",
		\end{xy}
\]
Thus the rays and the facets of $\Sigma\in\tscfanpm$ are written as
\begin{eqnarray}\label{ray 2}
\Sigma_1&=&\{v_1=(1,0),v_2,\ldots,v_{n-2}=(0,-1),v_{n-1}=(-1,0),v_n=v_0=(0,1)\},\\ \label{facet 2}
\Sigma_2&=&\{\sigma_1,\ldots,\sigma_{n-3},\sigma_{n-2}=\sigma_-,\sigma_{n-1}=\sigma_{-+},\sigma_n=\sigma_+\}.
\end{eqnarray}
Similarly, we define $\itscfanmp$ and $\tscfanmp$ as the subsets of $\itscfan$ and $\tscfan$ respectively which consist of fans containing  $\sigma_{+-}:=\cone\{(1,0),(0,-1)\}$.
The following observations are clear.

\begin{lemma}\label{transpose fan}
The following assertions hold.
\begin{enumerate}[\rm(a)]
\item The correspondence $\Sigma\mapsto {}^t\Sigma:=\{{}^t\sigma\mid \sigma\in \Sigma\}$ gives bijections $\itscfanpm \to \itscfanmp$ and $\tscfanpm \to \tscfanmp$, where ${}^t(-):\R^2\to\R^2$ is an isomorphism $(x,y)\mapsto(y,x)$.
\item Let $\Sigma\in\tscfan$. 
Then $\Sigma\in\tscfanpm$ (respectively, $\Sigma\in\tscfanmp$) holds if and only if $\df(\Sigma)$ has the form
\[(b_1,\ldots,b_m;0,0)\ \mbox{ (respectively, $(0,0;b_1,\ldots,b_m)$)}.\]
In this case, $b_i\ge0$ holds for any $1\le i\le m$. 
\end{enumerate}
\end{lemma}

Let us consider the following elementary operation of fans of rank 2.

\begin{definition} 
	\label{definition:quiddity sequence} 
Let $\Sigma$ be a (possibly infinite) nonsingular fan of rank 2. For a cone $\sigma:=\cone\{u,v\}$ of $\Sigma$, we define a new nonsingular fan $D_\sigma(\Sigma)$ by
\begin{eqnarray*}
		D_\sigma(\Sigma)_1&=&\Sigma_1\sqcup\{\cone\{u+v\}\},\\
		D_\sigma(\Sigma)_2&=&(\Sigma_2\setminus\{\sigma\})\sqcup\{\cone\{u,u+v\},\cone\{v,u+v\}\}.
    \end{eqnarray*}
We call $D_\sigma(\Sigma)$ the \emph{subdivision} of $\Sigma$ at $\sigma$. 
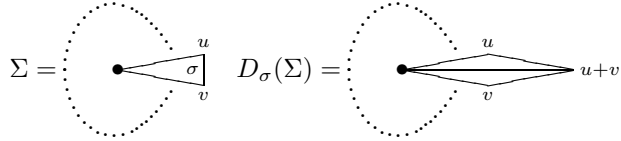
\[{\Sigma=\begin{xy}
0;<4pt,0pt>:<0pt,4pt>::
(0,0)*{\bullet},
(0,0)="2",
(5,2)="3",
(5,-2)="4",
(8,1.5)="u",
(8,-1.5)="v",
(8,2.5)*{{\scriptstyle u}},
(8,-2.5)*{{\scriptstyle v}},
"3";"4"**\crv{~*{.}(0,10)&(-10,0)&(0,-10)},
(7,0)*{{\scriptstyle\sigma}},
\ar@{-}"u";"v",
\ar@{-}"2";"v",
\ar@{-}"2";"u",
\end{xy}}\ \ \ 
{D_\sigma(\Sigma)=\begin{xy}
0;<4pt,0pt>:<0pt,4pt>::
(0,0)*{\bullet},
(0,0)="2",
(5,2)="3",
(5,-2)="4",
(8,1.5)="u",
(8,-1.5)="v",
(16,0)="w",
(8,2.5)*{{\scriptstyle u}},
(8,-2.5)*{{\scriptstyle v}},
(18.5,0)*{{\scriptstyle u+v}},
"3";"4"**\crv{~*{.}(0,10)&(-10,0)&(0,-10)},
\ar@{-}"2";"w",
\ar@{-}"2";"v",
\ar@{-}"2";"u",
\ar@{-}"w";"v",
\ar@{-}"w";"u",
\end{xy}}
\]
\end{definition}

For a sequence $a=(a_1\ldots,a_n)$ and $1\le i\le n$, we define a new sequence by
\begin{equation}\label{define D_j for a}
D_i(a) = (a_1,\ldots,a_{i-1},a_{i}+1,1,a_{i+1}+1,a_{i+2},\ldots,a_n).
\end{equation}
For a complete nonsingular fan $\Sigma$ with rays \eqref{ray} and $\sigma_i:=\cone\{v_i,v_{i+1}\}$ for $1\leq i \leq n$, we have 
    \begin{equation}\label{s D_j=D_j s}
        \df\circ D_{\sigma_i}(\Sigma) = D_i\circ\df(\Sigma).
    \end{equation}

\begin{example}
Figure \ref{fig:tree diagram} gives fans in $\tscfanpm$ with at most 8 facets, where
\[\Sigma_{a_1,\ldots,a_n}:=\Sigma(a_1,\ldots,a_n;0,0)\]
and each edge shows a subdivision. Figure \ref{fig:tree diagram 2} gives examples of algebras whose $g$-fans are given in Figure \ref{fig:tree diagram}.
For example, $\Sigma_{111}$ is the $g$-vector fan of a cluster algebra of type $A_2$ \cite{FZ2,FZ4}. Similarly, $\Sigma_{1212}$ and $\Sigma_{2121}$ are the $g$-vector fans of cluster algebras of type $B_2$, and $\Sigma_{131313}$ and $\Sigma_{313131}$ are the $g$-vector fans of cluster algebras of type $G_2$.
\end{example}

Later we need the following observation (cf.\ \cite[Section 4.3]{F}).

\begin{proposition}\label{proposition:F=F'}
Each fan in $\tscfanpm$ can be obtained from $\Sigma(0,0;0,0)$  by a  sequence of subdivisions.
	\[
	\Sigma(0,0;0,0)=\begin{xy}
		0;<4pt,0pt>:<0pt,4pt>::
		(0,-5)="0",
		(-5,0)="1",
		(0,0)*{\bullet},
		(0,0)="2",
		(5,0)="3",
		(0,5)="4",
		(1.5,1.5)*{{\scriptstyle +}},
		(-1.5,-1.5)*{{\scriptstyle -}},
		\ar@{-}"0";"1",
		\ar@{-}"1";"4",
		\ar@{-}"4";"3",
		\ar@{-}"2";"0",
		\ar@{-}"2";"1",
		\ar@{-}"2";"3",
		\ar@{-}"2";"4",
		\ar@{-}"3";"0",
	\end{xy}
\]
\end{proposition}

To prove this, we need the following preparation.

\begin{lemma}[{cf.\ \cite[p.43]{F}}]
	\label{lemma:quiddity sequence contains 1}
	Let $\Sigma\in \tscfanpm$ and $\df(\Sigma)=(a_1,\ldots,a_{n-2};0,0)$. If $n\ge 5$, then there exists $2\le i\le n-3$  satisfying $a_i=1$.
\end{lemma}
\begin{proof}
	Let $v_i=(x_i,y_i) \in \mathbb{Z}^2$ for $1\leq i\leq n$. 
	Assume that $n\ge 5$ and $a_i\ge 2$ for any $2\le i\le n-3$.
	We claim that $x_{i+1}\ge x_i$ holds for each  $1\le i\le n-3$.
	 In fact, $n\ge 5$ implies $x_2\ge 1=x_1$. Then we have
\[
x_{i+1}=a_{i}x_i-x_{i-1}\ge 2x_i-x_{i-1}\ge x_i
\]
for each  $2\le i\le n-3$, and the claim follows inductively. Consequently $1=x_1\le x_2 \le \cdots \le  x_{n-2}=0$ holds, a contradiction.
\end{proof}

We are ready to prove Proposition \ref{proposition:F=F'}.

\begin{proof}[Proof of Proposition \ref{proposition:F=F'}]
Let $F\subset\tscfanpm$ be the set of fans obtained from $\Sigma(0,0;0,0)$ by a sequence of subdivisions. It suffices to show $\tscfanpm=F$.

	We will show that each $\Sigma\in \tscfanpm$ belongs to $F$ by using induction on $n=\#\Sigma_2$.

	Clearly $n\ge 4$ holds. If $n=4$, then $\Sigma=\Sigma(0,0;0,0)\in F$.
	
	Suppose that $\Sigma$ with $\#\Sigma_2=n\ge 5$ belongs to $\tscfanpm$. In terms of \eqref{ray 2} and \eqref{facet 2}, there exists $2\le i\le n-3$ satisfying $v_i=v_{i-1}+v_{i+1}$ by Lemma \ref{lemma:quiddity sequence contains 1}.
	Since $v_{i-1},v_{i+1}$ forms a $\Z$-basis of $\Z^2$, we obtain a new fan $\Sigma'\in\tscfanpm$ by
	\begin{eqnarray*}
	\Sigma'_1&:=&\Sigma_1\setminus\{v_i\},\\
	\Sigma'_2&:=&(\Sigma_2\setminus\{\sigma_{i-1},\sigma_i\})\cup \{\sigma\}\ \mbox{ for }\ \sigma:=\cone\{v_{i-1},v_{i+1}\}.
	\end{eqnarray*}
	Since $\#\Sigma'_2=n-1$, the induction hypothesis implies $\Sigma'\in F$. Thus $\Sigma=D_\sigma(\Sigma')\in F$ holds.
\end{proof}

\begin{remark}
\begin{enumerate}[\rm(a)]
\item
For each $n\ge 1$, we have a bijection
\[\{\Sigma\in\tscfanpm\ |\ \#\Sigma_2=n+3\}\simeq\{\mbox{the ways to parenthesize $n$ factors completely}\},\]
where parentheses show how cones in the fourth quadrant are obtained by iterated subdivisions.
For example, $\Sigma_{141222}$ in Figure \ref{fig:tree diagram} has 5 cones $\sigma_1,\ldots,\sigma_5$ in the fourth quadrant in terms of \eqref{facet 2}, and they are parenthesized as $\sigma_1(((\sigma_2\sigma_3)\sigma_4)\sigma_5)$. In particular, we have
$$\#\{\Sigma\in\tscfanpm\ |\ \#\Sigma_2=n+3\}=\frac{1}{n}\dbinom{2n-2}{n-1}.$$
%\old{The set of quiddity sequences $(a_1,\ldots,a_n;0,0)$ of $\tscfanpm$ are in bijection with the set of triangulations of $n$-regular polygons, that is,}
\item
We also have a bijection
\[\{\Sigma\in\tscfanpm\ |\ \#\Sigma_2=n+3\}\simeq\{\mbox{Triangulations of a regular $(n+1)$-gon}\},\]
where $\Sigma_{a_1,\ldots,a_{n+1}}$ corresponds to a triangulation satisfying the following condition: 
Let $1,2,\ldots,n+1$ be the vertices of the regular $(n+1)$-gon in a clockwise direction, and $a_i$ ($1\leq i\leq n+1$) the number of triangles containing the vertex $i$ in the triangulation. 
For example, $\Sigma_{141222}$ corresponds to the following triangulation, where $1$ is the top vertex.

\[ %hexagon
\begin{tikzpicture}
    \coordinate(0) at (0:0); 
    \coordinate(30) at (30:0.8);
    \coordinate(90) at (90:0.8); 
    \coordinate(150) at (150:0.8);
    \coordinate(-30) at (-30:0.8);
    \coordinate(-90) at (-90:0.8); 
    \coordinate(-150) at (-150:0.8);    
    %\node at(0) {$\bullet$};
    \node at(30) {$\bullet$};
    \node at(90) {$\bullet$};
    \node at(150) {$\bullet$};
    \node at(-30) {$\bullet$};
    \node at(-90) {$\bullet$};
    \node at(-150) {$\bullet$};

    \draw[] (30)--(90)--(150)--(-150)--(-90)--(-30)--cycle;
    \draw[] (30)--(150) (30)--(-150) (30)--(-90);
\end{tikzpicture}
\]
\item
It is well-known that the triangulations of a regular polygon give rise to Conway-Coxeter friezes \cite{CC1,CC2} (see also \cite{BFGST,BPT}). Thus we have a certain bijection between $\tscfanpm$ and Conway-Coxeter friezes which preserves the quiddity sequences.
\end{enumerate}
\end{remark}

\section{Basic results in silting theory}

\subsection{Preliminaries}\label{section 3.1}
Let $A$ be a finite dimensional algebra over a field $k$. Let $K_0(\proj A)$ be the Grothendieck group of the additive category $\proj A$, which is identified with the Grothendieck group of the triangulated category $\Kb(\proj A)$.
We recall basic results on silting theory from \cite{AiI,AIR,AHIKM}. First we recall the definition of 2-term silting complexes.

\begin{definition}\label{define silting}
Let $T=(T^i,d^i)\in\Kb(\proj A)$.
\begin{enumerate}[\rm(a)]
\item $T$ is called  \emph{presilting} if $\Hom_{\Kb(\proj A)}(T,T[\ell])=0$ for all positive integers $\ell$.\item $T$ is called \emph{silting} if it is presilting and $\Kb(\proj A)=\thick T$.
\item $T$ is called  \emph{2-term} if $T^i = 0$ for all $i\not= 0,-1$. In this case, the class $[T]=[T^0]-[T^{-1}]\in K_0(\proj A)$ of $T$ is called the \emph{$g$-vector} of $T$.
\item An element of $K_0(\proj A)$ is \emph{rigid} if it is a $g$-vector of some 2-term presilting complex.
\end{enumerate}
We denote by $\silt A$ (respectively, $\psilt A$, $\twosilt A$,  $\twopsilt A$) the set of isomorphism classes of basic silting (respectively, presilting, 2-term silting, 2-term presilting) complexes of $\Kb(\proj A)$.
Note that a 2-term presilting complex $T$ is silting if and only if $|T|=|A|$ holds.

For $T,U\in\silt A$, we write $T\ge U$ if $\Hom_{\Kb(\proj A)}(T,U[\ell])=0$ holds for all positive integers $\ell$. Then $(\silt A,\ge)$ is a partially ordered set \cite{AiI}.
\end{definition}

In this paper, the subposet $(\twosilt A,\ge)$ of $(\silt A,\ge)$ plays a central role. It is known that $\Hasse(\twosilt A)$ is $n$-regular for $n:=|A|$. More precisely, let $T=T_1\oplus\cdots\oplus T_n\in\twosilt A$ with indecomposable $T_i$. For each $1\le i\le n$, there exists precisely one $T'\in\twosilt A$ such that $T'=T'_i\oplus(\bigoplus_{j\neq i}T_j)$ for some $T'_i\neq T_i$. In this case, we call $T'$ \emph{mutation} of $T$ at $T_i$ and write
\[T'=\mu_{T_i}(T)=\mu_i(T).\]
In this case, either $T>T'$ or $T'<T$ holds. 
We denote $T'$ by $\mu^-_i(T)$ (respectively, $\mu^+_i(T)$) if $T>T'$ and call it \emph{left mutation} (respectively, \emph{right mutation}). 
The following result is fundamental in silting theory. 

\begin{proposition}\label{mutation=exchange}
Let $T,T'\in\twosilt A$. Take a decomposition $T=T_1\oplus\cdots\oplus T_n$ with indecomposable $T_i$. Then the following conditions are equivalent.
\begin{enumerate}[\rm(a)]
\item $T>T'$, and $T$ and $T'$ are mutation of each other.
\item There is an arrow $T\to T'$ in $\Hasse(\twosilt A)$.
\item $T'=T'_i\oplus(\bigoplus_{j\neq i}T_j)$ and there is a triangle
\[T_i\xrightarrow{f} U_i\to T'_i\to T_i[1]\]
such that $f$ is a minimal left $(\add \bigoplus_{j\neq i}T_j)$-approximation.
\item $T'=T'_i\oplus(\bigoplus_{j\neq i}T_j)$ and there is a triangle 
\[T_i\to U_i\xrightarrow{g} T'_i\to T_i[1]\]
such that $g$ is a minimal right $(\add \bigoplus_{j\neq i}T_j)$-approximation.
\end{enumerate}
The triangles in (c) and (d) are isomorphic, and called an \emph{exchange triangle}.
\end{proposition}

To introduce the $g$-fan of a finite dimensional $k$-algebra $A$, 
we consider the real Grothendieck group of $A$:
\[K_0(\proj A)_\R:=K_0(\proj A)\otimes_{\mathbb{Z}}\R\simeq \R^{|A|}.\] 

\begin{definition}\label{g-fan}
For $T=T_1\oplus\cdots\oplus T_\ell\in\twopsilt A$ with indecomposable $T_i$, let
\begin{eqnarray*}
C(T):=\cone\{[T_1],\ldots,[T_\ell]\}\subset K_0(\proj A)_\R. 
\end{eqnarray*}
The \emph{$g$-fan} of $A$ is the set of cones:
\[\Sigma(A):=\{C(T)\mid T\in\twopsilt A\}.\] 
We say that $A$ is \emph{$g$-convex} if $\Sigma(A)$ is convex (see Definition \ref{from fan to polytope}).
\end{definition}

We say that $A$ is \emph{$g$-finite} (also \emph{$\tau$-tilting finite}) if $\#\twosilt A<\infty$. 
We give the following basic properties of $g$-fans.

\begin{proposition}\cite{AIR,As,DIJ,DF}\label{characterize g-finite}
Let $A$ be a finite dimensional algebra over a field $k$ and $n:=|A|$.
\begin{enumerate}[\rm(a)]
\item Any cone in $\Sigma(A)$ is a face of a cone of dimension $n$.
\item Any cone in $\Sigma(A)$ of dimension $n-1$ is a face of precisely two cones of dimension $n$. 
\item $(\Sigma(A),C(A))$ is a pairwise positive sign-coherent fan.
\item $A$ is $g$-finite (or equivalently, $\Sigma(A)$ is finite) if and only if $\Sigma(A)$ is complete.
\end{enumerate}
\end{proposition}

The following basic observation will be used frequently.

\begin{proposition}\label{+- condition}
Let $\Lambda$ be a finite dimensional algebra with orthogonal primitive idempotents $1=e_1+e_2$. Under the identification $P_1=(1,0)$ and $P_2=(0,1)$, the following assertions hold.
\begin{enumerate}[\rm(a)]
\item $\cone\{(-1,0),(0,1)\}\in\Sigma(\Lambda)$ if and only if $e_2\Lambda e_1=0$.
\item $\cone\{(1,0),(0,-1)\}\in\Sigma(\Lambda)$ if and only if $e_1\Lambda e_2=0$.
\end{enumerate}
\end{proposition}

Proposition \ref{+- condition} is explained by the following picture.
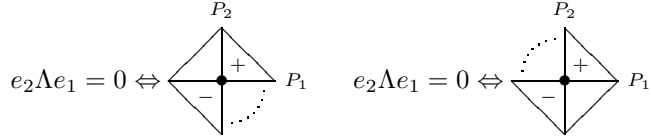
\[{e_2\Lambda e_1=0\Leftrightarrow\begin{xy}
0;<4pt,0pt>:<0pt,4pt>::
(0,-5)="0",
(-5,0)="1",
(0,0)*{\bullet},
(0,0)="2",
(5,0)="3",
(0,5)="4",
(0,-4)="a",
(4,0)="b",
(1.5,1.5)*{{\scriptstyle +}},
(-1.5,-1.5)*{{\scriptstyle -}},
(7,0)*{{\scriptstyle P_1}},
(0,6.7)*{{\scriptstyle P_2}}, 
\ar@{-}"0";"1",
\ar@{-}"1";"4",
\ar@{-}"4";"3",
\ar@{-}"2";"0",
\ar@{-}"2";"1",
\ar@{-}"2";"3",
\ar@{-}"2";"4",
\ar@/^-2mm/@{.} "a";"b",
\end{xy}}\ \ \ \ \  
{e_2\Lambda e_1=0\Leftrightarrow\begin{xy}
0;<4pt,0pt>:<0pt,4pt>::
(0,-5)="0",
(-5,0)="1",
(0,0)*{\bullet},
(0,0)="2",
(5,0)="3",
(0,5)="4",
(0,4)="a",
(-4,0)="b",
(1.5,1.5)*{{\scriptstyle +}},
(-1.5,-1.5)*{{\scriptstyle -}}, 
(7,0)*{{\scriptstyle P_1}},
(0,6.7)*{{\scriptstyle P_2}}, 
\ar@{-}"0";"1",
\ar@{-}"0";"3",
\ar@{-}"4";"3",
\ar@{-}"2";"0",
\ar@{-}"2";"1",
\ar@{-}"2";"3",
\ar@{-}"2";"4",
\ar@/^-2mm/@{.} "a";"b",
\end{xy}}\]

\begin{proof}
We only prove (a): $\Sigma(A)\in\tscfanpm$ if and only if $P_1[1]\oplus P_2\in\twosilt A$ if and only if $\Hom_{\Kb(\proj A)}(P_1,P_2)=0$ if and only if $e_2\Lambda e_1=0$.
\end{proof}

We end this subsection with recalling the sign decomposition technique studied in \cite{Ao,AHIKM}.  
We have to introduce the following notations. 

\begin{definition}\label{define A_epsilon}
Let $A$ be a basic finite dimensional algebra over a field $k$ with $|A|=n$, and $1=e_1+\cdots+e_n$ the orthogonal primitive idempotents. For $\epsilon\in\{\pm1\}^n$, 
we define 
\[K_0(\proj A)_{\epsilon,\R}:=\cone(\epsilon_i[e_iA]\mid i\in\{1,\ldots,n\})\]
and a subfan of $\Sigma(A)$  by
\[\Sigma_\epsilon(A):=\{\sigma\in\Sigma(A)\mid \sigma\subset K_0(\proj A)_{\epsilon,\R}\}.\]
Define idempotents of $A$ by
\[e^+_\epsilon:=\sum_{\epsilon_i=1}e_i\ \mbox{ and }\ e^-_\epsilon:=\sum_{\epsilon_i=-1}e_i.\]
We denote by $A_{\epsilon}$ the subalgebra of $A$ given by
\[A_\epsilon:=\left[\begin{array}{cc}e^+_\epsilon Ae^+_\epsilon&e^+_\epsilon Ae^-_\epsilon\\
0&e^-_\epsilon Ae^-_\epsilon\end{array}\right].\]
Define an ideal $I_\epsilon$ of $A_\epsilon$ by 
\[I_\epsilon:=\left[\begin{array}{cc}\rad (e^+_\epsilon Ae^+_\epsilon)\cap\Ann_{e^+_\epsilon Ae^+_\epsilon}(e^+_\epsilon Ae^-_\epsilon)&0\\
0&\rad(e^-_\epsilon Ae^-_\epsilon)\cap\Ann(e^+_\epsilon Ae^-_\epsilon)_{e^-_\epsilon Ae^-_\epsilon}\end{array}\right].\]
\end{definition}

The following result is often very useful to calculate $\Sigma_\epsilon(A)$.

\begin{proposition}\label{sign decomposition}\cite[Example 4.26]{AHIKM}
For each ideal $I$ of $A_\epsilon$ contained in $I_\epsilon$, the isomorphisms $-\otimes_{A_\epsilon}A:K_0(\proj A_\epsilon)_\R\simeq K_0(\proj A)_\R$ and $-\otimes_{A_\epsilon}(A_\epsilon/I):K_0(\proj A_\epsilon)_\R\simeq K_0(\proj A_\epsilon/I)_\R$ gives an isomorphism of fans
\[\Sigma_\epsilon(A)\simeq\Sigma_\epsilon(A_\epsilon/I).\]
\end{proposition}

\subsection{Silting complexes in terms of matrices}\label{section: Matrices and Presilting complexes}
In this subsection, we give basic properties of 2-term presilting complexes. Throughout this subsection, we assume the following.

\begin{assumption}\label{Lambda=AXB}
For rings $A$ and $B$ and an $A^{\rm op}\otimes_k B$-module $X$ which is finitely generated on both sides, let
\[\Lambda:=\left[\begin{array}{cc}
	A&X\\ 0&B
\end{array}\right].\]
Equivalently, $\Lambda$ is a ring with orthogonal idempotents $1=e_1+e_2$ satisfying $e_2\Lambda e_1=0$. In fact, we can recover $\Lambda$ from $A:=e_1\Lambda e_1$, $B:=e_2\Lambda e_2$ and $X:=e_1\Lambda e_2$ by the equality above.
\end{assumption}

Consider projective $\Lambda$-modules
\[P_1:=[A\ X],\ P_2:=[0\ B]\in\proj\Lambda.\]
For $s,t\ge0$, we denote by $M_{s,t}(X)$ the set of $s\times t$ matrices with entries in $X$. Then we have an isomorphism
\[M_{s,t}(X)\simeq\Hom_\Lambda(P_2^{\oplus t},P_1^{\oplus s})\]
sending $x\in M_{s,t}(X)$ to the left multiplication $x(\cdot):P_2^{\oplus t}\to P_1^{\oplus s}$. 
Thus we have a 2-term complex
\[P_x:=(P_2^{\oplus t}\xrightarrow{x(\cdot)} P_1^{\oplus s})\in\per\Lambda.\]
The following observation is basic.

\begin{proposition}\label{Ext 1}
	Let $s,t,u,v\ge0$, $x\in M_{s,t}(X)$ and $y\in M_{u,v}(X)$.
	\begin{enumerate}[\rm(a)]
	\item Then we have an exact sequence
	\[M_{u,s}(A)\oplus M_{v,t}(B)\xrightarrow{[(\cdot)x\ y(\cdot)]}M_{u,t}(X)\to\Hom_{\per\Lambda}(P_x,P_y[1])\to0.\]
	\item In particular, $P_x$ is presilting if and only if $M_{s,t}(X)=M_s(A)x+xM_t(B)$ holds.
	\end{enumerate}
\end{proposition}

\begin{proof}
	The assertion (a) follows from an exact sequence
	\[\Hom_\Lambda(P_1^{\oplus s},P_1^{\oplus u})\oplus\Hom_\Lambda(P_2^{\oplus t},P_2^{\oplus v})\xrightarrow{[(\cdot)x\ y(\cdot)]}\Hom_\Lambda(P_2^{\oplus t},P_1^{\oplus u})\to\Hom_{\per\Lambda}(P_x,P_y[1])\to0.\]
	The assertion (b) is immediate from (a).
\end{proof}

The following construction of silting complexes of $\Lambda$ will be used frequently, where $t(X)_B$ (respectively, $t_A( X)$) is the minimal number of generators of $X$ as a right $B$-module (respectively, left $A$-module).

\begin{proposition}\label{first mutation}
In Assumption \ref{Lambda=AXB}, assume that $A$ and $B$ are local $k$-algebras. 

\begin{enumerate}[\rm(a)]
\item $\Sigma(\Lambda)$ contains $\cone\{(0,-1),(1,-r)\}$ for $r:=t(X)_B=\dim(X/XJ_B)_{B/J_B}$.
\item $\Sigma(\Lambda)$ contains $\cone\{(1,0),(\ell,-1)\}$ for $\ell:=t_A (X)=\dim_{A/J_A}(X/J_AX)$.
\item Let $g_1,\ldots,g_r$ be a minimal set of generators of the $B$-module $X$. Then $\mu_1^+(\Lambda[1])=P_g\oplus P_2[1]\in\twosilt \Lambda$ holds for $g:=[g_1\ \cdots\ g_r]\in M_{1,r}(X)$.
\item Let $h_1,\ldots,h_\ell$ a minimal set of generators of the $A^{\op}$-module $X$. Then $\mu_2^-(\Lambda)=P_1\oplus P_h\in\twosilt \Lambda$ holds for $h:=\left[\begin{smallmatrix}h_1\\ \svdots\\ h_\ell\end{smallmatrix}\right]\in M_{\ell,1}(X)$. 
\end{enumerate}
\end{proposition}

By Propositions \ref{+- condition} and \ref{first mutation}, a part of $\Sigma(\Lambda)$ has the following form.
    \[\Sigma(\Lambda) =
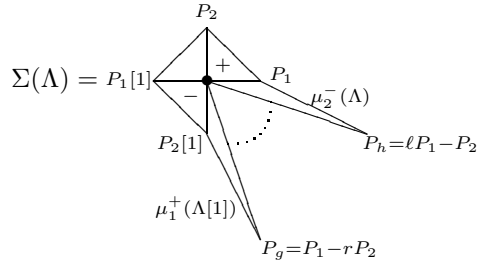
\begin{xy}
				0;<4pt,0pt>:<0pt,4pt>::
				(0,-5)="0",
				(-5,0)="1",
				(0,0)*{\bullet},
				(0,0)="2",
				(5,0)="3",
				(0,5)="4",
				(15,-5)="5",
				(5,-15)="6",
				(2,-6)="a",
				(6,-2)="b",
		(0,6.5)*{{\scriptstyle P_2}},
		(7,0.5)*{{\scriptstyle P_1}},
		(20,-6)*{{\scriptstyle P_h=\ell P_1-P_2}},
		(10.5,-16)*{{\scriptstyle P_g=P_1-rP_2}},
		(-2.5,-6)*{{\scriptstyle P_2[1]}},
		(-7.3,0)*{{\scriptstyle P_1[1]}},
				(1.5,1.5)*{{\scriptstyle +}},
				(-1.5,-1.5)*{{\scriptstyle -}},
				(12.5,-1.5)*{{\scriptstyle\mu_2^-(\Lambda)}},
				(-1,-12)*{{\scriptstyle\mu_1^+(\Lambda[1])}},
				(7,1)*{},
				(0,6.7)*{},
				\ar@{-}"0";"1",
				\ar@{-}"1";"4",
				\ar@{-}"3";"5",
				\ar@{-}"3";"4",
				\ar@{-}"2";"0",
				\ar@{-}"2";"1",
				\ar@{-}"2";"3",
				\ar@{-}"2";"4",
				\ar@{-}"2";"5",
				\ar@{-}"2";"6",
				\ar@{-}"0";"6",
				\ar@/^-2mm/@{.} "a";"b",
		\end{xy}\]
	
\begin{proof}
We only prove (a)(c) since (b)(d) are the duals. 
A minimal right $(\add P_2[1])$-approximation of $P_1[1]$ is given by
\[g(\cdot):P_2[1]^{\oplus r}\to P_1[1].\]
Thus the mutation of $\Lambda[1]$ at $P_1[1]$ is $P_g\oplus P_2$.
\end{proof}

Now we assume that $B$ is a local algebra. We fix a minimal set of generators $g_1,\ldots,g_r$ of the right $B$-module $X$ and set 
\[g:=[g_1\ \cdots\ g_r]\in M_{1,r}(X)\ \mbox{ and }\ 
\overline{g}:=[\overline{g_1}\ \cdots\ \overline{g_r}]\in M_{1,r}(X/XJ_B),	
\]
where $\overline{(\cdot)}$ is a canonical surjection $X \twoheadrightarrow X/X J_B$.
	Then we have an isomorphism
	\[
	\overline{g}(\cdot):M_{r,1}(B/J_B)\simeq X/XJ_B,
	\] 
	and we define a map $\pi:X\rightarrow M_{r,1}(B/J_B)$ by
	\[
	\pi:=(X\xrightarrow{\overline{(\cdot)}} X/XJ_B \xrightarrow{(\overline{g}(\cdot))^{-1}}M_{r,1}(B/J_B)).
	\]
For each $s,t\ge0$, an entry-wise application of $\pi$ gives a map
\[\pi:M_{s,t}(X)\to M_{s,t}(M_{r,1}(B/J_B))=M_{rs,t}(B/J_B).
\]
In other words, for the identity matrix $I_s \in M_s(k)$ and $\overline{g}I_s:=\left[\begin{smallmatrix}
\overline{g}   & &O \\
 &  \ddots	& \\
 O&   & \overline{g}\\
\end{smallmatrix}\right]
\in M_s(M_{1,r}(k))=M_{s,rs}(k)$, we have
\begin{equation}\label{define pi ab}
\overline{x}=(\overline{g}I_s)\pi(x)\ \mbox{ for each }\ x\in M_{s,t}(X).
\end{equation}
Define a morphism of $k$-algebras
\[\phi:M_s(A)\to M_{rs}(B/J_B)\ \mbox{ by }\ a(\overline{g}I_s)=(\overline{g}I_s)\phi(a).\]
Later we will use the following observation.

\begin{proposition}\label{proposition:full rank}
In Assumption \ref{Lambda=AXB}, assume that $B$ is a local algebra. Let $s,t\ge0$.
\begin{enumerate}[\rm(a)]	
\item 
$\pi:M_{s,t}(X)\to M_{rs,t}(B/J_B)$ is a morphism of $M_s(A)^{\op}\otimes_kM_t(B)$-modules, where we regard $M_{rs,t}(B/J_B)$ as an $M_s(A)^{\op}$-module via $\phi$.
\item Let $x\in M_{s,t}(X)$. If $P_x$ is presilting, then $\pi(x)\in M_{rs,t}(B/J_B)$ has full rank.
\end{enumerate}
\end{proposition}

\begin{proof}
(a) For any $a\in M_s(A)$, $x\in M_{s,t}(X)$ and $b\in M_t(B)$, we need to show $\pi(axb)=\phi(a)\pi(x)b$. 
In fact, 
	 \[(\overline{g}I_s)\phi(a)\pi(x)b=a(\overline{g}I_s)\pi(x)b\stackrel{\eqref{define pi ab}}{=}a\overline{x}b=\overline{axb}\stackrel{\eqref{define pi ab}}{=}(\overline{g}I_s)\pi(axb)\]
  gives the desired equality since $\overline{g}I_s(\cdot)$ is injective.

(b) By Proposition \ref{Ext 1}(b), we have $M_{s,t}(X)=M_s(A)x+xM_t(B)$. Applying $\pi$, we have
\begin{eqnarray*}M_{rs,t}(B/J_B)=\pi(M_s(A)x+xM_t(B))&\stackrel{{\rm(a)}}{=}&\phi(M_s(A))\pi(x)+\pi(x)M_t(B)\\
		&\subset& M_{rs}(B/J_B)\pi(x)+\pi(x)M_t(B/J_B).
\end{eqnarray*}
		Thus the right-hand side is $M_{rs,t}(B/J_B)$. This clearly implies that $\pi(x)$ has full rank.
\end{proof}

For completeness, we also give the dual statement of Proposition \ref{proposition:full rank}. 
Now we assume that $A$ is a local algebra. We fix a minimal set of generators $h_1,\ldots,h_\ell$ of the left $A$-module $X$ and set 
\[
h:=\left[\begin{smallmatrix}h_1\\ \svdots\\ h_\ell\end{smallmatrix}\right]\in M_{\ell,1}(X)\ \mbox{ and }\ \overline{h}:=\left[\begin{smallmatrix}\overline{h_1}\\ \svdots\\ \overline{h_\ell}\end{smallmatrix}\right]\in M_{\ell,1}(X/J_AX),
\]
where by abuse of notations, $\overline{(\cdot)}$ is a canonical surjection $X \twoheadrightarrow X/J_A X$. 
Then we have an isomorphism $(\cdot)\overline{h}:M_{1,\ell}(A/J_A)\simeq X/J_AX$. By abuse of notations, let
\[\pi:=(X\xrightarrow{(\cdot)} X/J_AX\xrightarrow{((\cdot )\overline{h})^{-1}}M_{1,\ell}(A/J_A)).\]
For each $s,t\ge0$, an entry-wise application of $\pi$ gives a map\[\pi:M_{s,t}(X)\to M_{s,t}(M_{1,\ell}(A/J_A))=M_{s,\ell t}(A/J_A).\]
Define a morphism of $k$-algebras
\[\phi:M_t(B)\to M_{\ell t}(A/J_A)\ \mbox{ by }\ (\overline{h}I_t)b=\phi(b)(\overline{h}I_s).\]
We have the following dual of Proposition \ref{proposition:full rank}.

\begin{proposition}\label{proposition:full rank 2}
In Assumption \ref{Lambda=AXB}, assume that $A$ is a local algebra. Let $s,t\ge0$.
\begin{enumerate}[\rm(a)]	
\item $\pi:M_{s,t}(X)\to M_{s,\ell t}(A/J_A)$ is a morphism of $M_s(A)^{\op}\otimes_kM_t(B)$-modules, where we regard $M_{s,\ell t}(A/J_A)$ as an $M_t(B)$-module via $\phi$.
\item Let $x\in M_{s,t}(X)$. If $P_x$ is presilting, then $\pi(x)\in M_{s,\ell t}(A/J_A)$ has full rank.
\end{enumerate}
\end{proposition}

\subsection{Uniserial property of $g$-finite algebras}

As an application of results in the previous subsection, we prove the following result, which is not used in the rest of this paper.

\begin{theorem}\label{uniserial 0}
Let $\Lambda$ be a finite dimensional elementary $k$-algebra, and $1=e_1+\cdots+e_n$ the orthogonal primitive idempotents.
If $\Lambda$ is $g$-finite, then for each $1\le i\neq j\le n$, $e_i\Lambda e_j/e_i\Lambda e_jJ_\Lambda e_j$ is a uniserial $(e_i\Lambda e_i)^{\op}$-module and $e_i\Lambda e_j/e_iJ_\Lambda e_iJ_\Lambda e_j$ is a uniserial $e_j\Lambda e_j$-module.
\end{theorem}

Thanks to sign decomposition, we can deduce Theorem \ref{uniserial 0} from the following result.  

\begin{theorem}\label{uniserial}
Let $A$ and $B$ be local $k$-algebras with $k\simeq A/J_A\simeq B/J_B$. If $X$ is a $A^{\op}\otimes_kB$-module such that $\left[\begin{array}{cc}
A&X\\ 0&B
\end{array}\right]$ is $g$-finite, then $X/XJ_B$ is a uniserial $A^{\op}$-module and $X/J_AX$ is a uniserial $B$-module.
\end{theorem}

\begin{proof}[Proof of Theorem \ref{uniserial}$\Rightarrow$Theorem \ref{uniserial 0}]
Since $\Lambda$ is $g$-finite, so is $\Gamma:=(e_i+e_j)\Lambda(e_i+e_j)$. By Proposition \ref{sign decomposition}, $\Gamma_{+-}=\left[\begin{array}{cc}
e_i\Lambda e_i&e_i\Lambda e_j\\ 0&e_j\Lambda e_j
\end{array}\right]$ is also $g$-finite. Thus the assertion follows from Theorem \ref{uniserial}.
\end{proof}

In the rest of this subsection, we prove Theorem \ref{uniserial}.

The following observation plays a key role in the proof, where we identify $K_0(\proj \Lambda)$ with $\Z^2$ via $[A\ X]\mapsto(1,0)$, $[0\ B]\mapsto(0,1)$.

\begin{lemma}\label{(-1,b)}
Let $\Lambda:=\left[\begin{array}{cc}
A&X\\ 0&k
\end{array}\right]$. Assume that $(1,-1)\in K_0(\proj\Lambda)$ is rigid.
\begin{enumerate}[\rm(a)]
\item There exists $h\in X$ such that $X=Ah$.
\item Let $\Lambda':=\left[\begin{array}{cc}
A&J_AX\\ 0&k
\end{array}\right]$ and $t\ge1$. If $(1,-t)\in K_0(\proj\Lambda)$ is rigid, then $(1,1-t)\in K_0(\proj\Lambda')$ is rigid.
\end{enumerate}
\end{lemma}

\begin{proof}
(a) By Proposition \ref{Ext 1}(b), there exists $h\in X$ satisfying $X=Ah+hk=Ah$.

(b) By Proposition \ref{Ext 1}(b), there exists $[x_1\ x_2\ \cdots\ x_t]\in M_{1,t}(X)$ such that
\begin{equation}\label{1,b equality}
M_{1,t}(X)=A[x_1\ \cdots\ x_t]+[x_1\ \cdots\ x_t]M_t(k).
\end{equation}
As in Section \ref{section: Matrices and Presilting complexes}, the element $h$ gives surjections
	\[
	\pi:=(X\xrightarrow{\overline{(\cdot)}} X/J_AX\xrightarrow{((\cdot)\overline{h})^{-1}}A/J_A=k)\ \mbox{ and }\ \pi:M_{1,t}(X)\to M_{1,t}(k).
	\]
By Proposition \ref{proposition:full rank 2}, $\pi(x)\in M_{1,t}(k)$ has full rank. By changing indices if necessary, we can assume $x_1\in A^\times h$. Multiplying an element in $A^\times$ from left, we can assume $x_1=h$.
Multiplying an element in $\GL_t(k)$ from right, we can assume $x_i\in J_Ah$ for each $2\le i\le t$. We claim 
\[M_{1,t-1}(J_AX)=A[x_2\ \cdots\ x_t]+[x_2\ \cdots\ x_t]M_{t-1}(k).\]
In fact, fix any $[y_2\ \cdots y_t]\in M_{1,t-1}(J_AX)$. By \eqref{1,b equality} there exist $a\in A$ and $b=[b_{ij}]_{1\le i,j\le t}\in M_t(k)$ such that
\begin{equation}\label{y=ax+xb}
[0\ y_2\ \cdots y_t]=a[h\ x_2\ \cdots x_t]+[h\ x_2\ \cdots x_t]b.
\end{equation}
Applying $\pi$, we obtain
\[[0\ 0\ \cdots 0]=\overline{a}[1\ 0\ \cdots 0]+[1\ 0\ \cdots 0]b\ \mbox{ in }\ M_{1,t}(k).\]
Thus we obtain $b_{12}=\cdots=b_{1t}=0$. Looking at the $i$-th entries for $2\le i\le t$ of \eqref{y=ax+xb}, we have
\[[y_2\ \cdots y_t]=a[x_2\ \cdots x_t]+[x_2\ \cdots x_t][b_{ij}]_{2\le i,j\le n}.\]
Thus the claim follows.
\end{proof}

We are ready to prove Theorem \ref{uniserial}.

\begin{proof}[Proof of Theorem \ref{uniserial}]
We prove that $X/XJ_B$ is a uniserial $A^{\op}$-module under a weaker assumption that $(1,-t)\in K_0(\proj\Lambda)$ is rigid for each $t\ge1$.
Since $\overline{\Lambda}:=\left[\begin{array}{cc}
A&X/XJ_B\\ 0&k
\end{array}\right]$ is a factor algebra of $\Lambda$, the element $(1,-t)\in K_0(\proj\overline{\Lambda})$ is rigid for each $t\ge1$. Replacing $\Lambda$ by $\overline{\Lambda}$, we can assume that
\[B=k\ \mbox{ and }\ \Lambda=\left[\begin{array}{cc}
A&X\\ 0&k
\end{array}\right].\]
We use induction on $\dim_kX$. By Lemma \ref{(-1,b)}(a), the $A^{\op}$-module $X$ has a unique maximal submodule $J_AX$. Let $\Lambda'=\left[\begin{array}{cc}
A&J_AX\\ 0&k
\end{array}\right]$. By Lemma \ref{(-1,b)}(b), $(1,-t)\in K_0(\proj\Lambda')$ is rigid for each $t\ge1$. By induction hypothesis, $J_AX$ is a uniserial $A^{\op}$-module. Therefore $X$ is also a uniserial $A^{\op}$-module.
\end{proof}

\section{Gluing, Rotation and Subdivision of fans}

\subsection{Gluing fans}
We start with introducing the following operation for fans of rank 2.

\begin{definition}
For $\Sigma \in \itscfanpm$ and $\Sigma' \in \itscfanmp$, we define $\Sigma\ast\Sigma'\in\itscfan$ by
\begin{eqnarray*}
(\Sigma\ast\Sigma')_1&:=&\Sigma_1\cup \Sigma'_1\\
(\Sigma\ast\Sigma')_2&:=&\left(\Sigma_2 \setminus\{\sigma_{-+}\} \right)\cup\left(\Sigma'_2 \setminus \{\sigma_{+-}\} \right).
\end{eqnarray*}
\begin{equation*}
{\Sigma=\begin{xy}
0;<4pt,0pt>:<0pt,4pt>::
(0,-5)="0",
(-5,0)="1",
(0,0)*{\bullet},
(0,0)="2",
(5,0)="3",
(0,5)="4",
(0,-4)="a",
(4,0)="b",
(1.5,1.5)*{{\scriptstyle +}},
(-1.5,-1.5)*{{\scriptstyle -}},
(4,-4)*{{\scriptstyle ?}},
(-4.5,4)*{{\scriptstyle \sigma_{-+}}},
\ar@{-}"0";"1",
\ar@{-}"1";"4",
\ar@{-}"4";"3",
\ar@{-}"2";"0",
\ar@{-}"2";"1",
\ar@{-}"2";"3",
\ar@{-}"2";"4",
\ar@/^-2mm/@{.} "a";"b",
\end{xy}}\ \ \ \ \  
{\Sigma'=\begin{xy}
0;<4pt,0pt>:<0pt,4pt>::
(0,-5)="0",
(-5,0)="1",
(0,0)*{\bullet},
(0,0)="2",
(5,0)="3",
(0,5)="4",
(0,4)="c",
(-4,0)="d",
(1.5,1.5)*{{\scriptstyle +}},
(-1.5,-1.5)*{{\scriptstyle -}},
(-4,4)*{{\scriptstyle !}},
(4.5,-4)*{{\scriptstyle \sigma_{+-}}}, 
\ar@{-}"0";"1",
\ar@{-}"0";"3",
\ar@{-}"4";"3",
\ar@{-}"2";"0",
\ar@{-}"2";"1",
\ar@{-}"2";"3",
\ar@{-}"2";"4",
\ar@/^-2mm/@{.} "c";"d",
\end{xy}}\ \ \ \ \ 
{\Sigma\ast\Sigma'=\begin{xy}
0;<4pt,0pt>:<0pt,4pt>::
(0,-5)="0",
(-5,0)="1",
(0,0)*{\bullet},
(0,0)="2",
(5,0)="3",
(0,5)="4",
(0,-4)="a",
(4,0)="b",
(0,4)="c",
(-4,0)="d",
(1.5,1.5)*{{\scriptstyle +}},
(-1.5,-1.5)*{{\scriptstyle -}},
(4,-4)*{{\scriptstyle ?}},
(-4,4)*{{\scriptstyle !}}, 
\ar@{-}"0";"1",
\ar@{-}"4";"3",
\ar@{-}"2";"0",
\ar@{-}"2";"1",
\ar@{-}"2";"3",
\ar@{-}"2";"4",
\ar@/^-2mm/@{.} "a";"b",
\ar@/^-2mm/@{.} "c";"d",
\end{xy}}
\end{equation*}
\end{definition}

Clearly we have 
\begin{eqnarray} \notag
        \itscfan &=& \itscfanpm\ast \itscfanmp:=\{\Sigma\ast\Sigma'\mid \Sigma\in \itscfanpm,\Sigma'\in \itscfanmp\},\\ \label{eq:gluing fans}
        \tscfan &=& \tscfanpm\ast \tscfanmp:=\{\Sigma\ast\Sigma'\mid \Sigma\in \tscfanpm,\Sigma'\in \tscfanmp\}.
\end{eqnarray}

Now we study an algebraic counterpart of the gluing fans. 
The following is a main result of this subsection, where we identify $K_0(\proj\Lambda)$ and $K_0(\proj\Lambda')$ with $\Z^2$ by $e_1\Lambda=(1,0)=e_1'\Lambda'$ and $e_2\Lambda=(0,1)=e_2'\Lambda'$.

\begin{theorem}[{Gluing Theorem}]\label{theorem:gluing1}
Let $\Lambda$ and $\Lambda'$ be elementary $k$-algebras of rank 2 with orthogonal primitive idempotents $1=e_1+e_2\in\Lambda$ and $1=e'_1+e'_2\in\Lambda'$. Assume $e_1\Lambda e_2=0$ and $e'_2\Lambda' e'_1=0$, or equivalently, $\Sigma(\Lambda)\in\itscfanpm$ and $\Sigma(\Lambda')\in\itscfanmp$ (Proposition \ref{+- condition}).
Then, there exists an elementary $k$-algebra $\Lambda\ast\Lambda'$ such that 
\begin{equation}
    \Sigma(\Lambda\ast\Lambda')= \Sigma(\Lambda)\ast\Sigma(\Lambda').
\end{equation}
\end{theorem}

Theorem \ref{theorem:gluing1} is explained by the following picture.
\[{\Sigma(\Lambda)=\begin{xy}
0;<4pt,0pt>:<0pt,4pt>::
(0,-5)="0",
(-5,0)="1",
(0,0)*{\bullet},
(0,0)="2",
(5,0)="3",
(0,5)="4",
(0,-4)="a",
(4,0)="b",
(1.5,1.5)*{{\scriptstyle +}},
(-1.5,-1.5)*{{\scriptstyle -}},
(4,-4)*{{\scriptstyle ?}}, 
\ar@{-}"0";"1",
\ar@{-}"1";"4",
\ar@{-}"4";"3",
\ar@{-}"2";"0",
\ar@{-}"2";"1",
\ar@{-}"2";"3",
\ar@{-}"2";"4",
\ar@/^-2mm/@{.} "a";"b",
\end{xy}}\ \ \ \ \  
{\Sigma(\Lambda')=\begin{xy}
0;<4pt,0pt>:<0pt,4pt>::
(0,-5)="0",
(-5,0)="1",
(0,0)*{\bullet},
(0,0)="2",
(5,0)="3",
(0,5)="4",
(0,4)="c",
(-4,0)="d",
(1.5,1.5)*{{\scriptstyle +}},
(-1.5,-1.5)*{{\scriptstyle -}},
(-4,4)*{{\scriptstyle !}}, 
\ar@{-}"0";"1",
\ar@{-}"0";"3",
\ar@{-}"4";"3",
\ar@{-}"2";"0",
\ar@{-}"2";"1",
\ar@{-}"2";"3",
\ar@{-}"2";"4",
\ar@/^-2mm/@{.} "c";"d",
\end{xy}}\ \ \ \ \ 
{\Sigma(\Lambda\ast\Lambda')=\begin{xy}
0;<4pt,0pt>:<0pt,4pt>::
(0,-5)="0",
(-5,0)="1",
(0,0)*{\bullet},
(0,0)="2",
(5,0)="3",
(0,5)="4",
(0,-4)="a",
(4,0)="b",
(0,4)="c",
(-4,0)="d",
(1.5,1.5)*{{\scriptstyle +}},
(-1.5,-1.5)*{{\scriptstyle -}},
(4,-4)*{{\scriptstyle ?}},
(-4,4)*{{\scriptstyle !}}, 
\ar@{-}"0";"1",
\ar@{-}"4";"3",
\ar@{-}"2";"0",
\ar@{-}"2";"1",
\ar@{-}"2";"3",
\ar@{-}"2";"4",
\ar@/^-2mm/@{.} "a";"b",
\ar@/^-2mm/@{.} "c";"d",
\end{xy}}\]

The construction of $\Lambda\ast\Lambda'$ is as follows: 
We can write 
\[
\Lambda=\left[\begin{array}{cc}
A&X\\ 0&B
\end{array}\right]\ 
\mbox{ and }\ \Lambda'=\left[\begin{array}{cc}
C&0\\ Y&D
\end{array}\right],
\]
where $A,B,C,D$ are local $k$-algebras, $X$ is an $A^{\rm op}\otimes_k B$-module, and $Y$ is an $D^{\rm op}\otimes_k C$-module.
Since $\Lambda$ and $\Lambda'$ are elementary, we have  $k\simeq A/J_A\simeq B/J_B \simeq C/J_C\simeq D/J_D$. 
Let $A\times_kC$ be a fiber product of canonical surjections $\overline{(\cdot)}:A\to k$ and $\overline{(\cdot)}:C\to k$, that is,
\[A\times_kC:=\{(a,c)\in A\times C\mid \overline{a}=\overline{c}\}.\]
Let $B\times_kD$ be a fibre product of $\overline{(\cdot)}:B\to k$ and $\overline{(\cdot)}:D\to k$. Using the projections $A\times_kC\to A$ and $B\times_kD\to B$, we regard $X$ as an $(A\times_kC)^{\op}\otimes_k(B\times_kD)$-module, and using the projections $A\times_kC\to C$ and $B\times_kD\to D$, we regard $Y$ as an $(B\times_kD)^{\op}\otimes_k(A\times_kC)$-module.

We prove that the algebra
\[\Lambda\ast\Lambda':=\left[\begin{array}{cc}
A\times_kC&X\\ Y&B\times_kD
\end{array}\right]\]
satisfies $\Sigma(\Lambda\ast\Lambda')= \Sigma(\Lambda)\ast\Sigma(\Lambda')$, where the multiplication of the elements of  $X$ and those of $Y$ are defined to be zero.

\begin{proof}[Proof of Theorem \ref{theorem:gluing1}]
Let $\Gamma:=\Lambda\ast\Lambda'$. It suffices to prove
\[\Sigma_{+-}(\Gamma)= \Sigma_{+-}(\Lambda)\ \mbox{ and }\ \Sigma_{-+}(\Gamma)= \Sigma_{-+}(\Lambda').\]
For $\epsilon=(+,-)$, we have $\Gamma_\epsilon=\left[\begin{array}{cc}
A\times_kC&X\\ 0&B\times_kD
\end{array}\right]$. The ideal $I:=\left[\begin{array}{cc}
\rad C&0\\ 0&\rad D
\end{array}\right]$ of $\Gamma_\epsilon$ is contained in $I_\epsilon$, and we have an isomorphism $\Gamma_\epsilon/I\simeq \Lambda$ of $k$-algebras. Applying Proposition \ref{sign decomposition} to $\Gamma$, we get $\Sigma_{+-}(\Gamma)= \Sigma_{+-}(\Lambda)$. 
By the same argument, $\Sigma_{-+}(\Gamma)= \Sigma_{-+}(\Lambda')$ holds, as desired.
\end{proof}

\begin{example}
Let $\Lambda = k[1 \xrightarrow{a} 2]$ and $\Lambda'=k[1\xleftarrow{b}2]$. Then these $g$-fans are $\Sigma_{111;00}$ and $\Sigma_{00;111}$ respectively.
\[
\Sigma(\Lambda) =\Sigma_{111;00}=\begin{xy}
				0;<4pt,0pt>:<0pt,4pt>::
				(0,-5)="0",
				(-5,0)="1",
				(0,0)*{\bullet},
				(0,0)="2",
				(5,0)="3",
				(0,5)="4",
				(5,-5)="5",
				(1.5,1.5)*{{\scriptstyle +}},
				(-1.5,-1.5)*{{\scriptstyle -}},
				\ar@{-}"0";"1",
				\ar@{-}"1";"4",
				\ar@{-}"3";"5",
                \ar@{-}"5";"0",
				\ar@{-}"3";"4",
				\ar@{-}"2";"0",
				\ar@{-}"2";"1",
				\ar@{-}"2";"3",
				\ar@{-}"2";"4",
				\ar@{-}"2";"5",
		\end{xy}\ \ \ \ \
  \Sigma(\Lambda') =\Sigma_{00;111}=\begin{xy}
				0;<4pt,0pt>:<0pt,4pt>::
				(0,-5)="0",
				(-5,0)="1",  
				(0,0)*{\bullet},
				(0,0)="2",
				(5,0)="3",
				(0,5)="4",
				(-5,5)="5",
				(1.5,1.5)*{{\scriptstyle +}},
				(-1.5,-1.5)*{{\scriptstyle -}},
				\ar@{-}"0";"1",
				\ar@{-}"1";"5",
                \ar@{-}"5";"4",
				\ar@{-}"3";"0",
				\ar@{-}"3";"4",
				\ar@{-}"2";"0",
				\ar@{-}"2";"1",
				\ar@{-}"2";"3",
				\ar@{-}"2";"4",
				\ar@{-}"2";"5",
		\end{xy}  
\]
Applying Gluing Theorem\;\ref{theorem:gluing1}, the following algebra $\Gamma$ has $\Sigma_{111;111}=\Sigma_{111;00}\ast\Sigma_{00;111}$ as its $g$-fan.
\[\Gamma = \dfrac{k\left[\begin{xy}( 0,0) *+{1}="1", ( 8,0) *+{2}="2",\ar@<1.5pt> "1";"2"^a \ar@<1.5pt> "2";"1"^b  \end{xy}\right]}{\langle ab, ba\rangle}\ \ \ \ \ \Sigma(\Gamma)=\Sigma_{111;111} =
\begin{xy}
				0;<4pt,0pt>:<0pt,4pt>::
				(0,-5)="0",
				(-5,0)="1",
				(0,0)*{\bullet},
				(0,0)="2",
				(5,0)="3",
				(0,5)="4",
				(5,-5)="5",
                (-5,5)="6",
				(1.5,1.5)*{{\scriptstyle +}},
				(-1.5,-1.5)*{{\scriptstyle -}},
				\ar@{-}"0";"1",
				\ar@{-}"1";"6",
                \ar@{-}"6";"4",
				\ar@{-}"3";"5",
                \ar@{-}"5";"0",
				\ar@{-}"3";"4",
				\ar@{-}"2";"0",
				\ar@{-}"2";"1",
				\ar@{-}"2";"3",
				\ar@{-}"2";"4",
				\ar@{-}"2";"5",
                \ar@{-}"2";"6",
		\end{xy}
\]
In fact, we have $\Gamma\cong \Lambda\ast \Lambda'$. 
\end{example}

\subsection{Rotation and Mutation}
We start with introducing a piecewise linear transformation of sign coherent fan of rank 2. This is a generalization of mutation of $g$-vectors of cluster algebras of rank 2 \cite{FZ4,NZ}, and also a special case of so called \emph{combinatorial mutation} \cite{ACGK,FH}.

\begin{definition}\label{define rotation}
For $\Sigma\in\itscfanpm$ with $\sigma_+=\cone\{(0,1),(1,0)\}$, take $\sigma=\cone\{(1,0),(\ell,-1)\}\in\Sigma_2$. 
Define a new sign-coherent fan $\Sigma'$ by
\begin{eqnarray*}
\Sigma'_1&:=&\left(\Sigma_1\setminus\{(0,1)\}\right)\cup\{(-\ell,1)\}\\
\Sigma'_2&:=&\left(\Sigma_2
    \setminus \{\sigma_+,\sigma_{-+}\}\right)\cup\{-\sigma,\cone\{(-\ell,1),(1,0)\}\},
\end{eqnarray*}
    where the positive and negative cones of $\Sigma'$ are $\sigma$ and $-\sigma$ respectively. 
    \[
    {\Sigma =\begin{xy}
				0;<4pt,0pt>:<0pt,4pt>::
				(0,-5)="0",
				(-5,0)="1",
				(0,0)*{\bullet},
				(0,0)="2",
				(5,0)="3",
				(0,5)="4",
				(15,-5)="5",
				(0,-4)="a",
				(4,-1.33)="b",
		(0,6.5)*{{\scriptstyle (0,1)}},
		(8,0.5)*{{\scriptstyle (1,0)}},
		(15.5,-6.5)*{{\scriptstyle (\ell,-1)}},
		(0,-6.5)*{{\scriptstyle (0,-1)}},
		(-8.5,0)*{{\scriptstyle (-1,0)}},
				(1.5,1.5)*{{\scriptstyle +}},
				(-1.5,-1.5)*{{\scriptstyle -}},
				(11.5,-2)*{{\scriptstyle\sigma}},
				(-4,4)*{{\scriptstyle \sigma_{-+}}},
				(7,1)*{},
				(0,6.7)*{},
				\ar@{-}"0";"1",
				\ar@{-}"1";"4",
				\ar@{-}"3";"5",
				\ar@{-}"3";"4",
				\ar@{-}"2";"0",
				\ar@{-}"2";"1",
				\ar@{-}"2";"3",
				\ar@{-}"2";"4",
				\ar@{-}"2";"5",		
				\ar@/^-2mm/@{.} "a";"b",
		\end{xy}}\ \ \ \ \  
		{\rho(\Sigma)\simeq\Sigma' =\begin{xy}
				0;<4pt,0pt>:<0pt,4pt>::
				(0,-5)="0",
				(-5,0)="1",
				(0,0)*{\bullet},
				(0,0)="2",
				(5,0)="3",
				(0,5)="4",
				(15,-5)="5",
				(-15,5)="6",
				(0,-4)="a",
				(4,-1.33)="b",
		(-15,6.5)*{{\scriptstyle (-\ell,1)}},
		(8,0.5)*{{\scriptstyle (1,0)}},
		(15.5,-6.5)*{{\scriptstyle (\ell,-1)}},
		(0,-6.5)*{{\scriptstyle (0,-1)}},
		(-8.5,-0.5)*{{\scriptstyle (-1,0)}},
				(5,-0.85)*{{\scriptstyle +}},
				(-5,0.85)*{{\scriptstyle -}},
				(11.5,-2)*{{\scriptstyle\sigma}},
				(-11.5,2)*{{\scriptstyle-\sigma}},
				(7,1)*{},
				(0,6.7)*{},
				\ar@{-}"0";"1",
				\ar@{-}"3";"5",
				\ar@{-}"2";"0",
				\ar@{-}"2";"1",
				\ar@{-}"2";"3",
				\ar@{-}"2";"5",
				\ar@{-}"6";"1",
				\ar@{-}"6";"2",
				\ar@{-}"6";"3",
				\ar@/^-2mm/@{.} "a";"b",
			\end{xy}}
			\]    
			We define the \emph{rotation} $\rho(\Sigma)\in\itscfanpm$ of $\Sigma$ as the image of $\Sigma'$ by a linear transformation of $\mathbb{R}^2$ mapping $(1,0)\mapsto (0,1)$ and $(\ell,-1)\mapsto (1,0)$.
\end{definition}

We give basic properties of rotation, where the name ``rotation'' is explained by (a) below.

\begin{proposition}\label{sr sequence}
Let $\Sigma\in\tscfanpm$ with facets \eqref{facet 2} and $\df(\Sigma)=(a_1,\ldots,a_{n-2};0,0)$.
           \begin{enumerate}[\rm(a)]
    \item We have
    \begin{equation*}
        \df(\rho(\Sigma))=(a_2,\ldots,a_{n-2},a_1;0,0).
    \end{equation*}
    In particular, $\rho^{n-2}(\Sigma)=\Sigma$ holds, and therefore $\rho$ is an invertible operation.
    \item For each $1\le i\le n-3$, we have
    \[D_{\sigma_i}(\Sigma) =  \rho^{n-3-i} \circ D_{\sigma_{n-3}}\circ \rho^{i+1}(\Sigma).\]
    \end{enumerate}
\end{proposition}

\begin{proof}
(a) Recall $\Sigma_1=\{v_1,\ldots,v_n\}$ and $a_iv_i=v_{i-1}+v_{i+1}$ for $1\le i\le n$. Moreover
\[\rho(\Sigma)_1=\{w_1,\ldots,w_n\}\ \mbox{ where }\ w_i:=v_{i+1}\ (i\neq n-1),\ w_{n-1}:=-v_{2}.\]
Hence we have
\begin{eqnarray*}
w_{i-1}+w_{i+1}&=&v_i+v_{i+2}=a_{i+1}v_{i+1}=a_{i+1}w_i\ \mbox{ for }\ i\neq n-2,n,\\ 
w_{n-1}+w_1&=&-v_2+v_2=0\cdot w_n,\\
w_{n-3}+w_{n-1}&=&v_{n-2}-v_2=-(v_n+v_2)=-a_1v_1=a_1v_{n-1}=a_1w_{n-2}.
\end{eqnarray*}
Thus $\df(\rho(\Sigma))=(a_2,\ldots,a_{n-2},a_1;0,0)$ as desired. 

(b) By (a), we have $\df\circ\rho^{i+1}(\Sigma)=(a_{i+2},\ldots,a_{n-2},a_1,\ldots,a_{i+1};0,0)$. Thus
\[\df\circ D_{\sigma_{n-3}}\circ\rho^{i+1}(\Sigma)\stackrel{\eqref{s D_j=D_j s}}{=}D_{n-3}\circ\df\circ\rho^{i+1}(\Sigma)=(a_{i+2},\ldots,a_{n-2},a_1,\ldots,a_{i-1},a_{i}+1,1,a_{i+1}+1;0,0).\]
By (a) again, we have 
\begin{eqnarray*}
\df\circ\rho^{n-3-i}\circ D_{\sigma_{n-3}}\circ\rho^{i+1}(\Sigma)&=&(a_1,\ldots,a_{i-1}, a_{i}+1,1,a_{i+1}+1,a_{i+2},\ldots,a_{n-2};0,0)\\
&=&D_i\circ\df(\Sigma)\stackrel{\eqref{s D_j=D_j s}}{=}\df\circ D_{\sigma_i}(\Sigma).
\end{eqnarray*}
Since a fan is uniquely determined by its quiddity sequence, the assertion follows.
\end{proof}

The following main result of this section shows that mutation of algebras is an algebraic counterpart of rotation of fans, where we identify $K_0(\proj\Lambda)$ with $\Z^2$ by $e_1\Lambda=(1,0)$ and $e_2\Lambda=(0,1)$.

\begin{theorem}[{Rotation Theorem}]\label{theorem:rotation}
Let $\Lambda$ be a finite dimensional $k$-algebra of rank 2 
with orthogonal primitive idempotents $1=e_1+e_2$. Assume $e_1\Lambda e_2=0$, or equivalently, $\Sigma(\Lambda)\in\itscfanpm$ (Proposition \ref{+- condition}).
 Then, there exists a finite dimensional $k$-algebra $\rho(\Lambda)$ such that
\[\Sigma(\rho(\Lambda))=\rho(\Sigma(\Lambda)).\]
Furthermore, if $\Lambda$ is elementary, then 
$\rho(\Lambda)$ is elementary. 
\end{theorem}

Theorem \ref{theorem:rotation} is explained by the following picture.
    \[
    {\Sigma(\Lambda) =\begin{xy}
				0;<4pt,0pt>:<0pt,4pt>::
				(0,-5)="0",
				(-5,0)="1",
				(0,0)*{\bullet},
				(0,0)="2",
				(5,0)="3",
				(0,5)="4",
				(15,-5)="5",
				(0,-4)="a",
				(4,-1.33)="b",
				(1.5,1.5)*{{\scriptstyle +}},
				(-1.5,-1.5)*{{\scriptstyle -}},
				(7,1)*{},
				(0,6.7)*{},
				\ar@{-}"0";"1",
				\ar@{-}"1";"4",
				\ar@{-}"3";"5",
				\ar@{-}"3";"4",
				\ar@{-}"2";"0",
				\ar@{-}"2";"1",
				\ar@{-}"2";"3",
				\ar@{-}"2";"4",
				\ar@{-}"2";"5",
				\ar@/^-2mm/@{.} "a";"b",
		\end{xy}}\ \ \ \ \  
		{\Sigma(\rho(\Lambda))\simeq\begin{xy}
				0;<4pt,0pt>:<0pt,4pt>::
				(0,-5)="0",
				(-5,0)="1",
				(0,0)*{\bullet},
				(0,0)="2",
				(5,0)="3",
				(0,5)="4",
				(15,-5)="5",
				(-15,5)="6",
				(0,-4)="a",
				(4,-1.33)="b",
				(5,-0.85)*{{\scriptstyle +}},
				(-5,0.85)*{{\scriptstyle -}},
				(7,1)*{},
				(0,6.7)*{},
				\ar@{-}"0";"1",
				\ar@{-}"3";"5",
				\ar@{-}"2";"0",
				\ar@{-}"2";"1",
				\ar@{-}"2";"3",
				\ar@{-}"2";"5",
				\ar@{-}"6";"1",
				\ar@{-}"6";"2",
				\ar@{-}"6";"3",
				\ar@/^-2mm/@{.} "a";"b",
			\end{xy}}
			\]
The algebra $\rho(\Lambda)$ is consturcted by
\[E:=\End_{\Kb(\proj \Lambda)}(\mu_2^-(\Lambda))\ \mbox{ and }\ \rho(\Lambda):=\left[\begin{array}{cc}
e_1Ee_1&0\\ e_2Ee_1&e_2Ee_2
\end{array}\right],\]
where $e_1\in E$ is an idempotent corresponding to the direct summand $e_1\Lambda$ of $\mu_2^-(\Lambda)$ and $e_2:=1-e_1\in E$.
To prove Theorem \ref{theorem:rotation}, we need the following preparation.

Let $A$ be a basic finite dimensional algebra over a field $k$ with $|A|=n$, and $1=e_1+\cdots+e_n$ the orthogonal primitive idempotents.
For $1\le i\le n$ and $\delta\in\{\pm1\}$, consider a half space
\[\R^n_{i,\delta}:=\{x_1e_1+\cdots+x_de_n\in\R^n\mid \delta x_i\ge0\}\]
and define a subfan of $\Sigma$ by
\[\Sigma_{i,\delta}:=\{\sigma\in\Sigma\mid\sigma\subset\R^n_{i,\delta}\}.\]
On the other hand, for elements $T\ge T'$ in $\silt A$, we consider the interval
\[[T',T]:=\{U\in\silt A\mid T\ge U\ge T'\}.\]
The following result provides a correspondence of a part of two $g$-fans.

\begin{proposition}\label{lower upper bij}
For $1\le i\le n$, let $B:=\End_A(\mu_i^-(A))$, where $\mu^-_i(A)=T_i\oplus(\bigoplus_{j\neq i}P_j^A)$. 
\begin{enumerate}[\rm(a)]
\item\cite[Threom 4.26]{AHIKM} There exists a triangle functor $F:\Kb(\proj A)\to\Kb(\proj B)$ which satisfies $F(T_i)\simeq P_i^B$ and $F(P_j^A)\simeq P_j^B$ for each $j\neq i$ and gives an isomorphism $K_0(\proj A)\simeq K_0(\proj B)$ 
and an isomorphism of fans
\[\Sigma_{i,-}(A)\simeq\Sigma_{i,+}(B).\]
\item There are isomorphisms $(1-e_i)A(1-e_i)\simeq(1-e_i)B(1-e_i)$ and $A/(1-e_i)\simeq B/(1-e_i)$ of $k$-algebras.
\end{enumerate}
\end{proposition}

\begin{proof}
(b) Although this is known to experts, we give a proof for convenience of the reader. The first isomorphism is clear. To prove the second one, notice that $A/(1-e_i)=\End_{\Kb(\proj A)}(P_i^A)/[A/P_i^A]$ and $B/(1-e_i)=\End_{\Kb(\proj A)}(T_i)/[T/T_i]$ hold, where $[X]$ denotes the ideal consisting of morphisms factoring through $\add X$.
Let $P_i \xrightarrow{f} Q \xrightarrow{g} T_i\xrightarrow{h} P_i[1]$ be an exchange triangle. Let $a\in e_iAe_i=\End_{\Kb(\proj A)}(P_i)$. 
Since $f$ is a minimal left $(\add A/P_i)$-approximation of $P_i$,
	we obtain the following commutative diagram.
	\[
	\xymatrix@R=1.5em{
		P_i  \ar[r]^f \ar[d]^a & Q \ar[r]^g \ar[d] & T_i \ar[r]^h\ar[d]^b&P_i[1]\ar[d]^{a[1]}\\
		P_i  \ar[r]^f & Q \ar[r]^g & T_i\ar[r]^h & P_i[1]}
	\]
	It is routine to check that the desired isomorphism $A/(1-e_i)A=\End_{\Kb(\proj A)}(P_i^A)/[A/P_i^A]\simeq B/(1-e_i)=\End_{\Kb(\proj A)}(T_i)/[T/T_i]$ is given by $a\mapsto b$.	
\end{proof}

We are ready to prove Theorem \ref{theorem:rotation}.

\begin{proof}[Proof of Theorem \ref{theorem:rotation}]
Let $T=P_1^\Lambda\oplus T_2:=\mu_2^-(\Lambda)$. Then $E=\End_{\Kb(\proj \Lambda)}(T)$.
	By Proposition \ref{lower upper bij}(a), we have a triangle functor $F:\Kb(\proj \Lambda)\to \Kb(\proj E)$ which satisfies
	\begin{equation*}
	F(P_1^\Lambda)=P_1^E\ \mbox{ and }\ F(T_2)=P_2^E
	\end{equation*}
	and induces an isomorphism $F:K_0(\proj \Lambda)\simeq K_0(\proj E)$ and an isomorphism of fans
	\begin{eqnarray*}
	&F\colon  
	\Sigma_{2,-}(\Lambda)
	\simeq \Sigma_{2,+}(E).&
    \end{eqnarray*}
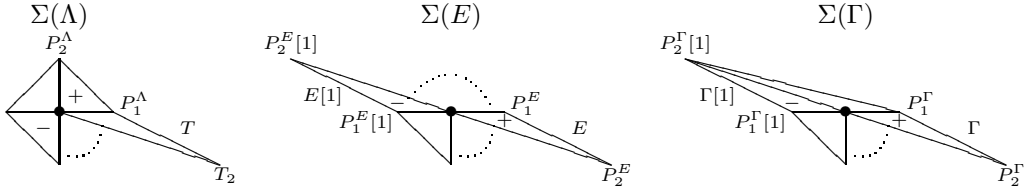
    \begin{equation*}
    \begin{array}{ccc}
    {\begin{xy}
				0;<4pt,0pt>:<0pt,4pt>::
				(0,-5)="0",
				(-5,0)="1",
				(0,0)*{\bullet},
				(0,0)="2",
				(5,0)="3",
				(0,5)="4",
				(15,-5)="5",
				(0,-4)="a",
				(4,-1.33)="b",
		(0,9)*{\Sigma(\Lambda)},
		(0,6.5)*{{\scriptstyle P_2^\Lambda}},
		(7,0.8)*{{\scriptstyle P_1^\Lambda}},
		(15.5,-6)*{{\scriptstyle T_2}},
				(12,-1.5)*{{\scriptstyle T}},
				(1.5,1.5)*{{\scriptstyle +}},
				(-1.5,-1.5)*{{\scriptstyle -}},
				(7,1)*{},
				(0,6.7)*{},
				\ar@{-}"0";"1",
				\ar@{-}"1";"4",
				\ar@{-}"3";"5",
				\ar@{-}"3";"4",
				\ar@{-}"2";"0",
				\ar@{-}"2";"1",
				\ar@{-}"2";"3",
				\ar@{-}"2";"4",
				\ar@{-}"2";"5",
				\ar@/^-2mm/@{.} "a";"b",
		\end{xy}}
		&{\begin{xy}
				0;<4pt,0pt>:<0pt,4pt>::
				(0,-5)="0",
				(-5,0)="1",
				(0,0)*{\bullet},
				(0,0)="2",
				(5,0)="3",
				(0,5)="4",
				(15,-5)="5",
				(-15,5)="6",
				(0,-4)="a",
				(4,-1.33)="b",
				(4,0)="c",
				(-4,1.33)="d",
		(0,9)*{\Sigma(E)},
		(-15,6.5)*{{\scriptstyle P_2^E[1]}},
		(7,0.8)*{{\scriptstyle P_1^E}},
		(15.5,-6)*{{\scriptstyle P_2^E}},
		(-8,-1)*{{\scriptstyle P_1^E[1]}},
				(5,-0.85)*{{\scriptstyle +}},
				(-5,0.85)*{{\scriptstyle -}},
				(12,-1.5)*{{\scriptstyle E}},
				(-12,1.5)*{{\scriptstyle E[1]}},
				(7,1)*{},
				(0,6.7)*{},
				\ar@{-}"0";"1",
				\ar@{-}"3";"5",
				\ar@{-}"2";"0",
				\ar@{-}"2";"1",
				\ar@{-}"2";"3",
				\ar@{-}"2";"5",
				\ar@{-}"6";"1",
				\ar@{-}"6";"2",
				\ar@/^-2mm/@{.} "a";"b",
				\ar@/^-4mm/@{.} "c";"d",
			\end{xy}}
		&{\begin{xy}
				0;<4pt,0pt>:<0pt,4pt>::
				(0,-5)="0",
				(-5,0)="1",
				(0,0)*{\bullet},
				(0,0)="2",
				(5,0)="3",
				(0,5)="4",
				(15,-5)="5",
				(-15,5)="6",
				(0,-4)="a",
				(4,-1.33)="b",
		(0,9)*{\Sigma(\Gamma)},
		(-15,6.5)*{{\scriptstyle P_2^{\Gamma}[1]}},
		(7,0.8)*{{\scriptstyle P_1^{\Gamma}}},
		(15.5,-6)*{{\scriptstyle P_2^{\Gamma}}},
		%(0,-6.5)*{{\scriptstyle v_{n-1}}},
		(-8,-1)*{{\scriptstyle P_1^{\Gamma}[1]}},
				(5,-0.85)*{{\scriptstyle +}},
				(-5,0.85)*{{\scriptstyle -}},
				%(1.5,1.5)*{{\scriptstyle +}},
				%(-1.5,-1.5)*{{\scriptstyle -}},
				(12,-1.5)*{{\scriptstyle \Gamma}},
				%(-5,4)*{{\scriptstyle\tau}},
				(-12,1.5)*{{\scriptstyle \Gamma[1]}},
				%(-4,4)*{{\scriptstyle \sigma''}},
				(7,1)*{},
				(0,6.7)*{},
				%(2.5,-3)*{{\scriptstyle ?}},
				\ar@{-}"0";"1",
				%\ar@{-}"1";"4",
				\ar@{-}"3";"5",
				%\ar@{-}"3";"4",
				\ar@{-}"2";"0",
				\ar@{-}"2";"1",
				\ar@{-}"2";"3",
				%\ar@{-}"2";"4",
				\ar@{-}"2";"5",
				\ar@{-}"6";"1",
				\ar@{-}"6";"2",
				\ar@{-}"6";"3",
				\ar@/^-2mm/@{.} "a";"b",
			\end{xy}}
			\end{array}
			\end{equation*}
The algebra $\Gamma=\rho(\Lambda)$ is $E_{-+}$ in terms of Definition \ref{define A_epsilon}. Thus Theorem \ref{sign decomposition} implies
		$$\Sigma_{-+}(\Gamma)=\Sigma_{-+}(E).$$
Under the isomorphism $K_0(\proj\Gamma)\simeq\Z^2$ given by $P_1^\Gamma\mapsto(0,1)$ and $P_2^\Gamma\mapsto(1,0)$, we obtain $\Sigma(\Gamma)=\rho(\Sigma(\Lambda))$. Therefore $\rho(\Lambda):=\Gamma$ satisfies the desired condition. 

It remains to prove the last assertion. By Proposition \ref{lower upper bij}(a), we have isomorphisms $e_1Ee_1\simeq e_1\Lambda e_1$ and $\Lambda/(e_1)\simeq E/(e_1)$ of $k$-algebras. Thus, if $\Lambda$ is elementary, then so are $E$ and $\Gamma$.
\end{proof}

We give two examples of Theorem \ref{theorem:rotation}. The first one satisfies $E=\Gamma$.

\begin{example}
			\label{example:2121}
			Let $\Lambda$ be the following algebra. Then $\Sigma(\Lambda)$ is the following fan by Example \ref{example:1212} below.
        \[\Lambda= \dfrac{k\left[\begin{xy}( 0,0) *+{1}="1",
				( 8,0) *+{2}="2",\ar "1";"2"^a  \ar @(lu,ld)"1";"1"_{b}  
			\end{xy}\right]}{\langle b^2\rangle}\ \ \ \ \ 		\Sigma(\Lambda)=\Sigma_{1212}=\begin{tikzpicture}[baseline=-6mm]
				%%[ 1, 2, 1, 2 ]
				\coordinate(0) at(0:0);
				\coordinate(x) at(0:0.66);
				\coordinate(y) at(90:0.66);
				\coordinate(shift1212) at (0);
                \node at(0.2,0.2) {$\scriptstyle +$};
                \node at(-0.2,-0.2) {$\scriptstyle -$};
				\coordinate(v0) at($0*(x)+0*(y)+(shift1212)$); 
				\coordinate(v1) at($1*(x)+0*(y)+(shift1212)$); 
				\coordinate(v2) at($1*(x)+-1*(y)+(shift1212)$); 
				\coordinate(v3) at($1*(x)+-2*(y)+(shift1212)$); 
				\coordinate(v4) at($0*(x)+-1*(y)+(shift1212)$); 
				\draw[fill=black] (v0) circle [radius = 0.55mm]; 
				\draw (v0)--(v1); 
				\draw (v0)--(v2); 
				\draw (v0)--(v3); 
				\draw (v0)--(v4); 
				\draw (v1)--(v2)--(v3)--(v4); 
				\coordinate(w1) at($1*(x)+0*(y)+(shift1212)$); 
				\coordinate(w2) at($0*(x)+-1*(y)+(shift1212)$); 
				\coordinate(w3) at($-1*(x)+0*(y)+(shift1212)$); 
				\coordinate(w4) at($0*(x)+1*(y)+(shift1212)$); 
				\draw (v0)--(w3); 
				\draw (v0)--(w4);
				\draw (w3)--(w4); 
				\draw (w1)--(w4);
				\draw (w2)--(w3); 
			\end{tikzpicture}\]
   
		We set $\mu_2(\Lambda)=T=T_1\oplus T_2:=[e_2 \Lambda\xrightarrow{a\cdot} e_1 \Lambda]\oplus e_1 \Lambda$
		and $E:=\End_{\Kb(\proj \Lambda)}(T)$. Then, we have 
		\[
		\Gamma=E= \dfrac{k\left[\begin{xy}( 0,0) *+{1}="1",
				( 8,0) *+{2}="2",\ar "1";"2"^a  \ar @(ru,rd)"2";"2"^{b}  
			\end{xy}\right]}{\langle b^2\rangle}
		\ \text{and }\ \Sigma(\Gamma)=\rho(\Sigma(\Lambda))=\Sigma_{2121}=\begin{tikzpicture}[baseline=-4mm]
				%%[ 2, 1, 2, 1 ]
				\coordinate(0) at(0:0);
				\coordinate(x) at(0:0.66);
				\coordinate(y) at(90:0.66);
                \node at(0.2,0.2) {$\scriptstyle +$};
                \node at(-0.2,-0.2) {$\scriptstyle -$};
				\coordinate(shift2121) at (0);
				\coordinate(v0) at($0*(x)+0*(y)+(shift2121)$); 
				\coordinate(v1) at($1*(x)+0*(y)+(shift2121)$); 
				\coordinate(v2) at($2*(x)+-1*(y)+(shift2121)$); 
				\coordinate(v3) at($1*(x)+-1*(y)+(shift2121)$); 
				\coordinate(v4) at($0*(x)+-1*(y)+(shift2121)$); 
				\draw[fill=black] (v0) circle [radius = 0.55mm]; 
				\draw (v0)--(v1); 
				\draw (v0)--(v2); 
				\draw (v0)--(v3); 
				\draw (v0)--(v4); 
				\draw (v1)--(v2)--(v3)--(v4); 
				\coordinate(w1) at($1*(x)+0*(y)+(shift2121)$); 
				\coordinate(w2) at($0*(x)+-1*(y)+(shift2121)$); 
				\coordinate(w3) at($-1*(x)+0*(y)+(shift2121)$); 
				\coordinate(w4) at($0*(x)+1*(y)+(shift2121)$); 
				\draw (v0)--(w3); 
				\draw (v0)--(w4);
				\draw (w3)--(w4); 
				\draw (w1)--(w4);
				\draw (w2)--(w3); 
			\end{tikzpicture}
		\]
		\end{example}

The second example satisfies $E\neq \Gamma$. 
\begin{example}\label{example:13122}
Let $\Lambda$ be the following algebra. Then $\Sigma(\Lambda)$ is the following fan by Example \ref{example:21312} below.
\[\Lambda=\dfrac{
	k\left[\begin{xy}( 0,0) *+{1}="1",
		( 8,0) *+{2}="2",\ar "1";"2"^a  \ar @(lu,ld)"1";"1"_{b}  \ar @(ru,rd)"2";"2"^{c}
	\end{xy}\right]}{\langle b^2, c^2, b ac\rangle}\ \ \ \ \ 
\Sigma(\Lambda)=\Sigma_{21312}=
\begin{tikzpicture}[baseline=-6mm]
		%%[ 2, 1, 3, 1, 2 ]
		\coordinate(0) at(0:0);
		\coordinate(x) at(0:0.66);
		\coordinate(y) at(90:0.66);
        \node at(0.2,0.2) {$\scriptstyle +$};
        \node at(-0.2,-0.2) {$\scriptstyle -$};
		\coordinate(shift21312) at (0:0);
		\coordinate(v0) at($0*(x)+0*(y)+(shift21312)$); 
		\coordinate(v1) at($1*(x)+0*(y)+(shift21312)$); 
		\coordinate(v2) at($2*(x)+-1*(y)+(shift21312)$); 
		\coordinate(v3) at($1*(x)+-1*(y)+(shift21312)$); 
		\coordinate(v4) at($1*(x)+-2*(y)+(shift21312)$); 
		\coordinate(v5) at($0*(x)+-1*(y)+(shift21312)$); 
		\draw[fill=black] (v0) circle [radius = 0.55mm]; 
		\draw (v0)--(v1); 
		\draw (v0)--(v2); 
		\draw (v0)--(v3); 
		\draw (v0)--(v4); 
		\draw (v0)--(v5); 
		\draw (v1)--(v2)--(v3)--(v4)--(v5); 
		\coordinate(w1) at($1*(x)+0*(y)+(shift21312)$); 
		\coordinate(w2) at($0*(x)+-1*(y)+(shift21312)$); 
		\coordinate(w3) at($-1*(x)+0*(y)+(shift21312)$); 
		\coordinate(w4) at($0*(x)+1*(y)+(shift21312)$); 
		\draw (v0)--(w3); 
		\draw (v0)--(w4);
		\draw (w3)--(w4); 
		\draw (w1)--(w4);
		\draw (w2)--(w3); 
	\end{tikzpicture}	
	\]
 
We set $\mu_2(\Lambda)=T=T_1\oplus T_2:=[e_2 \Lambda\xrightarrow{\left(\begin{smallmatrix} a\cdot \\ ac\cdot \end{smallmatrix}\right)} e_1 \Lambda^{\oplus 2}]\oplus e_1 \Lambda$ and $E:=\End_{\Kb(\proj \Lambda)}(T)$, where we switch the indices $1$ and $2$ unlike the proof of Theorem \ref{theorem:rotation}. Then, we have 
\[
E=\dfrac{k\left[\begin{xy}( 0,0) *+{1}="1",
		( 8,0) *+{2}="2",\ar@<2pt> "1";"2"^a \ar@<2pt> "2";"1"^{a'}   \ar @(lu,ld)"1";"1"_{b} 
	\end{xy}\right]}{\langle b^2, a'b, a'aa'\rangle}\ \mbox{ and }\ 
\Sigma(E)=\Sigma_{13122;111}=
\begin{tikzpicture}[baseline=-8mm]
	%%[ 1, 3, 1, 2, 2; 1, 1, 1 ]
	\coordinate(0) at(0:0);
	\coordinate(x) at(0:0.66);
	\coordinate(y) at(90:0.66);
    \node at(0.2,0.2) {$\scriptstyle +$};
    \node at(-0.2,-0.2) {$\scriptstyle -$};
	\coordinate(shift13122) at (0);
	\coordinate(v0) at($0*(x)+0*(y)+(shift13122)$); 
	\coordinate(v1) at($1*(x)+0*(y)+(shift13122)$); 
	\coordinate(v2) at($1*(x)+-1*(y)+(shift13122)$); 
	\coordinate(v3) at($2*(x)+-3*(y)+(shift13122)$); 
	\coordinate(v4) at($1*(x)+-2*(y)+(shift13122)$); 
	\coordinate(v5) at($0*(x)+-1*(y)+(shift13122)$); 
	\draw[fill=black] (v0) circle [radius = 0.55mm]; 
	\draw (v0)--(v1); 
	\draw (v0)--(v2); 
	\draw (v0)--(v3); 
	\draw (v0)--(v4); 
	\draw (v0)--(v5); 
	\draw (v1)--(v2)--(v3)--(v4)--(v5); 
	\coordinate(w1) at($1*(x)+0*(y)+(shift13122)$); 
	\coordinate(w2) at($0*(x)+-1*(y)+(shift13122)$); 
	\coordinate(w3) at($-1*(x)+0*(y)+(shift13122)$);
	\coordinate(w4) at($-1*(x)+1*(y)+(shift13122)$); 
	\coordinate(w5) at($0*(x)+1*(y)+(shift13122)$); 
	\draw (v0)--(w3); 
	\draw (v0)--(w4);
	\draw (v0)--(w5);
	\draw (w3)--(w4);
	\draw (w4)--(w5); 
	\draw (w1)--(w5);
	\draw (w2)--(w3);
\end{tikzpicture}
\] 
where new arrows $a$, $a'$ and $b$ are morphisms in $\Kb(\proj \Lambda)$ given by commutative diagrams 
\[
\begin{xy}
	(0,0)*+{0}="1", (20,0)*+{e_1 \Lambda}="2",
	(0,-10)*+{e_2 \Lambda}="3", (20,-10)*+{e_1 \Lambda^{\oplus 2}}="4",
	\ar "1";"2"
	\ar "1";"3"
	\ar "3";"4"^{\left(\begin{smallmatrix} a\cdot \\ ac\cdot  \end{smallmatrix}\right)}
	\ar "2";"4"^{\left(\begin{smallmatrix}0 \\ 1 \end{smallmatrix}\right)}
\end{xy}\hspace{10pt}
\begin{xy}
	(0,0)*+{e_2 \Lambda}="1", (20,0)*+{e_1 \Lambda^{\oplus 2}}="2",
	(0,-10)*+{0}="3", (20,-10)*+{e_1 \Lambda}="4",
	\ar "1";"2"^{\left(\begin{smallmatrix}a\cdot\\ ac\cdot \end{smallmatrix}\right)}
	\ar "1";"3"
	\ar "3";"4"
	\ar "2";"4"^{\left(\begin{smallmatrix}0 & b\cdot \end{smallmatrix}\right)}
\end{xy}
\hspace{10pt}
\begin{xy}
	(0,0)*+{e_2 \Lambda}="1", (20,0)*+{e_1 \Lambda^{\oplus 2}}="2",
	(0,-10)*+{e_2 \Lambda}="3", (20,-10)*+{e_1 \Lambda^{\oplus 2}}="4",
	\ar "1";"2"^{\left(\begin{smallmatrix} a\cdot \\ ac\cdot \end{smallmatrix}\right)}
	\ar "1";"3"_{c\cdot}
	\ar "3";"4"^{\left(\begin{smallmatrix} a\cdot \\ ac\cdot  \end{smallmatrix}\right)}
	\ar "2";"4"^{\left(\begin{smallmatrix}0 & 1 \\ 0 & 0 \end{smallmatrix}\right)}
\end{xy}
\]
respectively. Let $\Gamma:=E_{+-}=\left[\begin{array}{cc} e_1 E e_1 & e_1 E e_2 \\ 0 & e_2 E e_2 \end{array}\right]=\left[\begin{array}{cc} \langle e_1,b,aa',baa'\rangle_k & \langle a,ba,aa'a,baa'a\rangle_k \\ 0 & \langle e_2,a'a\rangle_k \end{array}\right]$. 
Then, we have
\[
\Gamma=\dfrac{k\left[\begin{xy}( 0,0) *+{1}="1",
 		( 8,0) *+{2}="2",\ar "1";"2"^a  \ar @(ru,rd)"2";"2"^{c}  \ar @(dl,dr)"1";"1"_(.2){b} \ar @ (ul,ur)"1";"1"^(.2){b'}
 	\end{xy}\right]}{\langle b^2, b'^2, c^2, b'b,b'a-ac\rangle}\ \mbox{ and }\ \Sigma(\Gamma)=\rho(\Sigma(\Lambda))=\Sigma_{13122}=\begin{tikzpicture}[baseline=-8mm]
	%%[ 1, 3, 1, 2, 2 ]
	\coordinate(0) at(0:0);
	\coordinate(x) at(0:0.66);
	\coordinate(y) at(90:0.66);
    \node at(0.2,0.2) {$\scriptstyle +$};
    \node at(-0.2,-0.2) {$\scriptstyle -$};
	\coordinate(shift13122) at (0);
	\coordinate(v0) at($0*(x)+0*(y)+(shift13122)$); 
	\coordinate(v1) at($1*(x)+0*(y)+(shift13122)$); 
	\coordinate(v2) at($1*(x)+-1*(y)+(shift13122)$); 
	\coordinate(v3) at($2*(x)+-3*(y)+(shift13122)$); 
	\coordinate(v4) at($1*(x)+-2*(y)+(shift13122)$); 
	\coordinate(v5) at($0*(x)+-1*(y)+(shift13122)$); 
	\draw[fill=black] (v0) circle [radius = 0.55mm]; 
	\draw (v0)--(v1); 
	\draw (v0)--(v2); 
	\draw (v0)--(v3); 
	\draw (v0)--(v4); 
	\draw (v0)--(v5); 
	\draw (v1)--(v2)--(v3)--(v4)--(v5); 
	\coordinate(w1) at($1*(x)+0*(y)+(shift13122)$); 
	\coordinate(w2) at($0*(x)+-1*(y)+(shift13122)$); 
	\coordinate(w3) at($-1*(x)+0*(y)+(shift13122)$); 
	\coordinate(w4) at($0*(x)+1*(y)+(shift13122)$); 
	\draw (v0)--(w3); 
	\draw (v0)--(w4);
	\draw (w3)--(w4); 
	\draw (w1)--(w4);
	\draw (w2)--(w3); 
\end{tikzpicture}
\]
where $b':=aa'$ and $c:=a'a$.
\end{example}

\subsection{Subdivision and Extension}
\label{subsect:subdivisions of g-fans}

In this section, we realize subdivisions of $g$-fans of rank 2 by extensions of algebras. 
The following is a main result of this section, where we identify $K_0(\proj\Lambda)$ with $\Z^2$ by $e_1\Lambda=(1,0)$ and $e_2\Lambda=(0,1)$.

\begin{theorem}[Subdivision Theorem]\label{theorem:subdivision}
	Let $\Lambda$ be a finite dimensional elementary $k$-algebra 
	with orthogonal primitive idempotents $1=e_1+e_2$. Assume $e_1\Lambda e_2=0$, or equivalently, $\Sigma(\Lambda)\in\itscfanpm$ (Proposition \ref{+- condition}). 
	Then, for $T\in\{\mu_1^{+}(\Lambda[1]),\mu_2^{-}(\Lambda)\}$, there exists a finite dimensional elementary $k$-algebra $D_T(\Lambda)$ such that
	\[\Sigma(D_T(\Lambda)) = D_{C(T)}(\Sigma(\Lambda)).\]
\end{theorem}

Theorem \ref{theorem:subdivision} is explained by the following picture, where $T=\mu_1^{+}(\Lambda[1])$ and $T'=\mu_2^{-}(\Lambda)$.
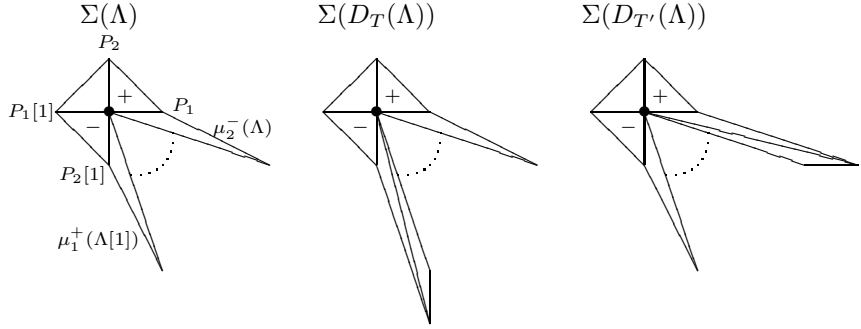
\[{\begin{xy}
				0;<4pt,0pt>:<0pt,4pt>::
				(0,-5)="0",
				(-5,0)="1",
				(0,0)*{\bullet},
				(0,0)="2",
				(5,0)="3",
				(0,5)="4",
				(15,-5)="5",
				(5,-15)="6",
				(2,-6)="a",
				(6,-2)="b",
		(0,6.5)*{{\scriptstyle P_2}},
		(7,0.5)*{{\scriptstyle P_1}}, 
		(-2.5,-6)*{{\scriptstyle P_2[1]}},
		(-7.3,0)*{{\scriptstyle P_1[1]}},
		(0,9)*{\Sigma(\Lambda)},
				(1.5,1.5)*{{\scriptstyle +}},
				(-1.5,-1.5)*{{\scriptstyle -}},
				(12.5,-1.5)*{{\scriptstyle\mu_2^-(\Lambda)}},
				(-1,-12)*{{\scriptstyle\mu_1^+(\Lambda[1])}},
				%(-4,4)*{{\scriptstyle \sigma_{-+}}},
				(7,1)*{},
				(0,6.7)*{},
				%(2.5,-3)*{{\scriptstyle ?}},
				\ar@{-}"0";"1",
				\ar@{-}"1";"4",
				\ar@{-}"3";"5",
				\ar@{-}"3";"4",
				\ar@{-}"2";"0",
				\ar@{-}"2";"1",
				\ar@{-}"2";"3",
				\ar@{-}"2";"4",
				\ar@{-}"2";"5",
				\ar@{-}"2";"6",
				\ar@{-}"0";"6",
				\ar@/^-2mm/@{.} "a";"b",
		\end{xy}}\ \ \ \ \ 
{\begin{xy}
				0;<4pt,0pt>:<0pt,4pt>::
				(0,-5)="0",
				(-5,0)="1",
				(0,0)*{\bullet},
				(0,0)="2",
				(5,0)="3",
				(0,5)="4",
				(15,-5)="5",
				(5,-15)="6",
		(5,-20)="7",
				(2,-6)="a",
				(6,-2)="b",
				(1.5,1.5)*{{\scriptstyle +}},
				(-1.5,-1.5)*{{\scriptstyle -}},
				(0,9)*{\Sigma(D_T(\Lambda))},
				(7,1)*{},
				(0,6.7)*{},
				\ar@{-}"0";"1",
				\ar@{-}"1";"4",
				\ar@{-}"3";"5",
				\ar@{-}"3";"4",
				\ar@{-}"2";"0",
				\ar@{-}"2";"1",
				\ar@{-}"2";"3",
				\ar@{-}"2";"4",
				\ar@{-}"2";"5",
				\ar@{-}"2";"6",
				\ar@{-}"2";"7",
				\ar@{-}"0";"7",
				\ar@{-}"6";"7",
				\ar@/^-2mm/@{.} "a";"b",
		\end{xy}}\ \ \ \ \ 
{\begin{xy}
				0;<4pt,0pt>:<0pt,4pt>::
				(0,-5)="0",
				(-5,0)="1",
				(0,0)*{\bullet},
				(0,0)="2",
				(5,0)="3",
				(0,5)="4",
				(15,-5)="5",
				(5,-15)="6",
		(20,-5)="7",
				(2,-6)="a",
				(6,-2)="b",
				(1.5,1.5)*{{\scriptstyle +}},
				(-1.5,-1.5)*{{\scriptstyle -}},
				(0,9)*{\Sigma(D_{T'}(\Lambda))},
				(7,1)*{},
				(0,6.7)*{},
				%(2.5,-3)*{{\scriptstyle ?}},
				\ar@{-}"0";"1",
				\ar@{-}"1";"4",
				\ar@{-}"7";"5",
				\ar@{-}"3";"4",
				\ar@{-}"2";"0",
				\ar@{-}"2";"1",
				\ar@{-}"2";"3",
				\ar@{-}"2";"4",
				\ar@{-}"2";"5",
				\ar@{-}"2";"6",
				\ar@{-}"2";"7",
				\ar@{-}"3";"7",
				\ar@{-}"6";"0",
				\ar@/^-2mm/@{.} "a";"b",
		\end{xy}}
\]
In the rest, we consider the case $T=\mu_1^+(\Lambda[1])$ since the case $T=\mu_2^-(\Lambda)$ is the dual.
The construction of $D_T(\Lambda)$ for $T=\mu_1^+(\Lambda[1])$ is as follows: 

\begin{construction}
By Proposition \ref{+- condition}, we can write 
\[
\Lambda=\left[\begin{array}{cc}
	A&X\\ 0&B
\end{array}\right].
\]
where $A,B$ are local $k$-algebras and $X$ is an $A^{\rm op}\otimes_k B$-module. Since $\Lambda$ is elementary, we have $k\simeq A/J_A\simeq B/J_B$. 
Let
\[\overline{X}:=X/XJ_B.\]
Then the $k$-dual $D\overline{X}$ is an $A$-module, and we regard it as an $A^{\op}$-module by using the action of $k$ through the natural surjection $A\to k$. Let
\[C:=A\oplus D\overline{X}\]
be a trivial extension algebra of $A$ by $D\overline{X}$. Let
\[\overline{(\cdot)}:A\to k,\ \overline{(\cdot)}:B\to k\ \mbox{ and }\ \overline{(\cdot)}:X\to\overline{X}\]
be canonical surjections. We regard
\[Y:=\left[\begin{array}{c}k\\ X\end{array}\right]\]
as a $C^{\op}\otimes_kB$-module by
\[(a,f)\cdot\left[\begin{smallmatrix}\alpha\\ x\end{smallmatrix}\right]\cdot b:=\left[\begin{smallmatrix}\overline{a}\alpha\overline{b}+f(\overline{x})\overline{b}\\ axb\end{smallmatrix}\right]\ \mbox{ for $(a,f)\in C=A\oplus D\overline{X}$, $\left[\begin{smallmatrix}\alpha\\ x\end{smallmatrix}\right]\in Y=\left[\begin{array}{c}k\\ X\end{array}\right]$ and $b\in B$.}\] 
Then we set
\[
D_T(\Lambda):=\left[\begin{array}{cc}
	C&Y\\ 0&B
\end{array}\right].
\]
\end{construction}

In the rest of this subsection, we prove Theorem \ref{theorem:subdivision}. 
For simplicity, we set
\[\Gamma:=D_T(\Lambda)\ \mbox{ and }\ Q_1:=[C\ Y],\ Q_2:=[0\ B]\in\proj\Gamma.\]
For $y\in M_{s,t}(Y)\simeq\Hom_\Gamma(Q_2^{\oplus t},Q_1^{\oplus s})$, we define
\[Q_y:=[Q_2^{\oplus t}\xrightarrow{y(\cdot)} Q_1^{\oplus s}]\in \Kb(\proj \Gamma).\]

We fix a minimal set of generators $g_1,\ldots,g_r$ of the $B$-module $X$. Then $(\overline{g_1},\dots, \overline{g_r})$ forms a $k$-basis of $\overline{X}=X/XJ_B$. Set 
	\[
		g:=[g_1\ \cdots\ g_r]\in M_{1,r}(X)\ \mbox{ and }\ 
\overline{g}:=[\overline{g_1}\ \cdots\ \overline{g_r}]\in M_{1,r}(X/XJ_B).	\]
We need the following easy observation. 

\begin{lemma}\label{2 cones}
$\Sigma(\Gamma)$ contains $\cone\{(0,1),(1,-r-1)\}$ and $\cone\{(1,-r-1),(1,-r)\}$. More explicitly, let
\[\left[\begin{smallmatrix}0\\ g\end{smallmatrix}\right]\in M_{1,r}(Y)\ \mbox{ and }\ \left[\begin{smallmatrix}0&1\\ g&0\end{smallmatrix}\right]\in M_{1,r+1}(Y).\]
Then $Q_{\left[\begin{smallmatrix}0&1\\ g&0\end{smallmatrix}\right]}\oplus Q_2[1]$ and $Q_{\left[\begin{smallmatrix}0\\ g\end{smallmatrix}\right]}\oplus Q_{\left[\begin{smallmatrix}0&1\\ g&0\end{smallmatrix}\right]}$ belong to $\twosilt\Gamma$.
\end{lemma}

\begin{proof}
	A minimal set of generators of the $B$-module $Y$ is given by the $r+1$ columns of $\left[\begin{smallmatrix}0&1\\ g&0\end{smallmatrix}\right]$. Thus 
	$Q_{\left[\begin{smallmatrix}0&1\\ g&0\end{smallmatrix}\right]}\oplus Q_2[1]\in\twosilt\Gamma$ holds by Proposition \ref{first mutation}.
	
	In the rest, we prove that $T:=Q_{\left[\begin{smallmatrix}0\\ g\end{smallmatrix}\right]}\oplus Q_{\left[\begin{smallmatrix}0&1\\ g&0\end{smallmatrix}\right]}$ is basic silting. By the first statement, $Q_{\left[\begin{smallmatrix}0&1\\ g&0\end{smallmatrix}\right]}$ is indecomposable. If $Q_{\left[\begin{smallmatrix}0\\ g\end{smallmatrix}\right]}$ is not indecomposable, then $|T|$ is bigger than two, a contradiction. Thus $T$ is basic.
	
		We will show that $T$ is presilting by using Proposition \ref{Ext 1}. By our choice of $g$, we have
	\[gM_{r,1}(B)=X\ \mbox{ and }\ (D\overline{X})g=M_{1,r}(k).\]
	Thus we have $\left[\begin{smallmatrix}0\\ g\end{smallmatrix}\right]M_r(B)=M_{1,r}(\left[\begin{smallmatrix}0\\ X\end{smallmatrix}\right])$ and $(D\overline{X})\left[\begin{smallmatrix}0\\ g\end{smallmatrix}\right]=M_{1,r}(\left[\begin{smallmatrix}k\\ 0\end{smallmatrix}\right])$, and hence
	\[C\left[\begin{smallmatrix}0\\ g\end{smallmatrix}\right]+\left[\begin{smallmatrix}0\\ g\end{smallmatrix}\right]M_r(B)\supset (D\overline{X})\left[\begin{smallmatrix}0\\ g\end{smallmatrix}\right]+\left[\begin{smallmatrix}0\\ g\end{smallmatrix}\right]M_r(B)=M_{1,r}(\left[\begin{smallmatrix}0\\ X\end{smallmatrix}\right])+M_{1,r}(\left[\begin{smallmatrix}k\\ 0\end{smallmatrix}\right])=M_{1,r}(Y).\]
	This clearly implies
	\[C\left[\begin{smallmatrix}0\\ g\end{smallmatrix}\right]+\left[\begin{smallmatrix}0&1\\ g&0\end{smallmatrix}\right]M_{r+1,r}(B)=M_{1,r}(Y),\]
	and a similar argument implies
	\[C\left[\begin{smallmatrix}0&1\\ g&0\end{smallmatrix}\right]+\left[\begin{smallmatrix}0\\ g\end{smallmatrix}\right]M_{r,r+1}(B)=M_{1,r+1}(Y).\]
	Thus Proposition \ref{Ext 1} implies that $T$ is presilting, as desired.
\end{proof}

As in Section \ref{section: Matrices and Presilting complexes}, the element $\overline{g}$ gives a surjection
	\[
	\pi:=(X\xrightarrow{\overline{(\cdot)}} \overline{X}\xrightarrow{(\overline{g}(\cdot))^{-1}}M_{r,1}(\overline{B})=M_{r,1}(k)),
	\]
	which extends to the map $\pi:M_{s,t}(X)\to M_{rs,t}(k)$ for each $s,t\ge 0$.

The following observation is crucial.

\begin{proposition}\label{indecomposable rigid g-vector}
	Let $s,t\ge0$. For $x\in M_{s,t}(X)$, consider $\left[\begin{smallmatrix}0\\ x\end{smallmatrix}\right]\in M_{s,t}(Y)$.
	\begin{enumerate}[\rm(a)]
		\item $P_x$ is indecomposable in $\Kb(\proj\Lambda)$ if and only if $Q_{\left[\begin{smallmatrix}0\\ x\end{smallmatrix}\right]}$ is indecomposable in $\Kb(\proj\Gamma)$.
		\item If $Q_{\left[\begin{smallmatrix}0\\ x\end{smallmatrix}\right]}$ is a presilting complex of $\Gamma$, then $P_x$ is a presilting complex of $\Lambda$.
		\item The converse of {\rm(b)} holds if $t\le rs$. 
		\item The restriction of $\Sigma(\Gamma)$ to $\{(x,y)\in\R^2\mid 0\le -y\le rx\}$ coincides with that of $\Sigma(\Lambda)$.
	\end{enumerate}
\end{proposition}

\begin{proof}
	Notice that $\Gamma$ is the trivial extension $\Lambda\oplus I$ of $\Lambda$ by the following ideal $I$ of $\Gamma$:
	\[I:=\left[\begin{array}{cc}
		D\overline{X}&k\\ 0&0
	\end{array}\right].\] 
	
	(a) Since $P_x\simeq Q_{\left[\begin{smallmatrix}0\\ x\end{smallmatrix}\right]}\otimes_\Gamma\Lambda$ and $Q_{\left[\begin{smallmatrix}0\\ x\end{smallmatrix}\right]}\simeq P_x\otimes_\Lambda\Gamma$, the assertion follows immediately.
	
	(b) Since $\Lambda=\Gamma/I$ and $Q_{\left[\begin{smallmatrix}0\\ x\end{smallmatrix}\right]}\otimes_\Gamma\Lambda\simeq P_x$, the assertion follows.
	
	(c) Assume that $P_x$ is a presilting complex of $\Lambda$. Then by Proposition \ref{Ext 1}(b), we have
	\begin{equation}\label{x presilting}
		M_{s,t}(X)=M_s(A)x+xM_t(B).
	\end{equation}
	Again by Proposition \ref{Ext 1}(b), it suffices to show the equality
	\[V:=M_s(C)\left[\begin{smallmatrix}0\\ x\end{smallmatrix}\right]+\left[\begin{smallmatrix}0\\ x\end{smallmatrix}\right]M_t(B)=M_{s,t}(\left[\begin{smallmatrix}k\\ X\end{smallmatrix}\right]).\]
	Since
	$V\supset M_s(A)\left[\begin{smallmatrix}0\\ x\end{smallmatrix}\right]+\left[\begin{smallmatrix}0\\ x\end{smallmatrix}\right]M_t(B)\stackrel{\eqref{x presilting}}{=}M_{s,t}(\left[\begin{smallmatrix}0\\ X\end{smallmatrix}\right])$ holds, it suffices to show
	\begin{equation}\label{V contains M_ab}
	V\supset M_{s,t}(\left[\begin{smallmatrix}k\\ 0\end{smallmatrix}\right]).
	\end{equation}
	By our assumption $t\le rs$ and Proposition \ref{proposition:full rank}(b), $\pi(x)$ has rank $t$ and the map
	\begin{equation}\label{pi_g(x) map}
		(\cdot)\pi(x):M_{s,rs}(k)\to M_{s,t}(k)
	\end{equation}
	is surjective. We denote by $g_1^*,\ldots,g_r^*$ the basis of $D\overline{X}$ which is dual to $g_1,\ldots,g_r$. Then the map $(\cdot)\left[\begin{smallmatrix}g_1^*\\ \svdots\\ g_r^*\end{smallmatrix}\right]:M_{1,r}(k)\simeq D\overline{X}$ is a bijection, and we denote its inverse by
	\[\pi':D\overline{X}\simeq M_{1,r}(k).\]
	It gives a bijection $\pi':M_{s}(D\overline{X})\simeq M_{s,rs}(k)$. We have a commutative diagram
	\[\xymatrix{
		M_{s}(D\overline{X})\times M_{s,t}(X)\ar[rr]^{\pi'\times\pi}\ar[drr]_{\rm eval.}&&M_{s,rs}(k)\times M_{rs,t}(k)\ar[d]^{\rm mult.}\\
		&&M_{s,t}(k)}\]
	where ${\rm eval.}$ is given by the evaluation map $D\overline{X}\times X\to D\overline{X}\times\overline{X}\to k$. Thus the commutativity of the diagram above and the surjectivity of \eqref{pi_g(x) map} shows that the map
	\[(\cdot)x:M_{s}(D\overline{X})\to M_{s,t}(k)\]
	is also surjective. Therefore the desired claim \eqref{V contains M_ab} follows from
	\[V\supset M_s(C)\left[\begin{smallmatrix}0\\ x\end{smallmatrix}\right]\supset M_s(D\overline{X})\left[\begin{smallmatrix}0\\ x\end{smallmatrix}\right]=M_{s,t}(\left[\begin{smallmatrix}k\\ 0\end{smallmatrix}\right]).\qedhere\]
(d) Immediate from (c).
\end{proof}

We are ready to prove Theorem \ref{theorem:subdivision}.

\begin{proof}[Proof of Theorem \ref{theorem:subdivision}]
The assertion follows from Lemma \ref{2 cones} and Proposition \ref{indecomposable rigid g-vector}(d).
\end{proof}

We give two examples of Subdivision Theorem \ref{theorem:subdivision}.

\begin{example}
	\label{example:1212}
Let $\Lambda$ be the following algebra. Then $\Sigma(\Lambda)$ is the following fan.
\[\Lambda=k[1\to 2]\ \ \ \Sigma(\Lambda)=\Sigma_{111}=\begin{tikzpicture}[baseline=0mm]
				\coordinate(0) at(0:0);
				\coordinate(x) at(0:0.66);
				\coordinate(y) at(90:0.66);
                \node at(0.2,0.2) {$\scriptstyle +$};
                \node at(-0.2,-0.2) {$\scriptstyle -$};
				\coordinate(shift111) at (0);
				%%[ 1, 1, 1 ]
				\coordinate(v0) at($0*(x)+0*(y)+(shift111)$); 
				\coordinate(v1) at($1*(x)+0*(y)+(shift111)$); 
				\coordinate(v2) at($1*(x)+-1*(y)+(shift111)$); 
				\coordinate(v3) at($0*(x)+-1*(y)+(shift111)$); 
				\draw[fill=black] (v0) circle [radius = 0.55mm]; 
				\draw (v0)--(v1); 
				\draw (v0)--(v2); 
				\draw (v0)--(v3); 
				\draw (v1)--(v2)--(v3); 
				\coordinate(w1) at($1*(x)+0*(y)+(shift111)$); 
				\coordinate(w2) at($0*(x)+-1*(y)+(shift111)$); 
				\coordinate(w3) at($-1*(x)+0*(y)+(shift111)$); 
				\coordinate(w4) at($0*(x)+1*(y)+(shift111)$); 
				\draw (v0)--(w3); 
				\draw (v0)--(w4);
				\draw (w3)--(w4); 
				\draw (w1)--(w4);
				\draw (w2)--(w3); 
			\end{tikzpicture}\]
		Applying Theorem\;\ref{theorem:subdivision} to $\Lambda$, we get
		\[
		\Gamma:=\begin{bmatrix}
			k\oplus Dk & \left[\begin{smallmatrix}k\\k\end{smallmatrix}\right]\\
			0 & k
		\end{bmatrix}=\dfrac{
			k\left[\begin{xy}( 0,0) *+{1}="1",
				( 8,0) *+{2}="2",\ar "1";"2"  \ar @(lu,ld)"1";"1"_{b}  
			\end{xy}\right]}{\langle b^2\rangle}\ \mbox{ and }\ \Sigma(\Gamma)=D_3(\Sigma(\Lambda))=\Sigma_{1212}=\begin{tikzpicture}[baseline=-6mm]
				%%[ 1, 2, 1, 2 ]
				\coordinate(0) at(0:0);
				\coordinate(x) at(0:0.66);
				\coordinate(y) at(90:0.66);
                \node at(0.2,0.2) {$\scriptstyle +$};
                \node at(-0.2,-0.2) {$\scriptstyle -$};
				\coordinate(shift1212) at (0);
				\coordinate(v0) at($0*(x)+0*(y)+(shift1212)$); 
				\coordinate(v1) at($1*(x)+0*(y)+(shift1212)$); 
				\coordinate(v2) at($1*(x)+-1*(y)+(shift1212)$); 
				\coordinate(v3) at($1*(x)+-2*(y)+(shift1212)$); 
				\coordinate(v4) at($0*(x)+-1*(y)+(shift1212)$); 
				\draw[fill=black] (v0) circle [radius = 0.55mm]; 
				\draw (v0)--(v1); 
				\draw (v0)--(v2); 
				\draw (v0)--(v3); 
				\draw (v0)--(v4); 
				\draw (v1)--(v2)--(v3)--(v4); 
				\coordinate(w1) at($1*(x)+0*(y)+(shift1212)$); 
				\coordinate(w2) at($0*(x)+-1*(y)+(shift1212)$); 
				\coordinate(w3) at($-1*(x)+0*(y)+(shift1212)$); 
				\coordinate(w4) at($0*(x)+1*(y)+(shift1212)$); 
				\draw (v0)--(w3); 
				\draw (v0)--(w4);
				\draw (w3)--(w4); 
				\draw (w1)--(w4);
				\draw (w2)--(w3); 
			\end{tikzpicture}\]
\end{example}
	
\begin{example}\label{example:21312}
Let $\Lambda$ be the following algebra. Then $\Sigma(\Lambda)$ is the following fan by Example \ref{example:2121}.
	\[\Lambda=\dfrac{k\left[\begin{xy}( 0,0) *+{1}="1",
			( 8,0) *+{2}="2",\ar "1";"2"^a  \ar @(ru,rd)"2";"2"^{b}  
		\end{xy}\right]}{\langle b^2\rangle}\ \ \ \Sigma(\Lambda)=\Sigma_{2121}=\begin{tikzpicture}[baseline=0mm]
			%%[ 2, 1, 2, 1 ]
			\coordinate(0) at(0:0);
			\coordinate(x) at(0:0.66);
			\coordinate(y) at(90:0.66);
            \node at(0.2,0.2) {$\scriptstyle +$};
            \node at(-0.2,-0.2) {$\scriptstyle -$};
			\coordinate(shift2121) at (0);
			\coordinate(v0) at($0*(x)+0*(y)+(shift2121)$); 
			\coordinate(v1) at($1*(x)+0*(y)+(shift2121)$); 
			\coordinate(v2) at($2*(x)+-1*(y)+(shift2121)$); 
			\coordinate(v3) at($1*(x)+-1*(y)+(shift2121)$); 
			\coordinate(v4) at($0*(x)+-1*(y)+(shift2121)$); 
			\draw[fill=black] (v0) circle [radius = 0.55mm]; 
			\draw (v0)--(v1); 
			\draw (v0)--(v2); 
			\draw (v0)--(v3); 
			\draw (v0)--(v4); 
			\draw (v1)--(v2)--(v3)--(v4); 
			\coordinate(w1) at($1*(x)+0*(y)+(shift2121)$); 
			\coordinate(w2) at($0*(x)+-1*(y)+(shift2121)$); 
			\coordinate(w3) at($-1*(x)+0*(y)+(shift2121)$); 
			\coordinate(w4) at($0*(x)+1*(y)+(shift2121)$); 
			\draw (v0)--(w3); 
			\draw (v0)--(w4);
			\draw (w3)--(w4); 
			\draw (w1)--(w4);
			\draw (w2)--(w3); 
		\end{tikzpicture}\]
		Applying Theorem\;\ref{theorem:subdivision} to $\Lambda$, we get
		\[\Gamma:=\begin{bmatrix}k\oplus D(ka) & \left[\begin{smallmatrix}k\\ \langle a,ab\rangle_k\end{smallmatrix}\right]\\
		0 & \langle e_2,b\rangle_k
	\end{bmatrix}	=\dfrac{
		k\left[\begin{xy}( 0,0) *+{1}="1",
			( 8,0) *+{2}="2",\ar "1";"2"^a  \ar @(lu,ld)"1";"1"_{c}  \ar @(ru,rd)"2";"2"^{b}
		\end{xy}\right]}{\langle b^2, c^2, cab\rangle}\ \mbox{ and }\ \Sigma(\Gamma)=D_4(\Sigma(\Lambda))=\Sigma_{21312}=
\begin{tikzpicture}[baseline=-4mm]
			%%[ 2, 1, 3, 1, 2 ]
			\coordinate(0) at(0:0);
			\coordinate(x) at(0:0.66);
			\coordinate(y) at(90:0.66);
            \node at(0.2,0.2) {$\scriptstyle +$};
            \node at(-0.2,-0.2) {$\scriptstyle -$};
			\coordinate(shift21312) at (0);
			\coordinate(v0) at($0*(x)+0*(y)+(shift21312)$); 
			\coordinate(v1) at($1*(x)+0*(y)+(shift21312)$); 
			\coordinate(v2) at($2*(x)+-1*(y)+(shift21312)$); 
			\coordinate(v3) at($1*(x)+-1*(y)+(shift21312)$); 
			\coordinate(v4) at($1*(x)+-2*(y)+(shift21312)$); 
			\coordinate(v5) at($0*(x)+-1*(y)+(shift21312)$); 
			\draw[fill=black] (v0) circle [radius = 0.55mm]; 
			\draw (v0)--(v1); 
			\draw (v0)--(v2); 
			\draw (v0)--(v3); 
			\draw (v0)--(v4); 
			\draw (v0)--(v5); 
			\draw (v1)--(v2)--(v3)--(v4)--(v5); 
			\coordinate(w1) at($1*(x)+0*(y)+(shift21312)$); 
			\coordinate(w2) at($0*(x)+-1*(y)+(shift21312)$); 
			\coordinate(w3) at($-1*(x)+0*(y)+(shift21312)$); 
			\coordinate(w4) at($0*(x)+1*(y)+(shift21312)$); 
			\draw (v0)--(w3); 
			\draw (v0)--(w4);
			\draw (w3)--(w4); 
			\draw (w1)--(w4);
			\draw (w2)--(w3); 
		\end{tikzpicture}\]
\end{example}

\subsection{Proof of Theorem \ref{theorem:sc fan = g fan intro}}

Let $k$ be a field. For a finite dimensional $k$-algebras $\Lambda$ of rank $2$, we regard the $g$-fan $\Sigma(\Lambda)$ as a fan in $\R^2$ by isomorphism $K_0(\proj\Lambda)\simeq\R^2$ given by $P_1\mapsto (1,0)$ and $P_2\mapsto (0,1)$. We denote by
\[\gfan(2)\]
the subset of $\tscfan$ consisting of $g$-fans of finite dimensional $k$-algebras of rank $2$. Let $\gelfan(2)$ be the subset of $\gfan(2)$ consisting of $g$-fans of finite dimensional elementary $k$-algebras of rank $2$.

The following is a main result of this paper. 

\begin{theorem}\label{main theorem}
For any field $k$, we have 
\begin{equation}
    \gelfan(2) = \gfan(2) = \tscfan. 
\end{equation}
That is, any sign-coherent fan in $\mathbb{R}^2$ can be realized as a $g$-fan $\Sigma(\Lambda)$ of some finite dimensional elementary $k$-algebra $\Lambda$.  
\end{theorem}

\begin{proof}
It suffices to show 
$\tscfan=\gelfan(2)$. 
Let 
\begin{equation*}
    \gelfan^{+-}(2):= \gelfan(2)\cap \tscfanpm\ \mbox{ and }\ \gelfan^{-+}(2):= \gelfan(2)\cap \tscfanmp.
\end{equation*} 
By Gluing Theorem \ref{theorem:gluing1}, we have
\begin{equation*}
\gelfan(2)=\gelfan^{+-}(2) \ast \gelfan^{-+}(2).
\end{equation*}
By Rotation Theorem \ref{theorem:rotation}, $\gelfan^{+-}(2)$ is closed under rotations. 
By Theorem \ref{theorem:subdivision} and Proposition \ref{sr sequence}(b), $\gelfan^{+-}(2)$ is closed under subdivisions.
Since $\Sigma(0,0;0,0)=\Sigma(k\times k)\in \gelfan^{+-}(2)$, Proposition \ref{proposition:F=F'} implies
\begin{equation*}
    \gelfan^{+-}(2) = \tscfanpm.  
\end{equation*}
Similarly, we have 
$\gelfan^{-+}(2) =\tscfanmp$. 
Consequently, we have 
\[\tscfan \stackrel{\eqref{eq:gluing fans}}{=} \tscfanpm\ast \tscfanmp =
\gelfan^{+-}(2) \ast \gelfan^{-+}(2) =
\gelfan(2).\qedhere\] 
\end{proof}

For given $\Sigma\in\tscfan$, our proof of Theorem \ref{main theorem} gives a concrete algorithm to construct a finite dimensional $k$-algebra $\Lambda$ satisfying $\Sigma(\Lambda)=\Sigma$. We demonstrate it in the following example.

\begin{example}
	\label{example:rotation-subdivision-gluing} 
	 We construct a finite dimensional $k$-algebra $\Gamma$ satisfying $\Sigma(\Gamma)=\Sigma_{13122;1212}$ by the following three steps.

\begin{enumerate}[\rm(I)]
\item We obtain a finite dimensional $k$-algebra
\[\Lambda=\dfrac{
			k\left[\begin{xy}( 0,0) *+{1}="1",
				( 8,0) *+{2}="2",\ar "1";"2"^{a_3} \ar @(ru,rd)"2";"2"^{a_4}  \ar @(ul,ur)"1";"1"^(.2){a_2} \ar @ (dl,dr)"1";"1"_(.2){a_1}
			\end{xy}\right]}{\langle a_1^2, a_2^2, a_4^2, a_2a_1, a_2a_3-a_3a_4\rangle}\]
satisfying $\Sigma(\Lambda)=\Sigma_{13122;00}$ by using Rotation Theorem \ref{theorem:rotation} and Subdivision Theorem \ref{theorem:subdivision} as follows.
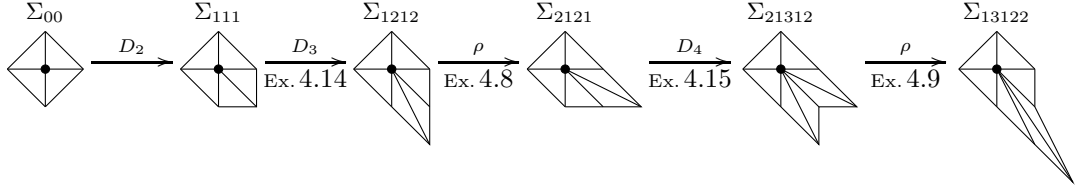
\[		\xymatrix@C3em{\begin{tikzpicture}[baseline=0mm] 
\coordinate(0) at(0:0);
\coordinate(x) at(0:0.5);
\coordinate(y) at(90:0.5);
\coordinate(shift00) at(0);
\coordinate(v0) at($0*(x)+0*(y)+(shift00)$); 
\coordinate(v1) at($1*(x)+0*(y)+(shift00)$); 
\coordinate(v2) at($0*(x)+-1*(y)+(shift00)$); 
\draw[fill=black] (v0) circle [radius = 0.55mm];
\draw (v0)--(v1); 
\draw (v0)--(v2); 
\draw (v1)--(v2); 
\coordinate(L) at($0*(x)+1.5*(y)+(shift00)$); 
\node at(L) {\small$\Sigma_{00}$};
\coordinate(w3) at($-1*(x)+0*(y)+(shift00)$); 
\coordinate(w4) at($0*(x)+1*(y)+(shift00)$); 
\draw (v0)--(w3); 
\draw (v0)--(w4);
\draw (w3)--(w4); 
\draw (v1)--(w4);
\draw (v2)--(w3); 
\end{tikzpicture}
\ar[r]^-{D_2}_-{} & 
\begin{tikzpicture}[baseline=0mm]
\coordinate(0) at(0:0);
\coordinate(x) at(0:0.5);
\coordinate(y) at(90:0.5);
\coordinate(shift111) at (0);
%%[ 1, 1, 1 ]
\coordinate(v0) at($0*(x)+0*(y)+(shift111)$); 
\coordinate(v1) at($1*(x)+0*(y)+(shift111)$); 
\coordinate(v2) at($1*(x)+-1*(y)+(shift111)$); 
\coordinate(v3) at($0*(x)+-1*(y)+(shift111)$); 
\draw[fill=black] (v0) circle [radius = 0.55mm]; 
\draw (v0)--(v1); 
\draw (v0)--(v2); 
\draw (v0)--(v3); 
\draw (v1)--(v2)--(v3); 
\coordinate(L) at($0*(x)+1.5*(y)+(shift111)$); 
\node at(L) {\small$\Sigma_{111}$};
\coordinate(w1) at($1*(x)+0*(y)+(shift111)$); 
\coordinate(w2) at($0*(x)+-1*(y)+(shift111)$); 
\coordinate(w3) at($-1*(x)+0*(y)+(shift111)$); 
\coordinate(w4) at($0*(x)+1*(y)+(shift111)$); 
\draw (v0)--(w3); 
\draw (v0)--(w4);
\draw (w3)--(w4); 
\draw (w1)--(w4);
\draw (w2)--(w3); 
\end{tikzpicture}\ar[r]^-{D_3}_-{{\rm Ex.}\,\ref{example:1212}}&\begin{tikzpicture}[baseline=0mm]
		%%[ 1, 2, 1, 2 ]
		\coordinate(0) at(0:0);
		\coordinate(x) at(0:0.5);
		\coordinate(y) at(90:0.5);
		\coordinate(shift1212) at (0);
		\coordinate(v0) at($0*(x)+0*(y)+(shift1212)$); 
		\coordinate(v1) at($1*(x)+0*(y)+(shift1212)$); 
		\coordinate(v2) at($1*(x)+-1*(y)+(shift1212)$); 
		\coordinate(v3) at($1*(x)+-2*(y)+(shift1212)$); 
		\coordinate(v4) at($0*(x)+-1*(y)+(shift1212)$); 
		\draw[fill=black] (v0) circle [radius = 0.55mm]; 
		\draw (v0)--(v1); 
		\draw (v0)--(v2); 
		\draw (v0)--(v3); 
		\draw (v0)--(v4); 
		\draw (v1)--(v2)--(v3)--(v4); 
		\coordinate(L) at($0*(x)+1.5*(y)+(shift1212)$); 
		\node at(L) {\small$\Sigma_{1212}$};
		\coordinate(w1) at($1*(x)+0*(y)+(shift1212)$); 
		\coordinate(w2) at($0*(x)+-1*(y)+(shift1212)$); 
		\coordinate(w3) at($-1*(x)+0*(y)+(shift1212)$); 
		\coordinate(w4) at($0*(x)+1*(y)+(shift1212)$); 
		\draw (v0)--(w3); 
		\draw (v0)--(w4);
		\draw (w3)--(w4); 
		\draw (w1)--(w4);
		\draw (w2)--(w3); 
	\end{tikzpicture}
\ar[r]^-{\rho}_-{{\rm Ex.}\,\ref{example:2121}}&
\begin{tikzpicture}[baseline=0mm]
	%%[ 2, 1, 2, 1 ]
	\coordinate(0) at(0:0);
	\coordinate(x) at(0:0.5);
	\coordinate(y) at(90:0.5);
	\coordinate(shift2121) at (0);
	\coordinate(v0) at($0*(x)+0*(y)+(shift2121)$); 
	\coordinate(v1) at($1*(x)+0*(y)+(shift2121)$); 
	\coordinate(v2) at($2*(x)+-1*(y)+(shift2121)$); 
	\coordinate(v3) at($1*(x)+-1*(y)+(shift2121)$); 
	\coordinate(v4) at($0*(x)+-1*(y)+(shift2121)$); 
	\draw[fill=black] (v0) circle [radius = 0.55mm]; 
	\draw (v0)--(v1); 
	\draw (v0)--(v2); 
	\draw (v0)--(v3); 
	\draw (v0)--(v4); 
	\draw (v1)--(v2)--(v3)--(v4); 
	\coordinate(L) at($0*(x)+1.5*(y)+(shift2121)$); 
	\node at(L) {\small$\Sigma_{2121}$};
	\coordinate(w1) at($1*(x)+0*(y)+(shift2121)$); 
	\coordinate(w2) at($0*(x)+-1*(y)+(shift2121)$); 
	\coordinate(w3) at($-1*(x)+0*(y)+(shift2121)$); 
	\coordinate(w4) at($0*(x)+1*(y)+(shift2121)$); 
	\draw (v0)--(w3); 
	\draw (v0)--(w4);
	\draw (w3)--(w4); 
	\draw (w1)--(w4);
	\draw (w2)--(w3); 
\end{tikzpicture}
		\ar[r]^-{D_4}_-{{\rm Ex.}\,\ref{example:21312}}&
		\begin{tikzpicture}[baseline=0mm]
			%%[ 2, 1, 3, 1, 2 ]
			\coordinate(0) at(0:0);
			\coordinate(x) at(0:0.5);
			\coordinate(y) at(90:0.5);
			\coordinate(shift21312) at (0);
			\coordinate(v0) at($0*(x)+0*(y)+(shift21312)$); 
			\coordinate(v1) at($1*(x)+0*(y)+(shift21312)$); 
			\coordinate(v2) at($2*(x)+-1*(y)+(shift21312)$); 
			\coordinate(v3) at($1*(x)+-1*(y)+(shift21312)$); 
			\coordinate(v4) at($1*(x)+-2*(y)+(shift21312)$); 
			\coordinate(v5) at($0*(x)+-1*(y)+(shift21312)$); 
			\draw[fill=black] (v0) circle [radius = 0.55mm]; 
			\draw (v0)--(v1); 
			\draw (v0)--(v2); 
			\draw (v0)--(v3); 
			\draw (v0)--(v4); 
			\draw (v0)--(v5); 
			\draw (v1)--(v2)--(v3)--(v4)--(v5); 
			\coordinate(L) at($0*(x)+1.5*(y)+(shift21312)$); 
			\node at(L) {\small$\Sigma_{21312}$};
			\coordinate(w1) at($1*(x)+0*(y)+(shift21312)$); 
			\coordinate(w2) at($0*(x)+-1*(y)+(shift21312)$); 
			\coordinate(w3) at($-1*(x)+0*(y)+(shift21312)$); 
			\coordinate(w4) at($0*(x)+1*(y)+(shift21312)$); 
			\draw (v0)--(w3); 
			\draw (v0)--(w4);
			\draw (w3)--(w4); 
			\draw (w1)--(w4);
			\draw (w2)--(w3); 
		\end{tikzpicture}
	\ar[r]^-{\rho}_-{{\rm Ex.}\,\ref{example:13122}}&
	\begin{tikzpicture}[baseline=0mm]
		%%[ 1, 3, 1, 2, 2 ]
		\coordinate(0) at(0:0);
		\coordinate(x) at(0:0.5);
		\coordinate(y) at(90:0.5);
		\coordinate(shift13122) at (0);
		\coordinate(v0) at($0*(x)+0*(y)+(shift13122)$); 
		\coordinate(v1) at($1*(x)+0*(y)+(shift13122)$); 
		\coordinate(v2) at($1*(x)+-1*(y)+(shift13122)$); 
		\coordinate(v3) at($2*(x)+-3*(y)+(shift13122)$); 
		\coordinate(v4) at($1*(x)+-2*(y)+(shift13122)$); 
		\coordinate(v5) at($0*(x)+-1*(y)+(shift13122)$); 
		\draw[fill=black] (v0) circle [radius = 0.55mm]; 
		\draw (v0)--(v1); 
		\draw (v0)--(v2); 
		\draw (v0)--(v3); 
		\draw (v0)--(v4); 
		\draw (v0)--(v5); 
		\draw (v1)--(v2)--(v3)--(v4)--(v5); 
		\coordinate(L) at($0*(x)+1.5*(y)+(shift13122)$); 
		\node at(L) {\small$\Sigma_{13122}$};
		\coordinate(w1) at($1*(x)+0*(y)+(shift13122)$); 
		\coordinate(w2) at($0*(x)+-1*(y)+(shift13122)$); 
		\coordinate(w3) at($-1*(x)+0*(y)+(shift13122)$); 
		\coordinate(w4) at($0*(x)+1*(y)+(shift13122)$); 
		\draw (v0)--(w3); 
		\draw (v0)--(w4);
		\draw (w3)--(w4); 
		\draw (w1)--(w4);
		\draw (w2)--(w3); 
	\end{tikzpicture}
		}
	\]
\item Similarly, we obtain a finite dimensional $k$-algebra
\[\Lambda':=\dfrac{
	k\left[\begin{xy}( 0,0) *+{1}="1",
		( 8,0) *+{2}="2",\ar "2";"1"_{b_1}  \ar @(ru,rd)"2";"2"^{b_2}  
	\end{xy}\right]}{\langle b_2^2\rangle}\]
satisfying $\Sigma(\Lambda')=\Sigma(0,0;1,2,1,2)$. 
\item We obtain a finite dimensional $k$-algebra
\[\Gamma=\dfrac{
	k\left[\begin{xy}( 0,0) *+{1}="1",
		( 10,0) *+{2}="2",\ar@<1.5pt> "1";"2"^{a_3}  \ar @(dr,dl)"2";"2"^(.2){a_4}  \ar @(ul,ur)"1";"1"^(.2){a_2} \ar @ (dl,dr)"1";"1"_(.2){a_1} \ar@<1.5pt> "2";"1"^{b_1}  \ar @(ur,ul)"2";"2"_(.2){b_2}
	\end{xy}\right]}{\langle a_1^2, a_2^2, a_4^2, a_2a_1,a_2a_3-a_3a_4, b_2^2\rangle+\langle a_ib_j,b_ja_i\mid i\in \{1,2,3,4\},j\in\{1,2\}\rangle}\]
satisfying $\Sigma(\Gamma)=\Sigma(1,3,1,2,2;1,2,1,2)$ by applying Gluing Theorem \ref{theorem:gluing1} to $\Lambda$ and $\Lambda'$, see Example \ref{example:gluing}.
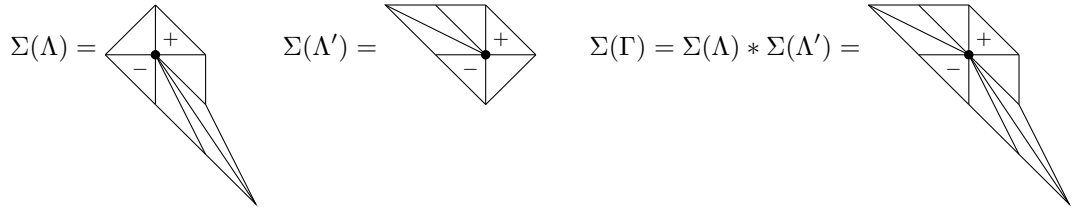
	 	\[
			 	\Sigma(\Lambda)=\begin{tikzpicture}[baseline=0mm]
	 		%%[ 1, 3, 1, 2, 2 ]
	 		\coordinate(0) at(0:0);
	 		\coordinate(x) at(0:0.66);
	 		\coordinate(y) at(90:0.66);
            \node at(0.2,0.2) {$\scriptstyle +$};
            \node at(-0.2,-0.2) {$\scriptstyle -$};
	 		\coordinate(shift13122) at (0);
	 		\coordinate(v0) at($0*(x)+0*(y)+(shift13122)$); 
	 		\coordinate(v1) at($1*(x)+0*(y)+(shift13122)$); 
	 		\coordinate(v2) at($1*(x)+-1*(y)+(shift13122)$); 
	 		\coordinate(v3) at($2*(x)+-3*(y)+(shift13122)$); 
	 		\coordinate(v4) at($1*(x)+-2*(y)+(shift13122)$); 
	 		\coordinate(v5) at($0*(x)+-1*(y)+(shift13122)$); 
	 		\draw[fill=black] (v0) circle [radius = 0.55mm]; 
	 		\draw (v0)--(v1); 
	 		\draw (v0)--(v2); 
	 		\draw (v0)--(v3); 
	 		\draw (v0)--(v4); 
	 		\draw (v0)--(v5); 
	 		\draw (v1)--(v2)--(v3)--(v4)--(v5); 
	 		\coordinate(L) at($0*(x)+1.5*(y)+(shift13122)$); 
	 		%\node at(L) {\small$\Sigma(1,3,1,2,2;0,0)$};
	 		\coordinate(w1) at($1*(x)+0*(y)+(shift13122)$); 
	 		\coordinate(w2) at($0*(x)+-1*(y)+(shift13122)$); 
	 		\coordinate(w3) at($-1*(x)+0*(y)+(shift13122)$); 
	 		\coordinate(w4) at($0*(x)+1*(y)+(shift13122)$); 
	 		\draw (v0)--(w3); 
	 		\draw (v0)--(w4);
	 		\draw (w3)--(w4); 
	 		\draw (w1)--(w4);
	 		\draw (w2)--(w3); 
	 	\end{tikzpicture}\quad 
			\Sigma(\Lambda')=\begin{tikzpicture}[baseline=0mm]
	 		%%[ 1, 2, 1, 2 ]
	 		\coordinate(0) at(0:0);
	 		\coordinate(x) at(0:0.66);
	 		\coordinate(y) at(90:0.66);
            \node at(0.2,0.2) {$\scriptstyle +$};
            \node at(-0.2,-0.2) {$\scriptstyle -$};
	 		\coordinate(shift1212) at (0);
	 		\coordinate(v0) at($0*(x)+0*(y)+(shift1212)$); 
	 		\coordinate(v1) at($0*(x)+1*(y)+(shift1212)$); 
	 		\coordinate(v2) at($-1*(x)+1*(y)+(shift1212)$); 
	 		\coordinate(v3) at($-2*(x)+1*(y)+(shift1212)$); 
	 		\coordinate(v4) at($-1*(x)+0*(y)+(shift1212)$); 
	 		\draw[fill=black] (v0) circle [radius = 0.55mm]; 
	 		\draw (v0)--(v1); 
	 		\draw (v0)--(v2); 
	 		\draw (v0)--(v3); 
	 		\draw (v0)--(v4); 
	 		\draw (v1)--(v2)--(v3)--(v4); 
	 		\coordinate(L) at($0*(x)+1.5*(y)+(shift1212)$); 
	 		%\node at(L) {\small$\Sigma(0,0;1,2,1,2)$};
	 		\coordinate(w1) at($0*(x)+1*(y)+(shift1212)$); 
	 		\coordinate(w2) at($-1*(x)+0*(y)+(shift1212)$); 
	 		\coordinate(w3) at($0*(x)+-1*(y)+(shift1212)$); 
	 		\coordinate(w4) at($1*(x)+0*(y)+(shift1212)$); 
	 		\draw (v0)--(w3); 
	 		\draw (v0)--(w4);
	 		\draw (w3)--(w4); 
	 		\draw (w1)--(w4);
	 		\draw (w2)--(w3); 
	 	\end{tikzpicture}
	 	\quad \quad \Sigma(\Gamma)=\Sigma(\Lambda)\ast \Sigma(\Lambda')=\begin{tikzpicture}[baseline=0mm]
	 		%%[ 1, 3, 1, 2, 2 ]
	 		\coordinate(0) at(0:0);
	 		\coordinate(x) at(0:0.66);
	 		\coordinate(y) at(90:0.66);
            \node at(0.2,0.2) {$\scriptstyle +$};
            \node at(-0.2,-0.2) {$\scriptstyle -$};
	 		\coordinate(shift13122) at (0);
	 		\coordinate(v0) at($0*(x)+0*(y)+(shift13122)$); 
	 		\coordinate(v1) at($1*(x)+0*(y)+(shift13122)$); 
	 		\coordinate(v2) at($1*(x)+-1*(y)+(shift13122)$); 
	 		\coordinate(v3) at($2*(x)+-3*(y)+(shift13122)$); 
	 		\coordinate(v4) at($1*(x)+-2*(y)+(shift13122)$); 
	 		\coordinate(v5) at($0*(x)+-1*(y)+(shift13122)$); 
	 		\draw[fill=black] (v0) circle [radius = 0.55mm]; 
	 		\draw (v0)--(v1); 
	 		\draw (v0)--(v2); 
	 		\draw (v0)--(v3); 
	 		\draw (v0)--(v4); 
	 		\draw (v0)--(v5); 
	 		\draw (v1)--(v2)--(v3)--(v4)--(v5); 
	 		\coordinate(L) at($0*(x)+1.5*(y)+(shift13122)$); 
	 		%\node at(L) {\small$\Sigma(1,3,1,2,2;1,2,1,2)$};
	 		\coordinate(w1) at($1*(x)+0*(y)+(shift13122)$);
	 		\coordinate(w2) at($0*(x)+-1*(y)+(shift13122)$); 
	 		\coordinate(w3) at($-1*(x)+0*(y)+(shift13122)$);
	 		\coordinate(w4) at($-2*(x)+1*(y)+(shift13122)$);
	 		\coordinate(w5) at($-1*(x)+1*(y)+(shift13122)$); 
	 		\coordinate(w6) at($0*(x)+1*(y)+(shift13122)$); 
	 		\draw (v0)--(w3);
	 		\draw (v0)--(w6); 
	 		\draw (v0)--(w4);
	 		\draw (v0)--(w5);
	 		\draw (w2)--(w3); 
	 		\draw (w3)--(w4); 
	 		\draw (w4)--(w5);
	 		\draw (w5)--(w6);
	 		\draw (w6)--(v1); 
	 	\end{tikzpicture}
	 	\]
\end{enumerate}
\end{example}

\begin{example}
	\label{example:gluing} 
Let $\Lambda$ and $\Lambda'$ be the following algebras.
\[
\Lambda:=\dfrac{
	k\left[\begin{xy}( 0,0) *+{1}="1",
		( 8,0) *+{2}="2",\ar "1";"2"^{a_3}  \ar @(ru,rd)"2";"2"^{a_4}  \ar @(ul,ur)"1";"1"^(0.2){a_2} \ar @ (dl,dr)"1";"1"_(.2){a_1}
	\end{xy}\right]}{\langle a_1^2, a_2^2, a_4^2, a_2a_1,a_2a_3-a_3a_4\rangle},\hspace{10pt}
\Lambda':=\dfrac{
	k\left[\begin{xy}( 0,0) *+{1}="1",
		( 8,0) *+{2}="2",\ar "2";"1"_{b_1}  \ar @(ru,rd)"2";"2"^{b_2}  
	\end{xy}\right]}{\langle b_2^2\rangle}
\]
By Examples \ref{example:13122} and \ref{example:1212} below, we have 
\[
\Sigma(\Lambda)=\Sigma_{13122;00}=\begin{tikzpicture}[baseline=0mm]
	%%[ 1, 3, 1, 2, 2 ]
	\coordinate(0) at(0:0);
	\coordinate(x) at(0:0.66);
	\coordinate(y) at(90:0.66);
    \node at(0.2,0.2) {$\scriptstyle +$};
    \node at(-0.2,-0.2) {$\scriptstyle -$};
	\coordinate(shift13122) at (0);
	\coordinate(v0) at($0*(x)+0*(y)+(shift13122)$); 
	\coordinate(v1) at($1*(x)+0*(y)+(shift13122)$); 
	\coordinate(v2) at($1*(x)+-1*(y)+(shift13122)$); 
	\coordinate(v3) at($2*(x)+-3*(y)+(shift13122)$); 
	\coordinate(v4) at($1*(x)+-2*(y)+(shift13122)$); 
	\coordinate(v5) at($0*(x)+-1*(y)+(shift13122)$); 
	\draw[fill=black] (v0) circle [radius = 0.55mm]; 
	\draw (v0)--(v1); 
	\draw (v0)--(v2); 
	\draw (v0)--(v3); 
	\draw (v0)--(v4); 
	\draw (v0)--(v5); 
	\draw (v1)--(v2)--(v3)--(v4)--(v5); 
	\coordinate(w1) at($1*(x)+0*(y)+(shift13122)$); 
	\coordinate(w2) at($0*(x)+-1*(y)+(shift13122)$); 
	\coordinate(w3) at($-1*(x)+0*(y)+(shift13122)$); 
	\coordinate(w4) at($0*(x)+1*(y)+(shift13122)$); 
	\draw (v0)--(w3); 
	\draw (v0)--(w4);
	\draw (w3)--(w4); 
	\draw (w1)--(w4);
	\draw (w2)--(w3); 
\end{tikzpicture}
\
\Sigma(\Lambda')=\Sigma_{00;1212}=\begin{tikzpicture}[baseline=0mm]
	%%[ 1, 2, 1, 2 ]
	\coordinate(0) at(0:0);
	\coordinate(x) at(0:0.66);
	\coordinate(y) at(90:0.66);
    \node at(0.2,0.2) {$\scriptstyle +$};
    \node at(-0.2,-0.2) {$\scriptstyle -$};
	\coordinate(shift1212) at (0);
	\coordinate(v0) at($0*(x)+0*(y)+(shift1212)$); 
	\coordinate(v1) at($0*(x)+1*(y)+(shift1212)$); 
	\coordinate(v2) at($-1*(x)+1*(y)+(shift1212)$); 
	\coordinate(v3) at($-2*(x)+1*(y)+(shift1212)$); 
	\coordinate(v4) at($-1*(x)+0*(y)+(shift1212)$); 
	\draw[fill=black] (v0) circle [radius = 0.55mm]; 
	\draw (v0)--(v1); 
	\draw (v0)--(v2); 
	\draw (v0)--(v3); 
	\draw (v0)--(v4); 
	\draw (v1)--(v2)--(v3)--(v4); 
	\coordinate(w1) at($0*(x)+1*(y)+(shift1212)$); 
	\coordinate(w2) at($-1*(x)+0*(y)+(shift1212)$); 
	\coordinate(w3) at($0*(x)+-1*(y)+(shift1212)$); 
	\coordinate(w4) at($1*(x)+0*(y)+(shift1212)$); 
	\draw (v0)--(w3); 
	\draw (v0)--(w4);
	\draw (w3)--(w4); 
	\draw (w1)--(w4);
	\draw (w2)--(w3); 
\end{tikzpicture}
\]
%Let $A=e_1 \Lambda e_1$, $X=e_1 \Lambda e_2$, $B=e_2 \Lambda e_2$, $C=e_1 \Lambda' e_1$, $Y=e_2 \Lambda' e_1$, $D=e_2 \Lambda'e_2$ and $\Gamma=\left[\begin{array}{cc}	A\times_kC&X\\ Y&B\times_kD\end{array}\right]$. 
On the other hand, we have
\[
\Lambda\ast\Lambda'= \dfrac{
k\left[\begin{xy}( 0,0) *+{1}="1",
	( 10,0) *+{2}="2",\ar@<1.5pt> "1";"2"^{a_3}  \ar @(dr,dl)"2";"2"^(.2){a_4}  \ar @(ul,ur)"1";"1"^(.2){a_2} \ar @ (dl,dr)"1";"1"_(.2){a_1} \ar@<1.5pt> "2";"1"^{b_1}  \ar @(ur,ul)"2";"2"_(.2){b_2}
\end{xy}\right]}{\langle a_1^2, a_2^2, a_4^2, a_2a_1,a_2a_3-a_3a_4, b_2^2\rangle+\langle a_ib_j,b_ja_i\mid i\in \{1,2,3,4\},j\in\{1,2\}\rangle}
\]
By Gluing Theorem \ref{theorem:gluing1}, we have
\[
\Sigma(\Lambda\ast\Lambda')=\Sigma(\Lambda)\ast\Sigma(\Lambda')=\Sigma_{13122;1212}=
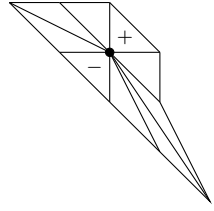
\begin{tikzpicture}[baseline=0mm]
	\coordinate(0) at(0:0);
	\coordinate(x) at(0:0.66);
	\coordinate(y) at(90:0.66);
    \node at(0.2,0.2) {$\scriptstyle +$};
    \node at(-0.2,-0.2) {$\scriptstyle -$};
	\coordinate(shift13122) at (0);
	\coordinate(v0) at($0*(x)+0*(y)+(shift13122)$); 
	\coordinate(v1) at($1*(x)+0*(y)+(shift13122)$); 
	\coordinate(v2) at($1*(x)+-1*(y)+(shift13122)$); 
	\coordinate(v3) at($2*(x)+-3*(y)+(shift13122)$); 
	\coordinate(v4) at($1*(x)+-2*(y)+(shift13122)$); 
	\coordinate(v5) at($0*(x)+-1*(y)+(shift13122)$); 
	\draw[fill=black] (v0) circle [radius = 0.55mm]; 
	\draw (v0)--(v1); 
	\draw (v0)--(v2); 
	\draw (v0)--(v3); 
	\draw (v0)--(v4); 
	\draw (v0)--(v5); 
	\draw (v1)--(v2)--(v3)--(v4)--(v5); 
	\coordinate(w1) at($1*(x)+0*(y)+(shift13122)$);
	\coordinate(w2) at($0*(x)+-1*(y)+(shift13122)$); 
	\coordinate(w3) at($-1*(x)+0*(y)+(shift13122)$);
	\coordinate(w4) at($-2*(x)+1*(y)+(shift13122)$);
	\coordinate(w5) at($-1*(x)+1*(y)+(shift13122)$); 
	\coordinate(w6) at($0*(x)+1*(y)+(shift13122)$); 
	\draw (v0)--(w3);
	\draw (v0)--(w6); 
	\draw (v0)--(w4);
	\draw (v0)--(w5);
	\draw (w2)--(w3); 
	\draw (w3)--(w4); 
	\draw (w4)--(w5);
	\draw (w5)--(w6);
	\draw (w6)--(v1); 
\end{tikzpicture}
\]
\end{example}

\section{$g$-Convex algebras of rank 2}
In this section, we will characterize algebras of rank 2 which have convex $g$-polygons. 

\subsection{Characterizations of $g$-convex algebras of rank 2}
Let $e,e'$ be pairwise orthogonal primitive idempotents in $A$ and $x\in e A e'$. Then we use the following notations.
\begin{enumerate}[$\bullet$]
	\item $x\in eAe'$ is a \emph{left generator} (respectively, \emph{right generator}) of $eAe'$ if $eAx=eAe'$ (respectively, $xAe'=eAe'$). 
	\item Define subalgebras $L_x\subset eAe$ and $R_x\subset e'Ae'$ as follows (see Lemma \ref{lemma:l and r are subalgebras}).
		\[
		L_x:=\{a\in e A e\mid ax\in x A e'\}\ \mbox{ and }\ R_x:=\{a\in e' A e'\mid xa\in e A x\}.
		\]
\end{enumerate}

Recall that, for a ring $\Lambda$ and a right (respectively, left) $\Lambda$-module $M$, we denote by $t(M)_\Lambda$ (respectively, $t_\Lambda(M)$) the minimal number of generators of $M$.

\begin{theorem}
	\label{theorem:charactarization_g-convex_ranktwo}
	Let $A$ be a basic finite dimensional algebra, $\{e_1, e_2\}$ a complete set of primitive orthogonal idempotents in $A$, and $P_i=e_i A$ $(i=1,2)$.
	\begin{enumerate}[\rm(a)]
		\item $A$ is $g$-convex if and only if 
			\[\Sigma(A)=\Sigma_{a;b}\ \mbox{ for some }\ a,b\in \{(0,0),(1,1,1),(1,2,1,2),(2,1,2,1)\}.\]
		\item Let $(l,r):=\left(t_{e_1Ae_1}(e_1Ae_2),t(e_1Ae_2)_{e_2Ae_2}\right)$. Then we have the following statements.
		\begin{enumerate}[$\bullet$]
			\item $\Sigma(A)=\Sigma_{00}\ast\Sigma$ for some $\Sigma\in\itscfanmp$ if and only if $(l,r)=(0,0)$.
			\item $\Sigma(A)=\Sigma_{111}\ast\Sigma$ for some $\Sigma\in\itscfanmp$ if and only if $(l,r)=(1,1)$.
			\item $\Sigma(A)=\Sigma_{1212}\ast\Sigma$ for some $\Sigma\in\itscfanmp$ if and only if $(l,r)=(1,2)$ and $t_{L_x}(e_1Ae_1)=2$ hold for some left generator $x$ of $e_1 A e_2$ and $L_x:=\{a\in e_1 A e_1\mid ax\in x A e_2\}$. 
				\item $\Sigma(A)=\Sigma_{2121}\ast\Sigma$ for some $\Sigma\in\itscfanmp$ if and only if $(l,r)=(2,1)$ and $t(e_2Ae_2)_{R_x}= 2$ hold for some right generator $x$ of $e_1 A e_2$ and $R_x:=\{b\in e_2 A e_2\mid xb\in e_1 A x\}$.
		\end{enumerate}
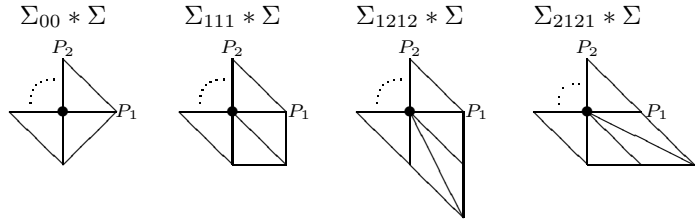
		\[
		%\begin{tabular}{|c|c|c|c|}
		%\hline {$\Sigma_{0,0;-}$}&{$\Sigma_{1,1,1;-}$}&{$\Sigma_{1,2,1,2;-}$}&{$\Sigma_{2,1,2,1;-}$}\\
		{\begin{xy}
				0;<4pt,0pt>:<0pt,4pt>::
				( 0,9) *+{\Sigma_{00}\ast\Sigma},
				(  0,0) *{\bullet}="(0,0)",
				( 0,5) ="(0,10)",
				( 0,6) *+{{\scriptstyle P_2}},
				( 5,0) ="(10,0)",
				( 6,0) *+{{\scriptstyle P_1}},
				%(-20,10) *+{\bullet}="(-20,10)",
				( 0,-5) ="(0,-10)",
				(  -5,0) ="(-10,0)",
				%(20,-10) *{\bullet}="(20,-10)",
				%(10,-10) *{\bullet}="(10,-10)",
				%(10,-20) *{\bullet}="(10,-20)",
				(0,3)*{}="(0,5)",
				(-3,0)*{}="(-5,0)",
				%(-10,10) *+{\bullet}="(-10,10)",
				\ar@{-} "(10,0)";"(-10,0)"
				\ar@{-} "(0,10)";"(0,-10)"
				%\ar@{-} "(0,0)";"(20,-10)"
				%\ar@{-} "(0,0)";"(10,-10)"
				%\ar@{-} "(0,0)";"(10,-20)"
				%\ar@{-} "(-20,10)";"(20,-10)"
				%\ar@{-} "(-10,10)";"(10,-10)"
				\ar@{-} "(10,0)";"(0,10)"
				%\ar@{-} "(0,10)";"(-20,10)"
				%\ar@{-} "(-20,10)";"(-10,0)"
				\ar@{-} "(-10,0)";"(0,-10)"
				%\ar@{-} "(0,-10)";"(20,-10)"
				\ar@{-} "(10,0)";"(0,-10)"
				%\ar@{-} "(10,-10)";"(0,-10)"
				%\ar@{-} "(10,-20)";"(0,-10)"
				\ar@/^-2mm/@{.} "(0,5)";"(-5,0)"
		\end{xy}}
		\ \ \ {\begin{xy}
				0;<2pt,0pt>:<0pt,2pt>::
				( 0,18) *+{\Sigma_{111}\ast\Sigma},
				(  0,0) *{\bullet},
				(  0,0) ="(0,0)",
				( 0,10) ="(0,10)",
				( 0,12) *+{{\scriptstyle P_2}},
				( 10,0) ="(10,0)",
				( 12.5,0) *+{{\scriptstyle P_1}},
				%(-20,10) *+{\bullet}="(-20,10)",
				( 0,-10)="(0,-10)",
				(  -10,0)="(-10,0)",
				%(20,-10) *{\bullet}="(20,-10)",
				(10,-10)="(10,-10)",
				%(10,-20) *{\bullet}="(10,-20)",
				(0,6)*{}="(0,5)",
				(-6,0)*{}="(-5,0)",
				%(-10,10) *+{\bullet}="(-10,10)",
				\ar@{-} "(10,0)";"(-10,0)"
				\ar@{-} "(0,10)";"(0,-10)"
				%\ar@{-} "(0,0)";"(20,-10)"
				\ar@{-} "(0,0)";"(10,-10)"
				%\ar@{-} "(0,0)";"(10,-20)"
				%\ar@{-} "(-20,10)";"(20,-10)"
				%\ar@{-} "(-10,10)";"(10,-10)"
				\ar@{-} "(10,0)";"(0,10)"
				%\ar@{-} "(0,10)";"(-20,10)"
				%\ar@{-} "(-20,10)";"(-10,0)"
				\ar@{-} "(-10,0)";"(0,-10)"
				%\ar@{-} "(0,-10)";"(20,-10)"
				\ar@{-} "(10,0)";"(10,-10)"
				\ar@{-} "(10,-10)";"(0,-10)"
				%\ar@{-} "(10,-20)";"(0,-10)"
				\ar@/^-2mm/@{.} "(0,5)";"(-5,0)"
		\end{xy}}
		\ \ \ {\begin{xy}
				0;<2pt,0pt>:<0pt,2pt>::
				( 0,18) *+{\Sigma_{1212}\ast\Sigma},
				(  0,0) *{\bullet},
				(  0,0) ="(0,0)",
				( 0,10) ="(0,10)",
				( 0,12) *+{{\scriptstyle P_2}},
				( 10,0) ="(10,0)",
				( 12.5,0) *+{{\scriptstyle P_1}},
				%(-20,10) *+{\bullet}="(-20,10)",
				( 0,-10)="(0,-10)",
				(  -10,0)="(-10,0)",
				%(20,-10) *{\bullet}="(20,-10)",
				(10,-10)="(10,-10)",
				(10,-20)="(10,-20)",
				(0,6)*{}="(0,5)",
				(-6,0)*{}="(-5,0)",
				%(-10,10) *+{\bullet}="(-10,10)",
				\ar@{-} "(10,0)";"(-10,0)"
				\ar@{-} "(0,10)";"(0,-10)"
				%\ar@{-} "(0,0)";"(20,-10)"
				\ar@{-} "(0,0)";"(10,-10)"
				\ar@{-} "(0,0)";"(10,-20)"
				%\ar@{-} "(-20,10)";"(20,-10)"
				%\ar@{-} "(-10,10)";"(10,-10)"
				\ar@{-} "(10,0)";"(0,10)"
				%\ar@{-} "(0,10)";"(-20,10)"
				%\ar@{-} "(-20,10)";"(-10,0)"
				\ar@{-} "(-10,0)";"(0,-10)"
				%\ar@{-} "(0,-10)";"(20,-10)"
				\ar@{-} "(10,0)";"(10,-10)"
				\ar@{-} "(10,-10)";"(10,-20)"
				\ar@{-} "(10,-20)";"(0,-10)"
				\ar@/^-2mm/@{.} "(0,5)";"(-5,0)"
			\end{xy}
		}\ \ \ {\begin{xy}
				0;<2pt,0pt>:<0pt,2pt>::
				( 0,18) *+{\Sigma_{2121}\ast\Sigma},
				(  0,0) *{\bullet},
				(  0,0) ="(0,0)",
				( 0,10) ="(0,10)",
				( 0,12) *+{{\scriptstyle P_2}},
				( 10,0) ="(10,0)",
				( 13,0) *+{{\scriptstyle P_1}},
				%(-20,10) *+{\bullet}="(-20,10)",
				( 0,-10)="(0,-10)",
				(  -10,0)="(-10,0)",
				(20,-10)="(20,-10)",
				(10,-10)="(10,-10)",
				(0,5)*{}="(0,5)",
				(-5,0)*{}="(-5,0)",
				%(-10,10) *+{\bullet}="(-10,10)",
				\ar@{-} "(10,0)";"(-10,0)"
				\ar@{-} "(0,10)";"(0,-10)"
				\ar@{-} "(0,0)";"(20,-10)"
				\ar@{-} "(0,0)";"(10,-10)"
				%\ar@{-} "(-20,10)";"(20,-10)"
				%\ar@{-} "(-10,10)";"(10,-10)"
				\ar@{-} "(10,0)";"(0,10)"
				%\ar@{-} "(0,10)";"(-20,10)"
				%\ar@{-} "(-20,10)";"(-10,0)"
				\ar@{-} "(-10,0)";"(0,-10)"
				\ar@{-} "(0,-10)";"(20,-10)"
				\ar@{-} "(20,-10)";"(10,0)"
				\ar@/^-2mm/@{.} "(0,5)";"(-5,0)"
		\end{xy}}
		%\hline
		%\end{tabular}
		\]
		\item Let $(l,r):=\left(t_{e_2Ae_2}(e_2Ae_1),t(e_2Ae_1)_{e_1Ae_1}\right)$. Then we have the following statements.
		\begin{enumerate}[$\bullet$]
			\item $\Sigma(A)=\Sigma\ast{}^t\Sigma_{00}$ for some $\Sigma\in\itscfanpm$ if and only if $(l,r)=(0,0)$.
			\item $\Sigma(A)=\Sigma\ast{}^t\Sigma_{111}$ for some $\Sigma\in\itscfanpm$ if and only if $(l,r)=(1,1)$.
			\item $\Sigma(A)=\Sigma\ast{}^t\Sigma_{1212}$ for some $\Sigma\in\itscfanpm$ if and only if $(l,r)=(1,2)$ and $t_{L_x}(e_2Ae_2)=2$ hold for some left generator $x$ of $e_2 A e_1$ and $L_x:=\{a\in e_2 A e_2\mid ax\in x A e_1\}$. 
				\item $\Sigma(A)=\Sigma\ast{}^t\Sigma_{2121}$ for some $\Sigma\in\itscfanpm$ if and only if $(l,r)=(2,1)$ and $t(e_1Ae_1)_{R_x}= 2$ hold for some right generator $x$ of $e_2 A e_1$ and $R_x:=\{b\in e_1 A e_1\mid xb\in e_2 A x\}$.
		\end{enumerate}
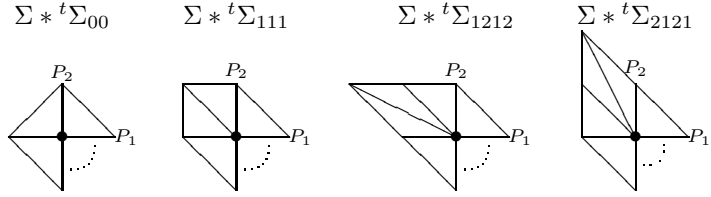
		\[
		%\begin{tabular}{|c|c|c|c|}
		%\hline {$\Sigma_{0,0;-}$}&{$\Sigma_{1,1,1;-}$}&{$\Sigma_{1,2,1,2;-}$}&{$\Sigma_{2,1,2,1;-}$}\\
		{\begin{xy}
				0;<0pt,4pt>:<4pt,0pt>::
				( 11.5,0) *+{\Sigma\ast{}^t\Sigma_{00}},
				(  0,0) *{\bullet}="(0,0)",
				( 0,5) ="(0,10)",
				( 0,6) *+{{\scriptstyle P_1}},
				( 5,0) ="(10,0)",
				( 6,0) *+{{\scriptstyle P_2}},
				%(-20,10) *+{\bullet}="(-20,10)",
				( 0,-5) ="(0,-10)",
				(  -5,0) ="(-10,0)",
				%(20,-10) *{\bullet}="(20,-10)",
				%(10,-10) *{\bullet}="(10,-10)",
				%(10,-20) *{\bullet}="(10,-20)",
				(0,3)*{}="(0,5)",
				(-3,0)*{}="(-5,0)",
				%(-10,10) *+{\bullet}="(-10,10)",
				\ar@{-} "(10,0)";"(-10,0)"
				\ar@{-} "(0,10)";"(0,-10)"
				%\ar@{-} "(0,0)";"(20,-10)"
				%\ar@{-} "(0,0)";"(10,-10)"
				%\ar@{-} "(0,0)";"(10,-20)"
				%\ar@{-} "(-20,10)";"(20,-10)"
				%\ar@{-} "(-10,10)";"(10,-10)"
				\ar@{-} "(10,0)";"(0,10)"
				%\ar@{-} "(0,10)";"(-20,10)"
				%\ar@{-} "(-20,10)";"(-10,0)"
				\ar@{-} "(-10,0)";"(0,-10)"
				%\ar@{-} "(0,-10)";"(20,-10)"
				\ar@{-} "(10,0)";"(0,-10)"
				%\ar@{-} "(10,-10)";"(0,-10)"
				%\ar@{-} "(10,-20)";"(0,-10)"
				\ar@/^2mm/@{.} "(0,5)";"(-5,0)"
		\end{xy}}
		\ \ \ {\begin{xy}
				0;<0pt,2pt>:<2pt,0pt>::
				( 23,0) *+{\Sigma\ast{}^t\Sigma_{111}},
				(  0,0) *{\bullet},
				(  0,0) ="(0,0)",
				( 0,10) ="(0,10)",
				( 0,12) *+{{\scriptstyle P_1}},
				( 10,0) ="(10,0)",
				( 12.5,0) *+{{\scriptstyle P_2}},
				%(-20,10) *+{\bullet}="(-20,10)",
				( 0,-10)="(0,-10)",
				(  -10,0)="(-10,0)",
				%(20,-10) *{\bullet}="(20,-10)",
				(10,-10)="(10,-10)",
				%(10,-20) *{\bullet}="(10,-20)",
				(0,6)*{}="(0,5)",
				(-6,0)*{}="(-5,0)",
				%(-10,10) *+{\bullet}="(-10,10)",
				\ar@{-} "(10,0)";"(-10,0)"
				\ar@{-} "(0,10)";"(0,-10)"
				%\ar@{-} "(0,0)";"(20,-10)"
				\ar@{-} "(0,0)";"(10,-10)"
				%\ar@{-} "(0,0)";"(10,-20)"
				%\ar@{-} "(-20,10)";"(20,-10)"
				%\ar@{-} "(-10,10)";"(10,-10)"
				\ar@{-} "(10,0)";"(0,10)"
				%\ar@{-} "(0,10)";"(-20,10)"
				%\ar@{-} "(-20,10)";"(-10,0)"
				\ar@{-} "(-10,0)";"(0,-10)"
				%\ar@{-} "(0,-10)";"(20,-10)"
				\ar@{-} "(10,0)";"(10,-10)"
				\ar@{-} "(10,-10)";"(0,-10)"
				%\ar@{-} "(10,-20)";"(0,-10)"
				\ar@/^2mm/@{.} "(0,5)";"(-5,0)"
		\end{xy}}
		\ \ \ {\begin{xy}
				0;<0pt,2pt>:<2pt,0pt>::
				( 23,0) *+{\Sigma\ast{}^t\Sigma_{1212}},
				(  0,0) *{\bullet},
				(  0,0) ="(0,0)",
				( 0,10) ="(0,10)",
				( 0,12) *+{{\scriptstyle P_1}},
				( 10,0) ="(10,0)",
				( 12.5,0) *+{{\scriptstyle P_2}},
				%(-20,10) *+{\bullet}="(-20,10)",
				( 0,-10)="(0,-10)",
				(  -10,0)="(-10,0)",
				%(20,-10) *{\bullet}="(20,-10)",
				(10,-10)="(10,-10)",
				(10,-20)="(10,-20)",
				(0,6)*{}="(0,5)",
				(-6,0)*{}="(-5,0)",
				%(-10,10) *+{\bullet}="(-10,10)",
				\ar@{-} "(10,0)";"(-10,0)"
				\ar@{-} "(0,10)";"(0,-10)"
				%\ar@{-} "(0,0)";"(20,-10)"
				\ar@{-} "(0,0)";"(10,-10)"
				\ar@{-} "(0,0)";"(10,-20)"
				%\ar@{-} "(-20,10)";"(20,-10)"
				%\ar@{-} "(-10,10)";"(10,-10)"
				\ar@{-} "(10,0)";"(0,10)"
				%\ar@{-} "(0,10)";"(-20,10)"
				%\ar@{-} "(-20,10)";"(-10,0)"
				\ar@{-} "(-10,0)";"(0,-10)"
				%\ar@{-} "(0,-10)";"(20,-10)"
				\ar@{-} "(10,0)";"(10,-10)"
				\ar@{-} "(10,-10)";"(10,-20)"
				\ar@{-} "(10,-20)";"(0,-10)"
				\ar@/^2mm/@{.} "(0,5)";"(-5,0)"
			\end{xy}
		}\ \ \ {\begin{xy}
				0;<0pt,2pt>:<2pt,0pt>::
				( 23,0) *+{\Sigma\ast{}^t\Sigma_{2121}},
				(  0,0) *{\bullet},
				(  0,0) ="(0,0)",
				( 0,10) ="(0,10)",
				( 0,12) *+{{\scriptstyle P_1}},
				( 10,0) ="(10,0)",
				( 13,0) *+{{\scriptstyle P_2}},
				%(-20,10) *+{\bullet}="(-20,10)",
				( 0,-10)="(0,-10)",
				(  -10,0)="(-10,0)",
				(20,-10)="(20,-10)",
				(10,-10)="(10,-10)",
				(0,5)*{}="(0,5)",
				(-5,0)*{}="(-5,0)",
				%(-10,10) *+{\bullet}="(-10,10)",
				\ar@{-} "(10,0)";"(-10,0)"
				\ar@{-} "(0,10)";"(0,-10)"
				\ar@{-} "(0,0)";"(20,-10)"
				\ar@{-} "(0,0)";"(10,-10)"
				%\ar@{-} "(-20,10)";"(20,-10)"
				%\ar@{-} "(-10,10)";"(10,-10)"
				\ar@{-} "(10,0)";"(0,10)"
				%\ar@{-} "(0,10)";"(-20,10)"
				%\ar@{-} "(-20,10)";"(-10,0)"
				\ar@{-} "(-10,0)";"(0,-10)"
				\ar@{-} "(0,-10)";"(20,-10)"
				\ar@{-} "(20,-10)";"(10,0)"
				\ar@/^2mm/@{.} "(0,5)";"(-5,0)"
		\end{xy}}
		\]
		%\hline
		%\end{tabular}
	\end{enumerate}
\end{theorem}

\begin{remark}
	For a left (respectively, right) generator $x$ of $e_1Ae_2$, $L_x$ (respectively, $R_x$) is unique up to conjugacy.
	In particular, $t_{L_x}(e_1Ae_1)$ (respectively, $t(e_2Ae_2)_{R_x}$) does not depend on the choice of a left (respectively, right) generator $x$.
\end{remark}

\begin{example}
\begin{enumerate}[\rm (a)]
		\item 
For each $a,b\in \{(0,0),(1,1,1),(1,2,1,2),(2,1,2,1)\}$, an algebra $A_{a,b}$ satisfying 
      $\Sigma(A_{a;b})=\Sigma_{a;b}$ can be defined as given in Table\;\ref{table:g-convex algebras of rank2}.
        Algebras $A_{00;00}, A_{11;00}, A_{1212;00}$ and $A_{2121;00}$ in the first row are taken from Figure \ref{fig:tree diagram 2}, and 
        we put $A_{00;111}:=A_{111;00}^{\rm op}$, $A_{00;1212}:=A_{2121;00}^{\rm op}$ and 
        $A_{00;2121}:=A_{1212;00}^{\rm op}$ in the first column.
        Further, other algebras are constructed using Gluing Thorem \ref{theorem:gluing1}, i.e.,
        \[
            A_{a;b}:=A_{a;00}\ast A_{00;b}.
        \]       

\begin{table}[h]
    \centering
\scalebox{1}{
\begin{tabular}{|c|c|c|c|c|}
            \hline
                $b\backslash a$       & $00$ & $111$ & $1212$ & $2121$ \\
            \hline
            \multirow{2}{*}{\raisebox{-6pt}{$00$}}  & $\begin{xy} ( 0,5) *+{k\left[1\hspace{8pt} 2\right]}="1"\end{xy}$  
                       &  \raisebox{8pt}{$k\left[\begin{xy} ( 0,0) *+{1}="1",( 8,0) *+{2}="2",\ar "1";"2"^{a}  \end{xy}\right]$}
                       & {\tiny $\dfrac{k\left[\begin{xy} ( 0,0) *+{1}="1",( 8,0) *+{2}="2",\ar "1";"2"^{a}  \ar @(ld,lu)"1";"1"^{c} \end{xy}\right]}{\langle c^2\rangle}$}
                       & {\tiny $\dfrac{k\left[\begin{xy} ( 0,0) *+{1}="1",( 8,0) *+{2}="2",\ar "1";"2"^{a}  \ar @(ru,rd)"2";"2"^{c} \end{xy}\right]}{\langle c^2\rangle}$}  \rule[15pt]{0pt}{20pt} \\
\hline       
\end{tabular}
}
\caption{$g$-convex algebras $A_{a;b}$ and these $g$-fans}
\label{table:g-convex algebras of rank2}
\end{table}

\item Let $K/k$ be a field extension with degree two, and $A$ be a $k$-algebra $\left[%
\right]\in e_2 A e_2$.
			Then we have $R_x=\left[%
\right]$, $u\not \in R_x$, and the following equations hold.
			\begin{itemize}
				\item $e_1 A e_2=\left[%
\right]=R_x+u R_x$
			\end{itemize}
			Further, we have $e_2 A e_1=0$. Therefore, Theorem\;\ref{theorem:charactarization_g-convex_ranktwo} implies that
			$\Sigma(A)$ has the following form.
			
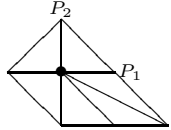
\[
			\begin{xy}
				0;<2pt,0pt>:<0pt,2pt>::
				(  0,0) ="(0,0)",
				(  0,0) *{\bullet},
				( 0,10) ="(0,10)",
				( 0,12) *+{{\scriptstyle P_2}},
				( 10,0) ="(10,0)",
				( 13,0) *+{{\scriptstyle P_1}},
				%(-20,10) *+{\bullet}="(-20,10)",
				( 0,-10)="(0,-10)",
				(  -10,0)="(-10,0)",
				(20,-10)="(20,-10)",
				(10,-10)="(10,-10)",
				(0,5)*{}="(0,5)",
				(-5,0)*{}="(-5,0)",
				%(-10,10) *+{\bullet}="(-10,10)",
				\ar@{-} "(10,0)";"(-10,0)"
				\ar@{-} "(0,10)";"(0,-10)"
				\ar@{-} "(0,0)";"(20,-10)"
				\ar@{-} "(0,0)";"(10,-10)"
				%\ar@{-} "(-20,10)";"(20,-10)"
				%\ar@{-} "(-10,10)";"(10,-10)"
				\ar@{-} "(10,0)";"(0,10)"
				%\ar@{-} "(0,10)";"(-20,10)"
				%\ar@{-} "(-20,10)";"(-10,0)"
				\ar@{-} "(-10,0)";"(0,-10)"
				\ar@{-} "(0,-10)";"(20,-10)"
				\ar@{-} "(20,-10)";"(10,0)"
				\ar@{-} "(0,10)";"(-10,0)"
				%\ar@/^-2mm/@{.} "(0,5)";"(-5,0)"
			\end{xy}
			\]
\end{enumerate}
\end{example}

\subsection{Proof of Theorem \ref{theorem:charactarization_g-convex_ranktwo}}

In this subsection, we prove Theorem \ref{theorem:charactarization_g-convex_ranktwo}.
The following observation shows Theorem \ref{theorem:charactarization_g-convex_ranktwo}(a) and gives another proof of \cite[Theorem 6.3]{AHIKM}.
	\begin{proposition}\label{proposition:sharp of convex g-fan}
		Let $A$ be as in Theorem \ref{theorem:charactarization_g-convex_ranktwo}. Then
		$A$ is $g$-convex if and only if $\Sigma(A)=\Sigma_{a;b}$ for some $a,b\in \{(0,0),(1,1,1),(1,2,1,2),(2,1,2,1)\}$.
	\end{proposition}

	\begin{proof}
		The ``if'' part is clear. Conversely, assume that $A$ is $g$-convex and $\Sigma(A)=\Sigma_{a;b}$ with $a=(a_1,\dots, a_n)$ and $b=(b_1,\dots,b_m)$. Then
		$a_i\le 2$ and $b_j\le 2$ hold for each $i,j$. Using Proposition \ref{proposition:F=F'}, it is easy to check that $a,b\in \{(0,0),(1,1,1),(1,2,1,2),(2,1,2,1)\}$ holds (see Figure \ref{fig:tree diagram}).
	\end{proof}

Next we show the following. 

\begin{lemma} 
	\label{lemma:l and r are subalgebras}
	Let $x\in e_1 A e_2$. Then $L_x$ is a subalgebra of $e_1 A e_1$, and $R_x$ is a subalgebra of $e_2 A e_2$.
\end{lemma}
\begin{proof}
This is a special case of the following easy fact: Let $A,B$ be rings, $M$ an $(A,B)$-module, and $x\in M$. Then $\{b\in B\mid xb\in Ax\}$ is a subring of $B$. 
\end{proof}

Now we give a key observation.
As in Section \ref{section: Matrices and Presilting complexes}, for $s,t\ge0$, $x\in M_{s,t}(e_1Ae_2)$, we define
\[
P_x:=(e_2A^{\oplus t}\xrightarrow{x(\cdot)} e_1A^{\oplus s})\in\Kb(\proj A).
\]

\begin{proposition}
	\label{proposition:(1,0)+(1,-1),(1,-1)+(1,-2)}
	Assume $t_{e_1Ae_1}(e_1Ae_2)=1$. For a left generator $x\in e_1Ae_2$, 
the following conditions are equivalent.
       \begin{enumerate}[\rm(a)] 
       \item $\Sigma(A)$ contains $\cone\{(1,-1),(1,-2)\}$.
       \item $t_{L_x}( e_1 A e_1)=2$.
       \item $e_1 A y+xM_{1,2}(e_2Ae_2)=M_{1,2}(e_1 A e_2)$ holds for some $u\in e_1Ae_1\setminus L_x$ and $y:=[x\ ux]$.
       \end{enumerate}
\end{proposition}

\begin{proof}
Notice that $P_x$ is indecomposable presilting by Proposition \ref{first mutation}.
 
(a)$\Rightarrow$(b) If $t_{L_x}( e_1 A e_1)= 1$, then $e_1 A e_1= L_x$ holds. Thus $e_1Ae_2=e_1Ax\subset xAe_2$ holds, 
and thus $x$ is a right generator. By Proposition \ref{first mutation}, $P_x\oplus P_2[1]\in \twosilt A$ holds, a contradiction to $\cone\{(1,-1),(1,-2)\}\in\Sigma(A)$.
Thus it suffices to prove $t_{L_x}( e_1 A e_1)\le2$.

Since $\cone\{(1,-1),(1,-2)\}\in\Sigma(A)$, there exists $y=[x_1\ x_2]\in M_{1,2}(e_1Ae_2)$ such that $P_x\oplus P_y$ is silting. By Proposition \ref{Ext 1}, we have 
\begin{eqnarray}\label{yy}
M_{1,2}(e_1 A e_2)=e_1Ay+yM_{2,2}(e_2Ae_2),\\ \label{yx}
M_{1,2}(e_1 A e_2)=e_1Ay+xM_{1,2}(e_2Ae_2).
\end{eqnarray}
Looking at the first entry of \eqref{yy}, at least one of $x_1$ and $x_2$ does not belong to $\rad_{e_1Ae_1}e_1Ae_2$. 
Without loss of generality, assume $x_1\notin \rad_{e_1Ae_1}e_1Ae_2$. Then there exists $a\in(e_1Ae_1)^\times$ such that $x=ax_1$. Since $P_y\simeq P_{ay}$, we can assume $x_1=x$ by replacing $y$ by $ay$.
Since $x$ is a left generator, there exists $u\in e_1Ae_1$ such that $x_2=ux$. Consequently, we can assume
\[y=[x\ ux].\]
For each $a\in e_1Ae_1$, \eqref{yx} implies that there exist $a'\in e_1 A e_1$ and $b,b'\in e_2 A e_2$ such that 
	\[
 	[0\ ax]= a'[x\ ux]+x[b\ b'].
 	\]
 	Then $a'$ and $a-a'u$ are in $L_x$, and hence $a=a'u+(a-a'u)\in L_xu+L_x$. Thus $e_1 A e_1 = L_x+L_x u$ and $t_{L_x}( e_1 A e_1)\le2$ hold, as desired.  
	 	
(b)$\Rightarrow$(c) Since $t_{L_x}( e_1 A e_1)=2$ and $L_x\not\subset\rad_{L_x}e_1Ae_1$, there exists $u\in e_1Ae_1\setminus L_x$ such that
\[L_xu+L_x=e_1Ae_1.\]
Multiplying $x$ from the right, we have $L_xux+L_xx=e_1Ax=e_1Ae_2$. Since $L_xx\subset xAe_2$, we have
\begin{equation}\label{uxx}
L_xux+xAe_2=e_1Ae_2.
\end{equation}
To prove (c), take any $[z\ w]\in M_{1,2}(e_1Ae_2)$. Since $x$ is a left generator, there exists $a\in e_1Ae_1$ such that $z=ax$. By \eqref{uxx}, there exist $l\in L_x$ and $b\in e_2Ae_2$ such that $w-aux=lux+xb$. 
By definition of $L_x$, there exists $c\in e_2Ae_2$ such that $lx=xc$. 
 Then we have 
 \[
 [z\ w]=(a+l)[x\ ux]+x[-c\ b]\in e_1 A y+xM_{1,2}(e_2Ae_2).
 \]

(c)$\Rightarrow$(a) By Proposition \ref{Ext 1}, the following assertions hold.
\begin{itemize}
\item $P_x$ is presilting if and only if (i) $e_1Ax+xAe_2=e_1Ae_2$.
\item $P_y$ is presilting if and only if (ii) $e_1Ay+yM_{2,2}(e_2Ae_2)=M_{1,2}(e_1 A e_2)$.
\item $\Hom_{\Kb(\proj A)}(P_x, P_y[1])=0$ if and only if (iii) $e_1 A x+yM_{2,1}(e_2Ae_2)=e_1 A e_2$.
\item $\Hom_{\Kb(\proj A)}(P_y, P_x[1])=0$ if and only if (iv) $e_1 A y+xM_{1,2}(e_2Ae_2)=M_{1,2}(e_1 A e_2)$.
\end{itemize}
It is clear that (iv) implies (ii), and (i) implies (iii). By looking at the first entry of the row vector, (iv) implies (i). 

Our assumption (c) implies that (iv) holds, and hence (i)-(iii) also hold. Thus $P_x\oplus P_y$ is presilting. 
It remains to show that $P_y$ is indecomposable. 
Suppose that $P_y$ is decomposable. By considering the $g$-vectors, we have that $P_y\simeq e_2A[1]\oplus P_z$ for some $z\in e_1Ae_2$. Since $[P_z]=[P_x]$, we have $P_z\simeq P_x$ by \cite[Theorem 6.5(a)]{DIJ}. This shows that $e_2A[1]\oplus P_x$ is silting.  
By Proposition \ref{first mutation}, we have $xAe_2=e_1Ae_2$ and $L_x=e_1Ae_1$. This contradicts $u\not\in L_x$.
\end{proof}

We are ready to prove Theorem \ref{theorem:charactarization_g-convex_ranktwo}(b).

\begin{proof}[Proof of \rm{Theorem \ref{theorem:charactarization_g-convex_ranktwo}(b)}]
The first and second statements follow from Proposition \ref{+- condition} and Proposition \ref{first mutation}.

We prove the third statement. By Proposition \ref{first mutation}, $\cone\{(1,0),(1,-1)\}\in \Sigma(A)$ if and only if $t_{e_1Ae_1}(e_1Ae_2)=1$, and  
$\cone\{(0,-1),(1,-2)\}\in \Sigma(A)$ if and only if $t(e_1Ae_2)_{e_2Ae_2}=2$. 
Thus the assertion follows from Proposition \ref{proposition:(1,0)+(1,-1),(1,-1)+(1,-2)}.

The fourth statement is the dual of the third statement.
\end{proof}

\section{Gluing fans II and Application}

\subsection{Gluing fans II}

In this subsection, we study another type of gluing fans of rank 2.
We start with introducing the following operation for fans of rank two.

\begin{definition}
For $\Sigma,\Sigma' \in \itscfanpm$ satisfying
\begin{equation}\label{sigma sigma'}
\sigma=\cone\{(0,-1),(1,-1)\}\in\Sigma\ \ \mbox{ and }\ \sigma'=\cone\{(1,-1),(1,0)\}\in\Sigma',
\end{equation}
we define $\Sigma\circ\Sigma'\in\itscfanpm$ by
\begin{eqnarray*}
(\Sigma\circ \Sigma')_1&:=&\Sigma_1\cup \Sigma'_1\\
(\Sigma\circ \Sigma')_2&:=&\left(\Sigma_2 \setminus\{\sigma\} \right)\cup\left(\Sigma'_2 \setminus \{\sigma'\} \right).
\end{eqnarray*}
\begin{equation*}
{\Sigma=\begin{xy}
0;<4pt,0pt>:<0pt,4pt>::
(0,-5)="0",
(-5,0)="1",
(0,0)*{\bullet},
(0,0)="2",
(5,0)="3",
(0,5)="4",
(5,-5)="5",
(2.9,-2.9)="a",
(4,0)="b",
(1.5,1.5)*{{\scriptstyle +}},
(-1.5,-1.5)*{{\scriptstyle -}},
(5,-2)*{{\scriptstyle ?}},
(1.5,-3.5)*{{\scriptstyle \sigma}},
\ar@{-}"0";"1",
\ar@{-}"1";"4",
\ar@{-}"4";"3",
\ar@{-}"0";"5",
\ar@{-}"2";"0",
\ar@{-}"2";"1",
\ar@{-}"2";"3",
\ar@{-}"2";"4",
\ar@{-}"2";"5",
\ar@/^-0.8mm/@{.} "a";"b",
\end{xy}}\ \ \ \ \  
{\Sigma'=\begin{xy}
0;<4pt,0pt>:<0pt,4pt>::
(0,-5)="0",
(-5,0)="1",
(0,0)*{\bullet},
(0,0)="2",
(5,0)="3",
(0,5)="4",
(5,-5)="5",
(2.9,-2.9)="d",
(0,-4)="c",
(1.5,1.5)*{{\scriptstyle +}},
(-1.5,-1.5)*{{\scriptstyle -}},
(2,-5)*{{\scriptstyle !}},
(3.5,-1.5)*{{\scriptstyle \sigma'}},
\ar@{-}"0";"1",
\ar@{-}"1";"4",
\ar@{-}"4";"3",
\ar@{-}"3";"5",
\ar@{-}"2";"0",
\ar@{-}"2";"1",
\ar@{-}"2";"3",
\ar@{-}"2";"4",
\ar@{-}"2";"5",
\ar@/^-0.8mm/@{.} "c";"d",
\end{xy}}\ \ \ \ \ 
{\Sigma\circ\Sigma'=\begin{xy}
0;<4pt,0pt>:<0pt,4pt>::
(0,-5)="0",
(-5,0)="1",
(0,0)*{\bullet},
(0,0)="2",
(5,0)="3",
(0,5)="4",
(5,-5)="5",
(2.9,-2.9)="a",
(4,0)="b",
(2.9,-2.9)="d",
(0,-4)="c",
(1.5,1.5)*{{\scriptstyle +}},
(-1.5,-1.5)*{{\scriptstyle -}},
(2,-5)*{{\scriptstyle !}},
(5,-2)*{{\scriptstyle ?}},
\ar@{-}"0";"1",
\ar@{-}"1";"4",
\ar@{-}"4";"3",
\ar@{-}"2";"0",
\ar@{-}"2";"1",
\ar@{-}"2";"3",
\ar@{-}"2";"4",
\ar@{-}"2";"5",
\ar@/^-0.8mm/@{.} "a";"b",
\ar@/^-0.8mm/@{.} "c";"d",
\end{xy}}
\end{equation*}
\end{definition}

\begin{remark}
When $\Sigma$ and $\Sigma'$ are complete, the assumption \eqref{sigma sigma'} is equivalent to that
\[\Sigma=\Sigma(a_1,\ldots,a_{n-1},1;0,0)\ \mbox{ and }\ \Sigma'=\Sigma(1,b_2,\ldots, b_m;0,0).\]
In this case, the above gluing is described as follows:
\begin{equation*} 
\Sigma\circ\Sigma'=\Sigma(a_1,\ldots,a_{n-2}, a_{n-1}+b_{2}-1, b_3,\ldots,b_{m};0,0).
\end{equation*}
\end{remark}

Now we study an algebraic counterpart of the gluing fans of this type. 

\begin{theorem}\label{theorem:gluing2}
Let $\Lambda$ and $\Lambda'$ be elementary $k$-algebras of rank 2 
with orthogonal primitive idempotents $1=e_1+e_2\in\Lambda$ and $1=e'_1+e'_2\in\Lambda'$ satisfying $e_1\Lambda e_2=0$, $e'_1\Lambda' e'_2=0$,
\[\sigma=\cone\{(0,-1),(1,-1)\}\in\Sigma(\Lambda)\ \ \mbox{ and }\ \sigma'=\cone\{(1,-1),(1,0)\}\in\Sigma(\Lambda').\]
Then, there exists an elementary $k$-algebra $\Lambda\circ\Lambda'$ such that 
\[\Sigma_2(\Lambda\circ\Lambda')=\Sigma_2(\Lambda)\circ\Sigma_2(\Lambda').\] 
\end{theorem}

Theorem \ref{theorem:gluing2} is explained by the following picture.
\[{\Sigma(\Lambda)=\begin{xy}
0;<4pt,0pt>:<0pt,4pt>::
(0,-5)="0",
(-5,0)="1",
(0,0)*{\bullet},
(0,0)="2",
(5,0)="3",
(0,5)="4",
(5,-5)="5",
(2.9,-2.9)="a",
(4,0)="b",
(1.5,1.5)*{{\scriptstyle +}},
(-1.5,-1.5)*{{\scriptstyle -}},
(5,-2)*{{\scriptstyle ?}},
(7,0)*{{\scriptstyle P_1}},
(0,6.7)*{{\scriptstyle P_2}},
(1.5,-3.5)*{{\scriptstyle \sigma}},
\ar@{-}"0";"1",
\ar@{-}"1";"4",
\ar@{-}"4";"3",
\ar@{-}"0";"5",
\ar@{-}"2";"0",
\ar@{-}"2";"1",
\ar@{-}"2";"3",
\ar@{-}"2";"4",
\ar@{-}"2";"5",
\ar@/^-0.8mm/@{.} "a";"b",
\end{xy}}\ \ \ \ \ {\Sigma(\Lambda')=\begin{xy}
0;<4pt,0pt>:<0pt,4pt>::
(0,-5)="0",
(-5,0)="1",
(0,0)*{\bullet},
(0,0)="2",
(5,0)="3",
(0,5)="4",
(5,-5)="5",
(2.9,-2.9)="d",
(0,-4)="c",
(1.5,1.5)*{{\scriptstyle +}},
(-1.5,-1.5)*{{\scriptstyle -}},
(2,-5)*{{\scriptstyle !}},
(7,0)*{{\scriptstyle P'_1}},
(0,6.7)*{{\scriptstyle P'_2}},
(3.5,-1.5)*{{\scriptstyle \sigma'}},
\ar@{-}"0";"1",
\ar@{-}"1";"4",
\ar@{-}"4";"3",
\ar@{-}"3";"5",
\ar@{-}"2";"0",
\ar@{-}"2";"1",
\ar@{-}"2";"3",
\ar@{-}"2";"4",
\ar@{-}"2";"5",
\ar@/^-0.8mm/@{.} "c";"d",
\end{xy}}\ \ \ \ \ {\Sigma(\Lambda\circ\Lambda')=\begin{xy}
0;<4pt,0pt>:<0pt,4pt>::
(0,-5)="0",
(-5,0)="1",
(0,0)*{\bullet},
(0,0)="2",
(5,0)="3",
(0,5)="4",
(5,-5)="5",
(2.9,-2.9)="a",
(4,0)="b",
(2.9,-2.9)="d",
(0,-4)="c",
(1.5,1.5)*{{\scriptstyle +}},
(-1.5,-1.5)*{{\scriptstyle -}},
(2,-5)*{{\scriptstyle !}},
(5,-2)*{{\scriptstyle ?}},
(7,0)*{{\scriptstyle Q_1}},
(0,6.7)*{{\scriptstyle Q_2}},
\ar@{-}"0";"1",
\ar@{-}"1";"4",
\ar@{-}"4";"3",
\ar@{-}"2";"0",
\ar@{-}"2";"1",
\ar@{-}"2";"3",
\ar@{-}"2";"4",
\ar@{-}"2";"5",
\ar@/^-0.8mm/@{.} "a";"b",
\ar@/^-0.8mm/@{.} "c";"d",
\end{xy}}\]

The construction of $\Lambda\circ\Lambda'$ is as follows: By our assumption, we can write
\begin{eqnarray*}
\Lambda=\left[\begin{array}{cc}
A&X\\ 0&B
\end{array}\right]\ \mbox{ and }\ P_1:=[A\ X],\ P_2:=[0\ B]\in\proj\Lambda,\\ 
\Lambda'=\left[\begin{array}{cc}
C&Y\\ 0&D
\end{array}\right]\ \mbox{ and }\ P'_1:=[C\ Y],\ P'_2:=[0\ D]\in\proj\Lambda',
\end{eqnarray*}
where
\begin{enumerate}[$\bullet$]
\item $A$, $B$, $C$, $D$ are local $k$-algebras such that $k\simeq A/J_A\simeq B/J_B\simeq C/J_C\simeq D/J_D$.
\item $X$ is an $A^{\op}\otimes_kB$-module and $Y$ is an $C^{\op}\otimes_kD$-module.
\item There exist $g\in X$ and $h\in Y$ such that $X=gB\neq0$ and $Y=Ch\neq0$ by Proposition \ref{first mutation}. \end{enumerate}
Let $A\times_kC$ (respectively, $B\times_kD$) be a fibre product of canonical surjections $A\to k$ and $C\to k$ (respectively, $B\to k$ and $D\to k$). 
As in Section \ref{section: Matrices and Presilting complexes}, we consider maps
\begin{eqnarray}\label{define pi}
\pi:X\to X/XJ_B\xrightarrow{(g(\cdot))^{-1}}B/J_B=k\ \mbox{ and }\ \pi':Y\to Y/J_CY\xrightarrow{((\cdot)h)^{-1}}C/J_C=k.
\end{eqnarray}
Let $X\times_kY$ be a fibre product of $\pi:X\to k$ and $\pi':Y\to k$. Then $X\times_kY$ is a $(A\times_kC)^{\op}\otimes_k(B\times_kD)$-module, and let
\[\Lambda\circ\Lambda':=\left[\begin{array}{cc}
A\times_kC&X\times_kY\\ 0&B\times_kD
\end{array}\right].\]

The rest of this section is devoted to proving Theorem \ref{theorem:gluing2}. For simplicity, let
\[\Gamma:=\Lambda\circ\Lambda'\ \mbox{ and }\ Q_1:=[A\times_kC\ X\times_kY],\ Q_2:=[0\ B\times_kD]\in\proj\Gamma.\]
Consider ideals of $\Gamma$ by
\[I=\left[\begin{array}{cc}
J_C&J_CY\\ 0&J_D
\end{array}\right]\ \mbox{ and }\ I'=\left[\begin{array}{cc}
J_A&XJ_B\\ 0&J_B
\end{array}\right].\]
Then there exist isomorphisms of $k$-algebras
\begin{equation}\label{factors of Gamma}
\Gamma/I\simeq\Lambda\ \mbox{ and }\ \Gamma/I'\simeq\Lambda'.
\end{equation}

As in Section \ref{section: Matrices and Presilting complexes}, for $s,t\ge0$, $x\in M_{s,t}(X)$, $y\in M_{s,t}(Y)$ and $(x',y')\in M_{s,t}(X\times_kY)$, we define
\begin{eqnarray*}
P_x&:=&(P_2^{\oplus t}\xrightarrow{x(\cdot)} P_1^{\oplus s})\in\per\Lambda,\\
P'_y&:=&({P'_2}^{\oplus t}\xrightarrow{y(\cdot)} {P'_1}^{\oplus s})\in\per\Lambda'\\
Q_{(x,y)}&:=&(Q_2^{\oplus t}\xrightarrow{(x',y')(\cdot)} Q_1^{\oplus s})\in\per\Gamma.
\end{eqnarray*}

\begin{proposition}\label{proposition:from Q to P}
Let $s,t\ge0$ and $(x,y)\in M_{s,t}(X\times_kY)$. If $Q_{(x,y)}$ is a presilting complex of $\Gamma$, then $P_x$ is a presilting complex of $\Lambda$ and $P'_y$ is a presilting complex of $\Lambda'$.
\end{proposition}

\begin{proof}
By \eqref{factors of Gamma} and $Q_{(x,y)}\otimes_\Gamma\Lambda=P_x$, the complex $P_x$  is presilting. The complex $P'_y$ is presilting similarly.
\end{proof}

Define maps $\overline{(\cdot)}:A\to C$ and $\overline{(\cdot)}:B\to D$ as the compositions of canonical maps
\[\overline{(\cdot)}:A\xrightarrow{\overline{(\cdot)}} k\subset C\ \mbox{ and }\ \overline{(\cdot)}:B\xrightarrow{\overline{(\cdot)}} k\subset D.\]
Using $\pi$ and $\pi'$ in \eqref{define pi}, define maps $\overline{(\cdot)}:X\to Y$ and $\overline{(\cdot)}:Y\to X$ by
\begin{equation}\label{define bar}
\overline{(\cdot)}:X\xrightarrow{\pi}k\xrightarrow{(\cdot)h}kh\subset Y\ \mbox{ and }\ \overline{(\cdot)}:Y\xrightarrow{\pi'}k\xrightarrow{(\cdot)g}kg\subset X.
\end{equation}
Then the first projection $X\times_kY\to X$, $(x,y)\mapsto x$ has a section given by
\[X\to X\times_kY,\ x\mapsto (x,\overline{x}),\]
and the second projection $X\times_kY\to Y$, $(x,y)\mapsto y$ has a section given by
\[Y\to X\times_kY,\ y\mapsto (\overline{y},y).\]
The following is a crucial result.

\begin{proposition}\label{proposition:from P to Q}
The following assertions hold.
\begin{enumerate}[\rm(a)]
\item Let $s\ge t$. For $x\in M_{s,t}(X)$, consider $(x,\overline{x})\in M_{s,t}(X\otimes_kY)$.
Then $P_x$ is a presilting complex of $\Lambda$ if and only if $Q_{(x,\overline{x})}$ is a presilting complex of $\Gamma$.
\item Let $s\le t$. For $y\in M_{s,t}(Y)$, consider $(\overline{y},y)\in M_{c,d}(X\otimes_kY)$.
Then $P'_y$ is a presilting complex of $\Lambda'$ if and only if $Q_{(\overline{y},y)}$ is a presilting complex of $\Gamma$.
\end{enumerate}
\end{proposition}

\begin{proof}
It suffices to prove (a) since (b) is dual to (a).

The ``if'' part is clear from Proposition \ref{proposition:from Q to P}.

We prove the ``only if'' part. By Proposition \ref{Ext 1}, it suffices to show
\[M_{s,t}(X\times_kY)=M_s(A\times_kC)(x,\overline{x})+(x,\overline{x})M_t(B\times_kD).\]
Since
\[X\times_kY=\{(0,y)\mid y\in J_CY\}+\{(z,\overline{z})\mid z\in X\},\]
it suffices to show the following assertions.
\begin{enumerate}[\rm(i)]
\item For each $y\in M_{s,t}(J_CY)$, we have $(0,y)\in M_s(A\times_kC)(x,\overline{x})$.
\item For each $z\in M_{s,t}(X)$, we have $(z,\overline{z})\in M_s(A\times_kC)(x,\overline{x})+(x,\overline{x})M_t(B\times_kD)$.
\end{enumerate}
We prove (i). Since $P_x$ is presilting,  $\pi(x)\in M_{s,t}(k)$ has full rank by Proposition \ref{proposition:full rank}.
Since $s\ge t$, the map $(\cdot)\pi(x):M_s(k)\to M_{s,t}(k)$ is surjective. Applying $J_C\otimes_k-$, the map $(\cdot)\pi(x):M_s(J_C)\to M_{s,t}(J_C)$ is also surjective, and so is the composition
\[(\cdot)\overline{x}\stackrel{\eqref{define bar}}{=}(\cdot)\pi(x)h:M_s(J_C)\xrightarrow{(\cdot)\pi(x)} M_{s,t}(J_C)\xrightarrow{(\cdot)h} M_{s,t}(J_CY).\]
Therefore there exists $c\in M_s(J_C)$ such that $y=c\overline{x}$. Then $(0,c)\in M_s(A\times_kC)$ satisfies
\[(0,c)(x,\overline{x})=(0,y).\]
We prove (ii). Since $P_x$ is presilting, we have $M_{s,t}(X)=M_s(A)x+xM_t(B)$ by Proposition \ref{Ext 1}. 
Thus there exist $a\in M_s(A)$ and $b\in M_t(B)$ such that $z=ax+xb$. Then
\[(a,\overline{a})(x,\overline{x})+(x,\overline{x})(b,\overline{b})=(ax+xb,\overline{ax+xb})=(z,\overline{z}).\]
Thus the assertion follows.
\end{proof}

Now we are ready to prove Theorem \ref{theorem:gluing2}.

\begin{proof}[Proof of Theorem \ref{theorem:gluing2}] 
By Proposition \ref{+- condition}, each of $\Sigma(\Lambda)$, $\Sigma(\Lambda')$ and $\Sigma(\Gamma)$ contains $\cone\{(-1,0),\ (0,1)\}$. By Propositions \ref{proposition:from Q to P} and \ref{proposition:from P to Q}, the following assertions hold.
\begin{enumerate}[\rm(i)]
\item Let $s\ge t$. Then there exists $x\in M_{s,t}(X)$ such that $P_x$ is presilting if and only if there exists $(x,y)\in M_{s,t}(X\times_kY)$ such that $Q_{(s,t)}$ is presilting.
\item Let $s\le t$. Then there exists $y\in M_{s,t}(Y)$ such that $P'_y$ is presilting if and only if there exists $(x,y)\in M_{s,t}(X\times_kY)$ such that $Q_{(s,t)}$ is presilting.
\end{enumerate}
Therefore the claim follows. 
\end{proof}

\begin{example}
	Let $\Lambda$ and $\Lambda'$ be the following algebras.
	\[
	\Lambda:=\dfrac{
		k\left[\begin{xy}( 0,0) *+{1}="1",
			( 8,0) *+{2}="2",\ar "1";"2"^a  \ar @(ru,rd)"2";"2"^{b}   
		\end{xy}\right]}{\langle b^2\rangle}\ \hspace{10pt}
	\Lambda':=\dfrac{
		k\left[\begin{xy}( 0,0) *+{1}="1",
			( 8,0) *+{2}="2",\ar "1";"2"^a  \ar @(ru,rd)"2";"2"^{d}  \ar @(ul,ur)"1";"1"^(.2){c_1} \ar @ (dl,dr)"1";"1"_(.2){c_2}
		\end{xy}\right]}{\langle c_1^2, c_2^2, d^2, c_1c_2,c_1a-ad\rangle}
	\]
	By Examples \ref{example:2121} and \ref{example:13122}, we have
	\[
	\Sigma(\Lambda)=\Sigma_{2121}=\begin{tikzpicture}[baseline=0mm]
		%%[ 2, 1, 2, 1 ]
		\coordinate(0) at(0:0);
		\coordinate(x) at(0:0.66);
		\coordinate(y) at(90:0.66);
        \node at(0.2,0.2) {$\scriptstyle +$};
        \node at(-0.2,-0.2) {$\scriptstyle -$};
		\coordinate(shift2121) at (0);
		\coordinate(v0) at($0*(x)+0*(y)+(shift2121)$); 
		\coordinate(v1) at($1*(x)+0*(y)+(shift2121)$); 
		\coordinate(v2) at($2*(x)+-1*(y)+(shift2121)$); 
		\coordinate(v3) at($1*(x)+-1*(y)+(shift2121)$); 
		\coordinate(v4) at($0*(x)+-1*(y)+(shift2121)$); 
		\draw[fill=black] (v0) circle [radius = 0.55mm]; 
		\draw (v0)--(v1); 
		\draw (v0)--(v2); 
		\draw (v0)--(v3); 
		\draw (v0)--(v4); 
		\draw (v1)--(v2)--(v3)--(v4); 
		\coordinate(w1) at($1*(x)+0*(y)+(shift2121)$); 
		\coordinate(w2) at($0*(x)+-1*(y)+(shift2121)$); 
		\coordinate(w3) at($-1*(x)+0*(y)+(shift2121)$); 
		\coordinate(w4) at($0*(x)+1*(y)+(shift2121)$); 
		\draw (v0)--(w3); 
		\draw (v0)--(w4);
		\draw (w3)--(w4); 
		\draw (w1)--(w4);
		\draw (w2)--(w3); 
	\end{tikzpicture}
	\quad \quad 
	\Sigma(\Lambda')=\Sigma_{13122}=\begin{tikzpicture}[baseline=0mm]
		%%[ 1, 3, 1, 2, 2 ]
		\coordinate(0) at(0:0);
		\coordinate(x) at(0:0.66);
		\coordinate(y) at(90:0.66);
        \node at(0.2,0.2) {$\scriptstyle +$};
        \node at(-0.2,-0.2) {$\scriptstyle -$};
		\coordinate(shift13122) at (0);
		\coordinate(v0) at($0*(x)+0*(y)+(shift13122)$); 
		\coordinate(v1) at($1*(x)+0*(y)+(shift13122)$); 
		\coordinate(v2) at($1*(x)+-1*(y)+(shift13122)$); 
		\coordinate(v3) at($2*(x)+-3*(y)+(shift13122)$); 
		\coordinate(v4) at($1*(x)+-2*(y)+(shift13122)$); 
		\coordinate(v5) at($0*(x)+-1*(y)+(shift13122)$); 
		\draw[fill=black] (v0) circle [radius = 0.55mm]; 
		\draw (v0)--(v1); 
		\draw (v0)--(v2); 
		\draw (v0)--(v3); 
		\draw (v0)--(v4); 
		\draw (v0)--(v5); 
		\draw (v1)--(v2)--(v3)--(v4)--(v5); 
		\coordinate(w1) at($1*(x)+0*(y)+(shift13122)$); 
		\coordinate(w2) at($0*(x)+-1*(y)+(shift13122)$); 
		\coordinate(w3) at($-1*(x)+0*(y)+(shift13122)$); 
		\coordinate(w4) at($0*(x)+1*(y)+(shift13122)$); 
		\draw (v0)--(w3); 
		\draw (v0)--(w4);
		\draw (w3)--(w4); 
		\draw (w1)--(w4);
		\draw (w2)--(w3); 
	\end{tikzpicture}
	\]
	Let $\Gamma:=\left[\begin{array}{cc}
	e_1 \Lambda e_1 \times_k e_1 \Lambda'e_1& e_1\Lambda e_2\times_k e_1 \Lambda'e_2\\ 0&e_2\Lambda e_2\times_ke_2 \Lambda'e_2
\end{array}\right]$. Then we have
\[
\Gamma=\dfrac{
	k\left[\begin{xy}( 0,0) *+{1}="1",
		( 8,0) *+{2}="2",\ar "1";"2"^a  \ar @(ur,ul)"2";"2"_(.2){b} \ar @(dr,dl)"2";"2"^(.2){d}  \ar @(ul,ur)"1";"1"^(.2){c_1} \ar @ (dl,dr)"1";"1"_(.2){c_2}
	\end{xy}\right]}{\langle b^2, c_1^2, c_2^2, d^2, c_1c_2,c_1a-ad, c_2ab, bd,db\rangle}
\ \mbox{ and }\ \Sigma(\Gamma)=\Sigma_{214122}=\begin{tikzpicture}[baseline=-12mm]
	%%[ 1, 3, 1, 2, 2 ]
	\coordinate(0) at(0:0);
	\coordinate(x) at(0:0.66);
	\coordinate(y) at(90:0.66);
    \node at(0.2,0.2) {$\scriptstyle +$};
    \node at(-0.2,-0.2) {$\scriptstyle -$};
	\coordinate(shift13122) at (0);
	\coordinate(v0) at($0*(x)+0*(y)+(shift13122)$); 
	\coordinate(v1) at($1*(x)+0*(y)+(shift13122)$); 
	\coordinate(v1') at($2*(x)+-1*(y)+(shift13122)$); 
	\coordinate(v2) at($1*(x)+-1*(y)+(shift13122)$); 
	\coordinate(v3) at($2*(x)+-3*(y)+(shift13122)$); 
	\coordinate(v4) at($1*(x)+-2*(y)+(shift13122)$); 
	\coordinate(v5) at($0*(x)+-1*(y)+(shift13122)$); 
	\draw[fill=black] (v0) circle [radius = 0.55mm]; 
	\draw (v0)--(v1);
	\draw (v0)--(v1'); 
	\draw (v0)--(v2); 
	\draw (v0)--(v3); 
	\draw (v0)--(v4); 
	\draw (v0)--(v5); 
	\draw (v1)--(v1')--(v2)--(v3)--(v4)--(v5); 
	\coordinate(w1) at($1*(x)+0*(y)+(shift13122)$); 
	\coordinate(w2) at($0*(x)+-1*(y)+(shift13122)$); 
	\coordinate(w3) at($-1*(x)+0*(y)+(shift13122)$); 
	\coordinate(w4) at($0*(x)+1*(y)+(shift13122)$); 
	\draw (v0)--(w3); 
	\draw (v0)--(w4);
	\draw (w3)--(w4); 
	\draw (w1)--(w4);
	\draw (w2)--(w3); 
\end{tikzpicture}
\]
\end{example}

\subsection{Application to non-complete $g$-fans}

Now we introduce a quiddity sequence of non-complete fans. 
We start with giving a general definition.
For a totally ordered set $I$, let 
\[\widetilde{I}:=(\{0\}\times\Z_{\ge0})\sqcup(I\times\Z)\sqcup(\{\infty\}\times\Z_{\le0})\subset(\{0\}\sqcup I\sqcup\{\infty\})\times\Z,\]
which is regarded as a totally ordered set by the lexicographic order, and $0$ is the minimum element and $\infty$ is the maximum element of $\{0\}\sqcup I\sqcup\{\infty\}$.

Assume that $\Sigma\in\itscfanpm$ is non-complete. Then there exists a totally ordered set $I$ such that
 we denote the rays of $\Sigma$ by
\[\Sigma_1=\{v_{i,n}\mid (i,n)\in\widetilde{I}\}\]
which are indexed in a clockwise orientation by the totally ordered set $\widetilde{I}$ such that
\[v_{0,0}=(0,1),\ v_{0,1}=(1,0),\ v_{\infty,-1}=(0,-1),\ v_{\infty,0}=(-1,0).\] 
We define the \emph{quiddity sequence} $\df(\Sigma)=(a_{i,n})_{(i,n)\in
\widetilde{I}\setminus\{(0,0),(\infty,0)\}}$ of $\Sigma$ by
\[v_{i,n-1}+v_{i,n+1}=a_{i,n}v_{i,n}.\]

Using the subsequence $\df_i:=(a_{i,n})_{n}$ of $\df(\Sigma)$ for each $i\in\{0\}\sqcup I\sqcup\{\infty\}$, we often write $\df(\Sigma)$ as
$$\df(\Sigma)=(\df_i)_{i\in\{0\}\sqcup I\sqcup\{\infty\}}.$$ 
For a finite dimensional algebra $\Lambda$ such that $\Sigma(\Lambda)\in\itscfanpm$ is non-complete, we define $\df(\Sigma(\Lambda))$ as above.

Now let $\Sigma,\Sigma'\in\itscfanpm$ be non-complete fans  satisfying \eqref{sigma sigma'}.
Let $I$ and $I'$ be totally ordered sets such that
\[\Sigma_1=\{v_{i,n}\mid (i,n)\in\widetilde{I}\}\ \mbox{ and }\ \Sigma'_1=\{v'_{i,n}\mid (i,n)\in\widetilde{I'}\},\]
and we denote by $(a_{i,n})_{(i,n)\in\widetilde{I}\setminus\{(0,0),(\infty,0)\}}$ and $(b_{i,n})_{(i,n)\in
\widetilde{I'}\setminus\{(0,0),(\infty,0)\}}$ the quiddity sequences of $\Sigma$ and $\Sigma'$ respectively.
Then we have $a_{\infty,-1}=1=b_{0,1}$ and $v_{\infty,-2}=(1,-1)=v'_{0,2}$.
We regard $I\sqcup\{x\}\sqcup I'$ be a totally ordered set such that $i<x<i'$ for all $i\in I$ and $i'\in I'$. Then
\[(\Sigma\circ\Sigma')_1=\{w_{i,n}\mid (i,n)\in\widetilde{I\sqcup\{x\}\sqcup I'}\},\] where 
\[w_{i,n}=\left\{\begin{array}{ll}
v_{i,n}&\mbox{if $i\in\{0\}\sqcup I$},\\
v_{\infty,n-2}&\mbox{if $i=x$,\ $n\le -1$},\\
(1,-1)&\mbox{if $i=x$,\ $n=0$},\\
v_{0,n+2}'&\mbox{if $i=x$,\ $n\ge 1$},\\
v_{i,n}'&\mbox{if $i\in I'\sqcup\{\infty\}$}.
\end{array}\right.\]
The quiddity sequence $(c_{i,n})_{(i,n)\in\widetilde{I\sqcup\{x\}\sqcup I'}\setminus\{(0,0),(\infty,0)\}}$ of $\Sigma\circ\Sigma'$ is given by
\[c_{i,n}=\left\{\begin{array}{ll}
a_{i,n}&\mbox{if $i\in\{0\}\sqcup I$},\\
a_{\infty,n-2}&\mbox{if $i=x$,\ $n\le -1$},\\
a_{\infty,-2}+b_{0,2}-1&\mbox{if $i=x$,\ $n=0$},\\
b_{0,n+2}&\mbox{if $i=x$,\ $n\ge 1$},\\
b_{i,n}&\mbox{if $i\in I'\sqcup\{\infty\}$}.
\end{array}\right.\]

Now, for each positive integer $N$,  we give explicitly a finite dimensional algebra $\Lambda_N$ such that $\Hasse(\twosilt\Lambda_N)$ has precisely $N$ connected components. 
Let 
\[\Lambda:=\dfrac{
		k\left[\begin{xy}( 0,0) *+{1}="1",
			( 8,0) *+{2}="2",\ar "1";"2"  \ar @(ru,rd)"2";"2"^{a}   
		\end{xy}\right]}{\langle a^4\rangle}
\]
Then we have the following result.

\begin{theorem}\label{many}
Let $N$ be a positive integer. Define a finite dimensional $k$-algebra $\Lambda_N$ by
\[\Lambda_1:=\Lambda\ \mbox{ and }\ \Lambda_{N+1}:=\Lambda\circ\rho(\Lambda_N).\] 
\begin{enumerate}[\rm(a)]
\item The complement of $\bigcup_{\sigma\in\Sigma(\Lambda_N)}\sigma$ in $K_0(\proj\Lambda_N)_\R$ is a disjoint union $\bigsqcup_{i=0}^{N-1}\R_{>0}u_i$ of half lines for $u_i:=(2,-1)\left[\begin{smallmatrix}0&1\\ -1&4\end{smallmatrix}\right]^i$.
\item 
We have
$\df(\Sigma(\Lambda_N))=(\df_0,\df_1,\ldots,\df_{N-1},\df_\infty)$, where
$$\df_i=\left\{\begin{array}{ll}
(141414\cdots)&\mbox{if $i=0$},\\
(\cdots14141714141\cdots)&\mbox{if $1\leq i\leq N-1$},\\
(\cdots14141\underbrace{44\cdots 4}_{N-1} )&\mbox{if $i=\infty$}.
\end{array}\right.$$ 
\item $\Hasse(\twosilt\Lambda_N)$ has precisely $N$ connected components.
\end{enumerate}
\end{theorem}

The case $N=1$ follows from a result of Wang \cite{W}.

\begin{lemma}\cite[Proof of Proposition 3.10]{W}\label{wang}
\begin{enumerate}[\rm(a)]
\item The complement of $\bigcup_{\sigma\in\Sigma(\Lambda)}\sigma$ in $K_0(\proj\Lambda_1)_\R$ is a half line $\R_{>0}(2,-1)$.
\item We have %\old{$\df(\Sigma(\Lambda))=(41414\cdots14141)$.}
$\df(\Sigma(\Lambda))=(\df_0,\df_\infty)$, where 
$\df_0=(41414\cdots)$ and $\df_\infty=(\cdots14141)$. 
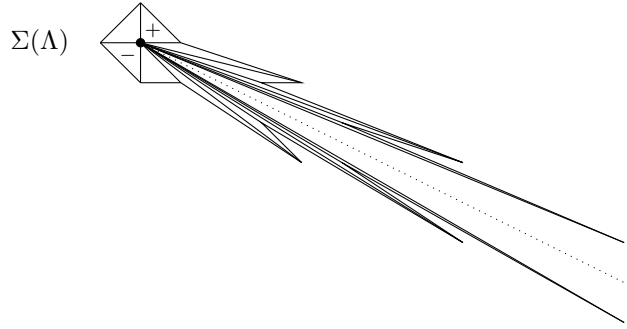
\[\begin{tikzpicture}[baseline=0mm]
\coordinate(0) at(0:0);
\coordinate(x) at(0:0.53);
\coordinate(y) at(90:0.53);
\node at(0.165,0.165) {$\scriptstyle +$};
\node at(-0.165,-0.165) {$\scriptstyle -$};
%%[ lambda0 ]
\coordinate(shift) at($0*(x)+0*(y)$); 
\node(lambda0) at(shift) {}; 
\coordinate(L) at($-2.5*(x)+0*(y)+(shift)$); %%%%optional
\node at(L) {$\Sigma(\Lambda)$};  %%%%optional

\coordinate(v0) at($0*(x)+0*(y)+(shift)$); 
\draw[fill=black] (v0) circle [radius = 0.55mm]; 
\coordinate(v1) at($1*(x)+0*(y)+(shift)$); 
\coordinate(v2) at($4*(x)+-1*(y)+(shift)$); 
\coordinate(v3) at($3*(x)+-1*(y)+(shift)$); 
\coordinate(v4) at($8*(x)+-3*(y)+(shift)$); 
\coordinate(v5) at($5*(x)+-2*(y)+(shift)$); 
\coordinate(v6) at($12*(x)+-5*(y)+(shift)$); 
    \draw (v0)--(v1); 
    \draw (v0)--(v2); 
    \draw (v0)--(v3); 
    \draw (v0)--(v4);
    \draw (v0)--(v5); 
    \draw (v0)--(v6);
    \draw (v1)--(v2)--(v3)--(v4)--(v5)--(v6);
\coordinate(u0) at($0*(x)+0*(y)+(shift)$); 
\coordinate(u1) at($0*(x)+-1*(y)+(shift)$); 
\coordinate(u2) at($1*(x)+-1*(y)+(shift)$); 
\coordinate(u3) at($4*(x)+-3*(y)+(shift)$); 
\coordinate(u4) at($3*(x)+-2*(y)+(shift)$); 
\coordinate(u5) at($8*(x)+-5*(y)+(shift)$); 
\coordinate(u6) at($5*(x)+-3*(y)+(shift)$); 
\coordinate(u7) at($12*(x)+-7*(y)+(shift)$); 
    \draw (u0)--(u1); 
    \draw (u0)--(u2); 
    \draw (u0)--(u3); 
    \draw (u0)--(u4);
    \draw (u0)--(u5); 
    \draw (u0)--(u6);
    \draw (u0)--(u7); 
    \draw (u1)--(u2)--(u3)--(u4)--(u5)--(u6)--(u7); 

\coordinate(w1) at($1*(x)+0*(y)+(shift)$); 
\coordinate(w2) at($0*(x)+-1*(y)+(shift)$); 
\coordinate(w3) at($-1*(x)+0*(y)+(shift)$); 
\coordinate(w4) at($0*(x)+1*(y)+(shift)$); 
    \draw (v0)--(w3); 
    \draw (v0)--(w4);
    \draw (w3)--(w4); 
    \draw (w1)--(w4);
    \draw (w2)--(w3); 
\coordinate(t0) at($2*(x)+-1*(y)+(shift)$); 
\draw[dotted] (v0)--($6*(t0)$);
\end{tikzpicture}\]
\item $\Hasse(\twosilt\Lambda)$ is connected.
\end{enumerate}
\end{lemma}

We are ready to prove Theorem \ref{many}.

\begin{proof}[Proof of Theorem \ref{many}]
By  Theorem \ref{theorem:rotation} and Lemma \ref{wang}, 
the complement of $\Sigma(\rho(\Lambda))$ is $\R_{>0}(1,-2)$ and we have 
$\df(\Sigma(\rho(\Lambda)))=(\df_0,\df_\infty)$, where 
$\df_0=(14141\cdots)$ and $\df_\infty=(\cdots41414)$. 
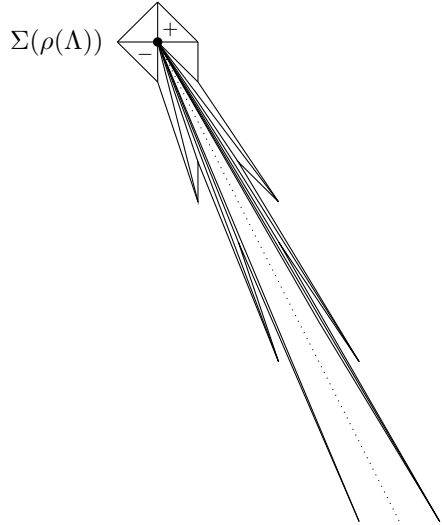
\[
%\hspace{5mm} %%%%%%%%%%%%%%%%%%%%%%%%%%%%%%%%%%%%%%%%%%%%%%
\begin{tikzpicture}[baseline=0mm]
\coordinate(0) at(0:0);
\coordinate(x) at(0:0.53);
\coordinate(y) at(90:0.53);
\node at(0.165,0.165) {$\scriptstyle +$};
\node at(-0.165,-0.165) {$\scriptstyle -$};

%%[ rholambda0 ]
\coordinate(rshift) at($0*(x)+0*(y)$); 
\node(rlambda0) at(rshift) {}; 
\coordinate(L) at($-2.5*(x)+0*(y)+(rshift)$); %%%%optional
\node at(L) {$\Sigma(\rho(\Lambda))$};  %%%%optional
\coordinate(v0) at($0*(x)+0*(y)+(rshift)$); 
\draw[fill=black] (v0) circle [radius = 0.55mm]; 
\coordinate(vv1) at($1*(x)+0*(y)+(rshift)$); 
\coordinate(vv2) at($1*(x)+-1*(y)+(rshift)$); 
\coordinate(vv3) at($3*(x)+-4*(y)+(rshift)$); 
\coordinate(vv4) at($2*(x)+-3*(y)+(rshift)$); 
\coordinate(vv5) at($5*(x)+-8*(y)+(rshift)$); 
\coordinate(vv6) at($3*(x)+-5*(y)+(rshift)$);
\coordinate(vv7) at($7*(x)+-12*(y)+(rshift)$);
\coordinate(vv8) at($4*(x)+-7*(y)+(rshift)$);
    \draw (v0)--(vv1); 
    \draw (v0)--(vv2); 
    \draw (v0)--(vv3); 
    \draw (v0)--(vv4);
    \draw (v0)--(vv5); 
    \draw (v0)--(vv6);
    \draw (v0)--(vv7);
    \draw (v0)--(vv8);
    \draw (vv1)--(vv2)--(vv3)--(vv4)--(vv5)--(vv6)--(vv7)--(vv8);
\coordinate(u0) at($0*(x)+0*(y)+(rshift)$); 
\coordinate(uu1) at($0*(x)+-1*(y)+(rshift)$); 
\coordinate(uu2) at($1*(x)+-4*(y)+(rshift)$); 
\coordinate(uu3) at($1*(x)+-3*(y)+(rshift)$); 
\coordinate(uu4) at($3*(x)+-8*(y)+(rshift)$); 
\coordinate(uu5) at($2*(x)+-5*(y)+(rshift)$); 
\coordinate(uu6) at($5*(x)+-12*(y)+(rshift)$); 
    \draw (u0)--(uu1); 
    \draw (u0)--(uu2); 
    \draw (u0)--(uu3); 
    \draw (u0)--(uu4);
    \draw (u0)--(uu5); 
    \draw (u0)--(uu6);
    \draw (uu1)--(uu2)--(uu3)--(uu4)--(uu5)--(uu6); 
\coordinate(w1) at($1*(x)+0*(y)+(rshift)$); 
\coordinate(w2) at($0*(x)+-1*(y)+(rshift)$); 
\coordinate(w3) at($-1*(x)+0*(y)+(rshift)$); 
\coordinate(w4) at($0*(x)+1*(y)+(rshift)$); 
    \draw (v0)--(w3); 
    \draw (v0)--(w4);
    \draw (w3)--(w4); 
    \draw (w1)--(w4);
    \draw (w2)--(w3); 
\coordinate(rt) at($1*(x)+-2*(y)+(rshift)$); 
\draw[dotted] (v0)--($6*(rt)$);
\end{tikzpicture}
\]

Thus we can apply Theorem \ref{theorem:gluing2} to $\Lambda$ and $\rho(\Lambda)$, and we get the algebra 
$\Lambda_2:=\Lambda \circ\rho(\Lambda).$
Then $\df(\Sigma(\Lambda_2))=(\df_0,\df_1,\df_\infty)$, where 
$$\df_i=\left\{\begin{array}{ll}
(41414\cdots)&\mbox{if $i=0$},\\
(\cdots14141714141\cdots)&\mbox{if $i=1$},\\
(\cdots41414 )&\mbox{if $i=\infty$},
\end{array}\right.$$
and the complement of $\Sigma(\Lambda_2)$ is 
$\R_{>0}(2,-1)\sqcup\R_{>0}(1,-2)$. In particular, 
$\Sigma(\Lambda_2)$ has two connected components.

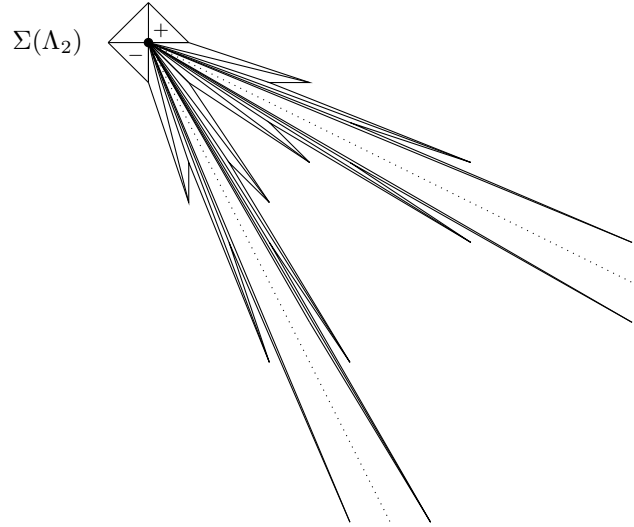
\[
\begin{tikzpicture}[baseline=0mm]
\coordinate(0) at(0:0);
\coordinate(x) at(0:0.53);
\coordinate(y) at(90:0.53);
\node at(0.165,0.165) {$\scriptstyle +$};
\node at(-0.165,-0.165) {$\scriptstyle -$};
%%[ lambda2 ]
\coordinate(shift2) at($0*(x)+0*(y)$); 
\node(lambda2) at(shift2) {}; 
\coordinate(L) at($-2.5*(x)+0*(y)+(shift2)$); %%%%optional
\node at(L) {$\Sigma(\Lambda_2)$};  %%%%optional
\coordinate(v0) at($0*(x)+0*(y)+(shift2)$); 
\draw[fill=black] (v0) circle [radius = 0.55mm]; 
\coordinate(v1) at($1*(x)+0*(y)+(shift2)$); 
\coordinate(v2) at($4*(x)+-1*(y)+(shift2)$); 
\coordinate(v3) at($3*(x)+-1*(y)+(shift2)$); 
\coordinate(v4) at($8*(x)+-3*(y)+(shift2)$); 
\coordinate(v5) at($5*(x)+-2*(y)+(shift2)$); 
\coordinate(v6) at($12*(x)+-5*(y)+(shift2)$); 
    \draw (v0)--(v1); 
    \draw (v0)--(v2); 
    \draw (v0)--(v3); 
    \draw (v0)--(v4);
    \draw (v0)--(v5); 
    \draw (v0)--(v6);
    \draw (v1)--(v2)--(v3)--(v4)--(v5)--(v6);
\coordinate(u0) at($0*(x)+0*(y)+(shift2)$); 
\coordinate(u1) at($0*(x)+-1*(y)+(shift2)$); 
\coordinate(u2) at($1*(x)+-1*(y)+(shift2)$); 
\coordinate(u3) at($4*(x)+-3*(y)+(shift2)$); 
\coordinate(u4) at($3*(x)+-2*(y)+(shift2)$); 
\coordinate(u5) at($8*(x)+-5*(y)+(shift2)$); 
\coordinate(u6) at($5*(x)+-3*(y)+(shift2)$); 
\coordinate(u7) at($12*(x)+-7*(y)+(shift2)$); 
    \draw (u0)--(u1); 
    \draw (u0)--(u2); 
    \draw (u0)--(u3); 
    \draw (u0)--(u4);
    \draw (u0)--(u5); 
    \draw (u0)--(u6);
    \draw (u0)--(u7); 
    \draw (u2)--(u3)--(u4)--(u5)--(u6)--(u7); 

\coordinate(vv1) at($1*(x)+0*(y)+(shift2)$); 
\coordinate(vv2) at($1*(x)+-1*(y)+(shift2)$); 
\coordinate(vv3) at($3*(x)+-4*(y)+(shift2)$); 
\coordinate(vv4) at($2*(x)+-3*(y)+(shift2)$); 
\coordinate(vv5) at($5*(x)+-8*(y)+(shift2)$); 
\coordinate(vv6) at($3*(x)+-5*(y)+(shift2)$);
\coordinate(vv7) at($7*(x)+-12*(y)+(shift2)$);
\coordinate(vv8) at($4*(x)+-7*(y)+(shift2)$);
    \draw (v0)--(vv1); 
    \draw (v0)--(vv2); 
    \draw (v0)--(vv3); 
    \draw (v0)--(vv4);
    \draw (v0)--(vv5); 
    \draw (v0)--(vv6);
    \draw (v0)--(vv7);
    \draw (v0)--(vv8);
    \draw (vv2)--(vv3)--(vv4)--(vv5)--(vv6)--(vv7)--(vv8);
\coordinate(u0) at($0*(x)+0*(y)+(shift2)$); 
\coordinate(uu1) at($0*(x)+-1*(y)+(shift2)$); 
\coordinate(uu2) at($1*(x)+-4*(y)+(shift2)$); 
\coordinate(uu3) at($1*(x)+-3*(y)+(shift2)$); 
\coordinate(uu4) at($3*(x)+-8*(y)+(shift2)$); 
\coordinate(uu5) at($2*(x)+-5*(y)+(shift2)$); 
\coordinate(uu6) at($5*(x)+-12*(y)+(shift2)$); 
    \draw (u0)--(uu1); 
    \draw (u0)--(uu2); 
    \draw (u0)--(uu3); 
    \draw (u0)--(uu4);
    \draw (u0)--(uu5); 
    \draw (u0)--(uu6);
    \draw (uu1)--(uu2)--(uu3)--(uu4)--(uu5)--(uu6); 
\coordinate(w1) at($1*(x)+0*(y)+(shift2)$); 
\coordinate(w2) at($0*(x)+-1*(y)+(shift2)$); 
\coordinate(w3) at($-1*(x)+0*(y)+(shift2)$); 
\coordinate(w4) at($0*(x)+1*(y)+(shift2)$); 
    \draw (v0)--(w3); 
    \draw (v0)--(w4);
    \draw (w3)--(w4); 
    \draw (w1)--(w4);
    \draw (w2)--(w3); 
\coordinate(t) at($2*(x)+-1*(y)+(shift2)$); 
\draw[dotted] (v0)--($6*(t)$);
\coordinate(rt) at($1*(x)+-2*(y)+(shift2)$); 
\draw[dotted] (v0)--($6*(rt)$);
\end{tikzpicture}
\]

Next we apply the rotation theorem to $\Lambda_2$ and we get the algebra $\rho(\Lambda_2)$. Then we have 
 $\df(\Sigma(\rho(\Lambda_2)))=(\df_0,\df_1,\df_\infty)$, where 
$$\df_i=\left\{\begin{array}{ll}
(14141\cdots)&\mbox{if $i=0$},\\
(\cdots14141714141\cdots)&\mbox{if $i=1$},\\
(\cdots1414144 )&\mbox{if $i=\infty$}.
\end{array}\right.$$
By Definition \ref{define rotation}, 
$\rho$ induces a linear transformation $\left[\begin{smallmatrix}0&1\\ -1&4\end{smallmatrix}\right]$ of $\mathbb{R}^2$ and we have 
$(2,-1)\left[\begin{smallmatrix}0&1\\ -1&4\end{smallmatrix}\right]=(1,-2)$ and 
$(1,-2)\left[\begin{smallmatrix}0&1\\ -1&4\end{smallmatrix}\right]=(2,-7)$.  
Hence the complement of $\Sigma(\rho(\Lambda_2))$ is $\R_{>0}(1,-2)\sqcup\R_{>0}(2,-7)$.
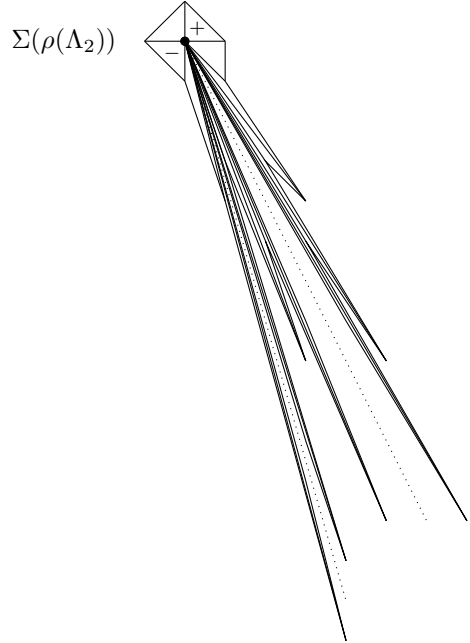
\[
\begin{tikzpicture}[baseline=0mm]
\coordinate(0) at(0:0);
\coordinate(x) at(0:0.53);
\coordinate(y) at(90:0.53);
\node at(0.165,0.165) {$\scriptstyle +$};
\node at(-0.165,-0.165) {$\scriptstyle -$};
%%[ lambda3 ]
\coordinate(shift3) at($0*(x)+0*(y)$); 
\node(rho_lambda2) at(shift3) {}; 
\coordinate(L) at($-3*(x)+0*(y)+(shift3)$); %%%%optional
\node at(L) {$\Sigma(\rho(\Lambda_2))$};  %%%%optional

\coordinate(v0) at($0*(x)+0*(y)+(shift3)$); 
\draw[fill=black] (v0) circle [radius = 0.55mm]; 
\coordinate(v1) at($1*(x)+0*(y)+(shift3)$); 
\coordinate(v2) at($4*(x)+-1*(y)+(shift3)$); 
\coordinate(v3) at($3*(x)+-1*(y)+(shift3)$); 
\coordinate(v4) at($8*(x)+-3*(y)+(shift3)$); 
\coordinate(v5) at($5*(x)+-2*(y)+(shift3)$); 
\coordinate(v6) at($12*(x)+-5*(y)+(shift3)$); 
    \draw (v0)--(v1); 
%    \draw (v0)--(v2); 
%    \draw (v0)--(v3); 
%    \draw (v0)--(v4);
%    \draw (v0)--(v5); 
%    \draw (v0)--(v6);
%    \draw (v1)--(v2)--(v3)--(v4)--(v5)--(v6);
\coordinate(u0) at($0*(x)+0*(y)+(shift3)$); 
\coordinate(u1) at($0*(x)+-1*(y)+(shift3)$); 
\coordinate(u2) at($1*(x)+-1*(y)+(shift3)$); 
\coordinate(u3) at($4*(x)+-3*(y)+(shift3)$); 
\coordinate(u4) at($3*(x)+-2*(y)+(shift3)$); 
\coordinate(u5) at($8*(x)+-5*(y)+(shift3)$); 
\coordinate(u6) at($5*(x)+-3*(y)+(shift3)$); 
\coordinate(u7) at($12*(x)+-7*(y)+(shift3)$); 
    %\draw (u0)--(u1); 
    \draw (v1)--(u2);
%    \draw (u0)--(u2); 
%    \draw (u0)--(u3); 
%    \draw (u0)--(u4);
%    \draw (u0)--(u5); 
%    \draw (u0)--(u6);
%    \draw (u0)--(u7); 
%    \draw (u2)--(u3)--(u4)--(u5)--(u6)--(u7); 
\coordinate(w1) at($1*(x)+0*(y)+(shift3)$); 
\coordinate(w2) at($0*(x)+-1*(y)+(shift3)$); 
\coordinate(w3) at($-1*(x)+0*(y)+(shift3)$); 
\coordinate(w4) at($0*(x)+1*(y)+(shift3)$); 
    \draw (v0)--(w3); 
    \draw (v0)--(w4);
    \draw (w3)--(w4); 
    \draw (w1)--(w4);
    \draw (w2)--(w3); 

\coordinate(a0) at($0*(x)+0*(y)+(shift3)$); 
\coordinate(a1) at($0*(x)+-1*(y)+(shift3)$); 
\coordinate(a2) at($1*(x)+-4*(y)+(shift3)$); 
\coordinate(a3) at($4*(x)+-15*(y)+(shift3)$); 
\coordinate(a4) at($3*(x)+-11*(y)+(shift3)$); 
    \draw (v0)--(a1); 
    \draw (a0)--(a2); 
    \draw (a0)--(a3); 
    \draw (a0)--(a4); 
    \draw (a1)--(a2)--(a3)--(a4);

\coordinate(b0) at($0*(x)+0*(y)+(shift3)$);
\coordinate(b1) at($1*(x)+-3*(y)+(shift3)$); 
\coordinate(b2) at($4*(x)+-13*(y)+(shift3)$); 
\coordinate(b3) at($3*(x)+-10*(y)+(shift3)$); 
%\coordinate(b4) at($8*(x)+-27*(y)+(shift3)$); 
%\coordinate(b5) at($5*(x)+-17*(y)+(shift3)$); 
    \draw (b0)--(b1); 
    \draw (b0)--(b2); 
    \draw (b0)--(b3); 
%    \draw (b0)--(b4); 
%    \draw (b0)--(b5); 
    \draw (b1)--(b2)--(b3);

\coordinate(c0) at($0*(x)+0*(y)+(shift3)$);
\coordinate(c1) at($1*(x)+-3*(y)+(shift3)$); 
\coordinate(c2) at($3*(x)+-8*(y)+(shift3)$); 
\coordinate(c3) at($2*(x)+-5*(y)+(shift3)$); 
\coordinate(c4) at($5*(x)+-12*(y)+(shift3)$); 
\coordinate(c5) at($3*(x)+-7*(y)+(shift3)$); 
    \draw (c0)--(c1); 
    \draw (c0)--(c2); 
    \draw (c0)--(c3); 
    \draw (c0)--(c4); 
    \draw (c0)--(c5); 
    \draw (c1)--(c2)--(c3)--(c4)--(c5);

\coordinate(d0) at($0*(x)+0*(y)+(shift3)$);
\coordinate(d1) at($1*(x)+-1*(y)+(shift3)$); 
\coordinate(d2) at($3*(x)+-4*(y)+(shift3)$); 
\coordinate(d3) at($2*(x)+-3*(y)+(shift3)$); 
\coordinate(d4) at($5*(x)+-8*(y)+(shift3)$); 
\coordinate(d5) at($3*(x)+-5*(y)+(shift3)$); 
\coordinate(d6) at($7*(x)+-12*(y)+(shift3)$); 
\coordinate(d7) at($4*(x)+-7*(y)+(shift3)$); 
%\coordinate(d8) at($9*(x)+-16*(y)+(shift3)$); 
%\coordinate(d9) at($5*(x)+-9*(y)+(shift3)$); 
    \draw (d0)--(d1); 
    \draw (d0)--(d2); 
    \draw (d0)--(d3); 
    \draw (d0)--(d4); 
    \draw (d0)--(d5); 
    \draw (d0)--(d6); 
    \draw (d0)--(d7); 
%    \draw (d0)--(d8); 
%    \draw (d0)--(d9); 
    \draw (d1)--(d2)--(d3)--(d4)--(d5)--(d6)--(d7); 

\coordinate(t1) at($2*(x)+-1*(y)+(shift3)$); 
%\draw[dotted] (v0)--($6*(t1)$);
\coordinate(t2) at($1*(x)+-2*(y)+(shift3)$); 
\draw[dotted] (v0)--($6*(t2)$);
\coordinate(t3) at($2*(x)+-7*(y)+(shift3)$); 
\draw[dotted] (v0)--($2*(t3)$);

\end{tikzpicture}
\]
Now we can again apply Theorem \ref{theorem:gluing2} to $\Lambda$ and 
$\rho(\Lambda_2)$, and we get the algebra 
$$\Lambda_3:=\Lambda \circ\rho(\Lambda_2).$$
Then we have 
$\df(\Sigma(\Lambda_3))=(\df_0,\df_1,\df_2,\df_\infty)$, where 
$$\df_i=\left\{\begin{array}{ll}
(14141\cdots)&\mbox{if $i=0$},\\
(\cdots14141714141\cdots)&\mbox{if $i=1,2$},\\
(\cdots1414144 )&\mbox{if $i=\infty$},
\end{array}\right.$$
and the complement of $\Sigma(\Lambda_3)$ is $\bigsqcup_{i=0}^2\R_{>0}u_i$, where $u_i=(2,-1)\left[\begin{smallmatrix}0&1\\ -1&4\end{smallmatrix}\right]^i$ for $0\leq i\leq 2$. 
In particular, 
$\Sigma(\Lambda_3)$ has three connected components. 
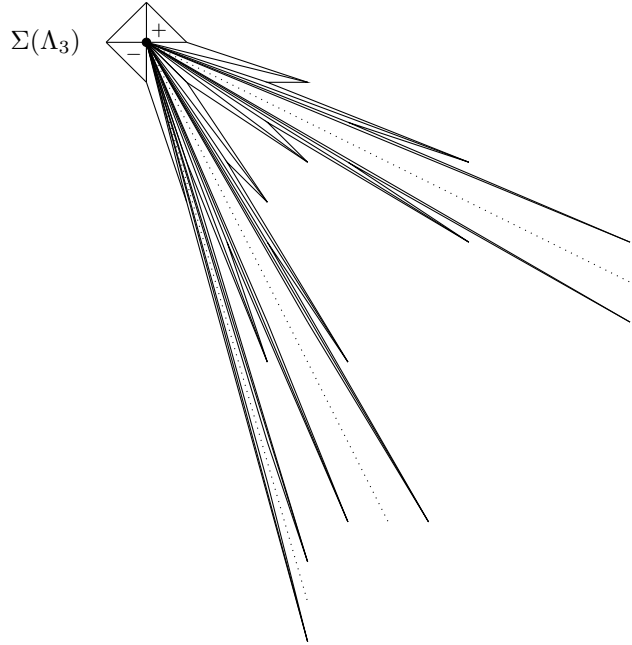
\[
\begin{tikzpicture}[baseline=0mm]
\coordinate(0) at(0:0);
\coordinate(x) at(0:0.53);
\coordinate(y) at(90:0.53);
\node at(0.165,0.165) {$\scriptstyle +$};
\node at(-0.165,-0.165) {$\scriptstyle -$};
%%[ lambda3 ]
\coordinate(shift3) at($0*(x)+0*(y)$); 
\node(lambda3) at(shift3) {}; 
\coordinate(L) at($-2.5*(x)+0*(y)+(shift3)$); %%%%optional
\node at(L) {$\Sigma(\Lambda_3)$};  %%%%optional

\coordinate(v0) at($0*(x)+0*(y)+(shift3)$); 
\draw[fill=black] (v0) circle [radius = 0.55mm]; 
\coordinate(v1) at($1*(x)+0*(y)+(shift3)$); 
\coordinate(v2) at($4*(x)+-1*(y)+(shift3)$); 
\coordinate(v3) at($3*(x)+-1*(y)+(shift3)$); 
\coordinate(v4) at($8*(x)+-3*(y)+(shift3)$); 
\coordinate(v5) at($5*(x)+-2*(y)+(shift3)$); 
\coordinate(v6) at($12*(x)+-5*(y)+(shift3)$); 
    \draw (v0)--(v1); 
    \draw (v0)--(v2); 
    \draw (v0)--(v3); 
    \draw (v0)--(v4);
    \draw (v0)--(v5); 
    \draw (v0)--(v6);
    \draw (v1)--(v2)--(v3)--(v4)--(v5)--(v6);
\coordinate(u0) at($0*(x)+0*(y)+(shift3)$); 
\coordinate(u1) at($0*(x)+-1*(y)+(shift3)$); 
\coordinate(u2) at($1*(x)+-1*(y)+(shift3)$); 
\coordinate(u3) at($4*(x)+-3*(y)+(shift3)$); 
\coordinate(u4) at($3*(x)+-2*(y)+(shift3)$); 
\coordinate(u5) at($8*(x)+-5*(y)+(shift3)$); 
\coordinate(u6) at($5*(x)+-3*(y)+(shift3)$); 
\coordinate(u7) at($12*(x)+-7*(y)+(shift3)$); 
    \draw (u0)--(u1); 
    \draw (u0)--(u2); 
    \draw (u0)--(u3); 
    \draw (u0)--(u4);
    \draw (u0)--(u5); 
    \draw (u0)--(u6);
    \draw (u0)--(u7); 
    \draw (u2)--(u3)--(u4)--(u5)--(u6)--(u7); 
\coordinate(w1) at($1*(x)+0*(y)+(shift3)$); 
\coordinate(w2) at($0*(x)+-1*(y)+(shift3)$); 
\coordinate(w3) at($-1*(x)+0*(y)+(shift3)$); 
\coordinate(w4) at($0*(x)+1*(y)+(shift3)$); 
    \draw (v0)--(w3); 
    \draw (v0)--(w4);
    \draw (w3)--(w4); 
    \draw (w1)--(w4);
    \draw (w2)--(w3); 

\coordinate(a0) at($0*(x)+0*(y)+(shift3)$); 
\coordinate(a1) at($0*(x)+-1*(y)+(shift3)$); 
\coordinate(a2) at($1*(x)+-4*(y)+(shift3)$); 
\coordinate(a3) at($4*(x)+-15*(y)+(shift3)$); 
\coordinate(a4) at($3*(x)+-11*(y)+(shift3)$); 
    \draw (v0)--(a1); 
    \draw (a0)--(a2); 
    \draw (a0)--(a3); 
    \draw (a0)--(a4); 
    \draw (a1)--(a2)--(a3)--(a4);

\coordinate(b0) at($0*(x)+0*(y)+(shift3)$);
\coordinate(b1) at($1*(x)+-3*(y)+(shift3)$); 
\coordinate(b2) at($4*(x)+-13*(y)+(shift3)$); 
\coordinate(b3) at($3*(x)+-10*(y)+(shift3)$); 
%\coordinate(b4) at($8*(x)+-27*(y)+(shift3)$); 
%\coordinate(b5) at($5*(x)+-17*(y)+(shift3)$); 
    \draw (b0)--(b1); 
    \draw (b0)--(b2); 
    \draw (b0)--(b3); 
%    \draw (b0)--(b4); 
%    \draw (b0)--(b5); 
    \draw (b1)--(b2)--(b3);

\coordinate(c0) at($0*(x)+0*(y)+(shift3)$);
\coordinate(c1) at($1*(x)+-3*(y)+(shift3)$); 
\coordinate(c2) at($3*(x)+-8*(y)+(shift3)$); 
\coordinate(c3) at($2*(x)+-5*(y)+(shift3)$); 
\coordinate(c4) at($5*(x)+-12*(y)+(shift3)$); 
\coordinate(c5) at($3*(x)+-7*(y)+(shift3)$); 
    \draw (c0)--(c1); 
    \draw (c0)--(c2); 
    \draw (c0)--(c3); 
    \draw (c0)--(c4); 
    \draw (c0)--(c5); 
    \draw (c1)--(c2)--(c3)--(c4)--(c5);

\coordinate(d0) at($0*(x)+0*(y)+(shift3)$);
\coordinate(d1) at($1*(x)+-1*(y)+(shift3)$); 
\coordinate(d2) at($3*(x)+-4*(y)+(shift3)$); 
\coordinate(d3) at($2*(x)+-3*(y)+(shift3)$); 
\coordinate(d4) at($5*(x)+-8*(y)+(shift3)$); 
\coordinate(d5) at($3*(x)+-5*(y)+(shift3)$); 
\coordinate(d6) at($7*(x)+-12*(y)+(shift3)$); 
\coordinate(d7) at($4*(x)+-7*(y)+(shift3)$); 
%\coordinate(d8) at($9*(x)+-16*(y)+(shift3)$); 
%\coordinate(d9) at($5*(x)+-9*(y)+(shift3)$); 
    \draw (d0)--(d1); 
    \draw (d0)--(d2); 
    \draw (d0)--(d3); 
    \draw (d0)--(d4); 
    \draw (d0)--(d5); 
    \draw (d0)--(d6); 
    \draw (d0)--(d7); 
%    \draw (d0)--(d8); 
%    \draw (d0)--(d9); 
    \draw (d1)--(d2)--(d3)--(d4)--(d5)--(d6)--(d7); 

\coordinate(t1) at($2*(x)+-1*(y)+(shift3)$); 
\draw[dotted] (v0)--($6*(t1)$);
\coordinate(t2) at($1*(x)+-2*(y)+(shift3)$); 
\draw[dotted] (v0)--($6*(t2)$);
\coordinate(t3) at($2*(x)+-7*(y)+(shift3)$); 
\draw[dotted] (v0)--($2*(t3)$);
\end{tikzpicture}
\]
By applying the same argument repeatedly, we get the desired assertions.
\end{proof}

\section*{Acknowledgments} 
T.A is supported by JSPS Grants-in-Aid for Scientific Research JP19J11408. A.H is supported by JSPS Grant-in-Aid for Scientists Research (C) 20K03513. O.I is supported by JSPS Grant-in-Aid for Scientific Research (B) 16H03923, (C) 18K3209 and (S) 15H05738. R.K is supported by JSPS Grant-in-Aid for Young Scientists (B) 17K14169. Y.M is supported by Grant-in-Aid for Scientific Research (C) 20K03539. 

The authors would like to thank Karin Baur for useful discussions on fans and friezes.

\end{document}